\documentclass[12pt,reqno]{amsart}

\usepackage[margin=1in]{geometry}  
\usepackage{amsmath}
\usepackage{amssymb}
\usepackage[pdftex]{graphicx}
\usepackage{verbatim}   
\usepackage{color}    
\usepackage{mathrsfs}  
\usepackage[colorlinks=true]{hyperref}
\hypersetup{urlcolor=blue, citecolor=blue, pdfstartview=FitH}
\hypersetup{pdfstartview={XYZ null null 1.25}}
\usepackage{fontenc}
\usepackage{textcomp}
\usepackage{fix-cm}
\usepackage{array}
\usepackage{enumerate}
\usepackage{latexsym, amsfonts, amssymb}
\usepackage{color}
\usepackage{txfonts}
\usepackage{comment}
\usepackage{subfigure}
\usepackage{accents}

\numberwithin{equation}{section}

\newtheorem{theorem}[equation]{Theorem}
\newtheorem{proposition}[equation]{Proposition}
\newtheorem{lemma}[equation]{Lemma}
\newtheorem{corollary}[equation]{Corollary}

\newtheorem{remark}[equation]{Remark}

\title[Generalized Two-component Hunter-Saxton System]
      {Global Estimates and Blow-up Criteria for the Generalized Hunter-Saxton System}

\author[Alejandro Sarria]{}


 \email{alejandro.sarria@colorado.edu}

\begin{document}
\maketitle

\centerline{\scshape Alejandro Sarria}
\medskip
{\footnotesize
 \centerline{Department of Mathematics}
   \centerline{University of Colorado, Boulder}
   \centerline{Boulder, CO 80309-0395, USA}
} 

\medskip


\bigskip

\centerline{(\small{\emph{Dedicated to Professor Ralph Saxton on the occasion of his $58th$ birthday)}}}


\begin{abstract}
The generalized, two-component Hunter-Saxton system comprises several well-known models of fluid dynamics and serves as a tool for the study of one-dimensional fluid convection and stretching. In this article a general representation formula for periodic solutions to the system, which is valid for arbitrary values of parameters $(\lambda,\kappa)\in\mathbb{R}\times\mathbb{R}$, is derived. This allows us to examine in great detail qualitative properties of blow-up as well as the asymptotic behaviour of solutions, including convergence to steady states in finite or infinite time.
\end{abstract}

\section{Introduction}\hfill
\label{sec:intro}

We are concerned with the study of regularity in solutions to a boundary value problem for a system of equations comprising several well-known models of fluid dynamics as well as modeling convection and stretching in one-dimensional fluid evolution equations,
\begin{equation}
\label{eq:hs}
\setlength{\jot}{12pt}
\begin{cases}
u_{xt}+uu_{xx}-\lambda u_x^2-\kappa\rho^2=I(t),\,\,\,\,\,\,\,\,\,\,\,\,\,&t>0,
\\[-3ex]
\\
\rho_t+u\rho_x=2\lambda\rho u_x,\,\,\,\,\,\,\,\,\,&t>0,
\\[-3ex]
\\
u(x,0)=u_0(x),\,\,\,\,\rho(x,0)=\rho_0(x),\,\,\,\,\,\,&x\in[0,1],
\end{cases}
\end{equation}
where $\lambda$ and $\kappa$ denote arbitrary parameters, the nonlocal term $I(t)$ is given by
\begin{equation}
\label{eq:nonlocal}
\begin{split}
I(t)=-\kappa\int_0^1{\rho^2\,dx}-(1+\lambda)\int_0^1{u_x^2\,dx},
\end{split}
\end{equation}
and solutions are subject to periodic boundary conditions
\begin{equation}
\label{eq:pbc}
\begin{split}
u(0,t)=u(1,t),\,\,\,\,\,\,\,u_x(0,t)=u_x(1,t),\,\,\,\,\,\,\,\rho(0,t)=\rho(1,t).
\end{split}
\end{equation}
System (\ref{eq:hs}) was first introduced in \cite{Wunsch5} as \emph{the generalized Hunter-Saxton system} due to its connection, via $(\lambda,\kappa)=(-1/2,\pm1/2)$, to the Hunter-Saxton (HS) system. Both models have been studied extensively in the literature, see for instance \cite{Wunsch4}-\cite{Wunsch6}, \cite{Lenells2}-\cite{Liu2}, \cite{Wu1}, \cite{Moon1}, and \cite{Moon0}. The HS system is a particular case of the Gurevich-Zybin system describing the formation of large scale structure in the universe, as well as harmonic wave generation in nonlinear optics (c.f. \cite{Pavlov1} and references therein). Moreover, it arises as the ``short-wave'' (or ``high-frequency'') limit, obtained via the change of variables $(x,t)\to(\epsilon x,\epsilon t)$ in (\ref{eq:hs})i) and then letting $\epsilon\to0$ in the resulting equation, of the two-component Camassa-Holm (CH) system (\cite{Aconstantin0}, \cite{Escher0}),
which, in turn, is derived from the Green-Naghdi equations (\cite{Johnson1}, \cite{Green1}); widely used in coastal oceanography and approximate the free-surface Euler equations. The two-component CH system has been the subject of extensive studies (\cite{Aconstantin0}, \cite{Escher0}, \cite{Guan1}-\cite{Guo2}, \cite{Lenells2}, \cite{Mohajer1}, \cite{Mustafa1}, \cite{Zhang1}). In this context, $u$ represents the leading approximation of the horizontal fluid velocity, whereas $\rho$ stands for the horizontal deviation of the rescaled free water surface from equilibrium. Further, for $\rho_0\equiv0$, the two-component CH system reduces to the well-known CH equation, a nonlinear dispersive wave equation that arises in the study of propagation of unidirectional irrotational waves over a flat bed, as well as water waves moving over an underlying shear flow ((\cite{Camassa1}, \cite{Dullin1}, \cite{Johnson1}, \cite{Aconstantin2}), \cite{Johnson2}, \cite{Misiolek1}). The CH equation is completely integrable, has an infinite number of conserved quantities, and its solitary wave solutions are solitons (\cite{Camassa1}). It also admits ``peakons'' and ``breaking wave'' solutions. The former are waves of the form $u(x,t)=ce^{-\left|x-ct\right|}$ which are smooth everywhere except at their peaks, whereas, the latter represent waves whose height, $u$, remains finite while its slope, $u_x$, blows up. System (\ref{eq:hs}) is also particularly of interest due to the potential exhibition of such non-linear phenomena in its solutions, phenomena, we should note, is not inherent to small-amplitude models but exist in the full governing water-wave equations. 

For $\rho\equiv0$, or $\rho=\sqrt{-1}\,u_x$,\, (\ref{eq:hs}) becomes the generalized, inviscid Proudman-Johnson (giPJ) equation (\cite{Proudman1}-\cite{Saxton1}, \cite{Childress}, \cite{Okamoto1}, \cite{Okamoto2}, \cite{Aconstantin1}, \cite{Weyl1}, \cite{Wunsch1}, \cite{Wunsch2}, \cite{Wunsch3}, c.f. \cite{Chen1}, \cite{Chen2} for viscous model) comprising, for $\lambda=-1$, the Burgers' equation of gas dynamics, differentiated twice in space. If $\lambda=\frac{1}{n-1}$, it models stagnation point-form solutions (\cite{Childress}, \cite{Saxton1}, \cite{Sarria3}) to the $n-$dimensional incompressible Euler equations 
\begin{equation}
\label{eq:EE}
\boldsymbol{u}_t+(\boldsymbol{u}\cdot\nabla)\boldsymbol{u}=-\nabla p,\,\,\,\,\,\,\,\,\,\,\,\,\,\,\,\,\,\,\,\,\,\nabla\cdot\boldsymbol u=0,
\end{equation}
where $\boldsymbol{u}$ denotes the $n-$dimensional fluid velocity and $p$ represents its pressure. Alternatively, for $\lambda=1$ and $n=2$ it appears as a reduced one-dimensional model for the three-dimensional inviscid primitive equations of large scale oceanic and atmospheric dynamics (\cite{Cao1}, c.f. \cite{Cao2}, \cite{Cao3} for viscous case). Lastly, when $\lambda=-\frac{1}{2}$, the giPJ equation becomes the HS equation, describing the orientation of waves in massive director field nematic liquid crystals (\cite{Hunter1}, \cite{Dafermos1}, \cite{Khesin2}, \cite{Beals1}, \cite{Tiglay1}, \cite{Yin1}). For periodic solutions, HS also has a deep geometric interpretation  as it describes geodesics on the group of orientation-preserving diffeomorphisms on the unit circle modulo the subgroup of rigid rotations with respect to a particular right-invariant metric (\cite{Khesin1}, \cite{Bressan1}, \cite{Tiglay1}, \cite{Lenells1}). It is known (\cite{Lenells2}) that the HS system admits an interpretation as the Euler equation on the superconformal algebra of contact vector fields on the $1|2-$dimensional supercircle. Particularly, this geometric perspective has been useful in the construction of global weak solutions to the HS system (\cite{Lenells2}, \cite{Lenells3}, \cite{Wunsch6}). 

From a more heuristic point of view, (\ref{eq:hs}) may serve as a tool to better understand the role that convection and stretching play in the regularity of solutions to one-dimensional fluid evolution equations; it has been argued that the convection term can sometimes cancel some of the nonlinear effects (\cite{Ohkitani1}, \cite{Hou1}, \cite{Okamoto3}). More particularly, differentiating (\ref{eq:hs})i) in space, and setting $\omega=-u_{xx}$, yields
\begin{equation}
\label{eq:diff}
\begin{cases}
&\omega_{t}+\underbrace{u\omega_{x}}_{convection}+(1-2\lambda)\underbrace{\omega u_{x}}_{stretching}+2\kappa\underbrace{\rho\rho_x}_{coupling}=0,
\\
&\rho_t+\underbrace{u\rho_x}_{convection}=2\lambda u_x\rho.
\end{cases}
\end{equation}
The nonlinear terms in equation (\ref{eq:diff})i) represent the competition in fluid convection between nonlinear steepening and amplification due to $(1-2\lambda)$-dimensional stretching and $2\kappa$-dimensional coupling (\cite{Holm1}, \cite{Wunsch5}). More particularly, the parameter $\lambda\in\mathbb{R}$ is related to the ratio of stretching to convection, while $\kappa\in\mathbb{R}$ denotes a real dimensionless constant measuring the impact of the coupling between $u$ and $\rho$.

Additional fluid models belonging to the family of equations (\ref{eq:hs}) include: For ($\lambda=-\kappa=\infty$), equation (\ref{eq:hs})ii) reduces, after the introduction of new variables, to the well-known Constantin-Lax-Majda equation (\cite{Const2}), a one-dimensional model for the three-dimensional vorticity equation for which finite-time blow-up solutions are known to exist. If $\lambda=-\kappa=1/2$, the inviscid von Karman-Batchelor flow (\cite{Chae1}, \cite{Hou1}), derived from the 3D incompressible Euler equations, is known to have periodic global strong solutions. Lastly, if we assume that, as long as it is defined, $\rho$ never vanishes on $[0,1]$, then 
\begin{equation}
\label{eq:hs2}
\setlength{\jot}{12pt}
\begin{cases}
u_{xt}+uu_{xx}-u_x^2+\rho^2=\int_0^1{\rho^2\,dx}-2\int_0^1{u_x^2\,dx},
\\[-3ex]
\\
\rho_t+u\rho_x=\frac{1}{2}\rho u_x,
\end{cases}
\end{equation}
which represents a slight variation of (\ref{eq:hs}), can be obtained from the 2D inviscid, incompressible Boussinesq equations 
\begin{equation}
\label{eq:Boussinesq}
\begin{cases}
&\boldsymbol{u}_t+(\boldsymbol{u}\cdot\nabla)\boldsymbol{u}=-\nabla p+\theta\stackrel{\rightarrow}{e_2},
\\
&\nabla\cdot\boldsymbol{u}=0,
\\
&\theta_t+\boldsymbol{u}\cdot\nabla\theta=0
\end{cases}
\end{equation}
by considering velocities $\boldsymbol{u}$ and scalar temperatures $\theta$ (or densities) of the form
$$\boldsymbol{u}(x,y,t)=(u(x,t),-yu_x(x,t)),\,\,\,\,\,\,\,\,\,\,\,\,\,\,\,\,\,\,\theta(x,y,t)=y\rho(x,t)^2$$
on an infinitely long two-dimensional channel $(x,y)\in[0,1]\times\mathbb{R}$. The Boussinesq equations model large-scale atmospheric and oceanic fluids (see for instance \cite{Gill1}, \cite{Pedlosky1}, \cite{Majda1}, \cite{Sarria5}). Above, $\theta$ denotes either the scalar temperature or density, and $\stackrel{\rightarrow}{e_2}$ is the standard unit vector in the vertical direction. 

Before giving an outline of the paper, we discuss some previous results.

\subsection{Previous Results}
\label{subsec:previous}
In this section, we review a few results on the regularity of solutions to system (\ref{eq:hs}) (\cite{Wunsch4}, \cite{Wunsch5}, \cite{Wu1}, \cite{Liu1}). For additional blow-up or global-in-time criteria the reader may refer to \cite{Moon0} and \cite{Moon1}. 

First of all, the local well-posedness of (\ref{eq:hs}) has already been established, see for instance \cite{Wunsch4}, \cite{Wunsch5}, or \cite{Wu1}. Then the following is known:
\begin{itemize}
\item\label{it:1} If $(\lambda,\kappa)=(-1/2,-1/2)$, suppose $\int_0^1{\rho_0^2\,dx}\leq\int_0^1{(u_0')^2dx}$ and $\min_{x\in[0,1]}u_0'<0$. Then $u_x$ diverges in finite time.
\item\label{it:2} For $(\lambda,\kappa)\in\{-1/2\}\times\mathbb{R}^-$, assume $u_0$ is odd with $u_0(0)<0$, while $\rho_0$ is even and $\rho_0(0)=0$. Moreover, suppose $\left\|u_0'\right\|^2_2+k\left\|\rho_0\right\|^2_2\geq0$. Then $u_x(0,t)$ diverges.
\item\label{it:3} If $(\lambda,\kappa)\in\{-1/2\}\times\mathbb{R}^+$, $(u,\rho)$ blows up if $\min_{x}u_x\to-\infty$ or $\left\|\rho_x\right\|_{\infty}\to+\infty$.
\item\label{it:4} Suppose $(\lambda,\kappa)\in\{-1/2\}\times\mathbb{R}^+$, and $\rho_0$ never vanishes. Then solutions are global. Similarly for $(\lambda,\kappa)\in\{0\}\times\mathbb{R}^+$.
\item\label{it:5} Suppose $(\lambda,\kappa)\in[-1/4,0)\times\mathbb{R}^{-}$ and 
\begin{equation*}
\min_{x\in[0,1]}u_0'(x)<-\sqrt{\frac{k}{\left|2\lambda\right|-2\epsilon}\left\|\rho_0\right\|_{p}^{^{-\frac{1}{2\lambda}}}}<0,\,\,\,\,\,\,\,\,\,\,\,\,\,\,\epsilon\in\left(0,\left|\lambda\right|\right)
\end{equation*}
for $p=-\frac{1}{2\lambda}$. Then $u_x$ will diverge.
\item\label{it:6} For $(\lambda,\kappa)=(-1/2,1/2)$, nontrivial $\rho_0$, and nonconstant $u_0$ assume there is $x_0\in[0,1]$ such that $\rho_0(x_0)=0.$ Then solutions diverge.
\end{itemize}

\section{Outline of the Paper}
\label{sec:outline}

The outline for the remainder of the paper is as follows. In \S\ref{sec:solution}, we derive representation formulae for general solutions to (\ref{eq:hs}). This is done using the method of characteristics to reformulate the system as a nonlinear second-order ODE, which we are then able to solve using the prescribed boundary conditions. In \S\ref{subsec:notation} we establish terminology and introduce useful preliminary results. Then we begin our study of regularity in Sections \S\ref{subsec:regularityak<0} and \S\ref{subsec:regularityak>0}, where we examine the case of parameter values $\lambda$ and $\kappa$ such that $\lambda\kappa<0$ and respectively $\lambda\kappa>0$. More particularly, Theorems \ref{thm:global1ak<0} and \ref{thm:blow1ak<0} in \S\ref{subsec:regularityak<0} establish criteria for the finite-time convergence of solutions to steady states and their finite-time blow-up, respectively. Then Theorems \ref{thm:ak>0lambdanegkappapos} and \ref{thm:singular} in \S\ref{subsec:regularityak>0} examine finite-time blow-up and global existence in time for $(\lambda,\kappa)\in\mathbb{R}^-\times\mathbb{R}^-$ and respectively $(\lambda,\kappa)\in\mathbb{R}^+\times\mathbb{R}^+$, with the latter case leading to the ``most singular'' solutions. Lastly, specifics examples are provided in \S\ref{sec:examples}, while trivial or simpler cases are deferred to Appendix \ref{sec:simplecases} or the Corollaries \ref{coro:special1} and  \ref{coro:singular2} in Appendix \ref{sec:doublemult}.

\section{General Solution Along Characteristics}\hfill
\label{sec:solution}

In this section we derive new solution formulae for $u_x$ and $\rho$, along characteristics, for arbitrary $(\lambda,\kappa)\in\mathbb{R}\times\mathbb{R}$. 

Fix $\alpha\in[0,1]$ and define, for as long as $u$ exists, characteristics $\gamma$ via the initial value problem
\begin{equation}
\label{eq:charac}
\begin{split}
\dot\gamma(\alpha,t)=u(\gamma(\alpha,t),t),\,\,\,\,\,\,\,\,\,\,\,\,\,\,\gamma(\alpha,0)=\alpha,
\end{split}
\end{equation}
where ${}^{\cdot}=\frac{d}{dt}$. Then equation (\ref{eq:hs})ii) may be written as
$$\left(\ln\rho(\gamma(\alpha,t),t)\right)^{\cdot}=2\lambda u_x(\gamma(\alpha,t),t),$$
which yields
\begin{equation}
\label{eq:one}
\begin{split}
\rho(\gamma(\alpha,t),t)=\rho_0(\alpha)\cdot e^{2\lambda\int_0^t{u_x(\gamma(\alpha,s),s)\,ds}}.
\end{split}
\end{equation}
But (\ref{eq:charac}) implies that
\begin{equation}
\label{eq:two}
\begin{split}
\dot\gamma_{\alpha}=u_x(\gamma(\alpha,t),t)\cdot\gamma_{\alpha},\,\,\,\,\,\,\,\,\,\,\,\,\,\,\,\gamma_{\alpha}(\alpha,0)=1,
\end{split}
\end{equation}
so that the `jacobian', $\gamma_{\alpha}$, satisfies
\begin{equation}
\label{eq:three}
\begin{split}
\gamma_{\alpha}=e^{\int_0^t{u_x(\gamma(\alpha,s),s)\,ds}}.
\end{split}
\end{equation}
From (\ref{eq:one}) and (\ref{eq:three}), we conclude that
\begin{equation}
\label{eq:rho0}
\begin{split}
\rho(\gamma(\alpha,t),t)=\rho_0(\alpha)\cdot\gamma_{\alpha}(\alpha,t)^{2\lambda}.
\end{split}
\end{equation}
Now, using (\ref{eq:two}) and (\ref{eq:rho0}) we may write (\ref{eq:hs})i), along $\gamma$, as
\begin{equation}
\label{eq:four}
\begin{split}
\frac{d}{dt}(u_x(\gamma(\alpha,t),t))=\lambda\left(\dot\gamma_{\alpha}\cdot\gamma_{\alpha}^{-1}\right)^2+\kappa\left(\rho_0\cdot\gamma_{\alpha}^{2\lambda}\right)^2+I(t).
\end{split}
\end{equation}
Then differentiating (\ref{eq:two}) in time and subsequently using (\ref{eq:two}) and (\ref{eq:four}), gives
\begin{equation*}
\begin{split}
\ddot\gamma_{\alpha}=\left(\dot\gamma_{\alpha}\cdot\gamma_{\alpha}^{-1}\right)^2\cdot\gamma_{\alpha}+\left(\lambda\left(\dot\gamma_{\alpha}\cdot\gamma_{\alpha}^{-1}\right)^2+\kappa\left(\rho_0\cdot\gamma_{\alpha}^{2\lambda}\right)^2+I(t)\right)\cdot\gamma_{\alpha},
\end{split}
\end{equation*}
which may be rearranged as
\begin{equation}
\label{eq:six}
\begin{split}
I(t)+\kappa\left(\rho_0\cdot\gamma_{\alpha}^{2\lambda}\right)^2=\frac{\ddot\gamma_{\alpha}\cdot\gamma_{\alpha}-\left(1+\lambda\right)\dot\gamma_{\alpha}^2}{\gamma_{\alpha}^2},
\end{split}
\end{equation}
or alternatively for $\lambda\neq0$,\footnote[1]{Refer to the Appendix for the case $\lambda=0$.}
\begin{equation}
\label{eq:seven}
\begin{split}
I(t)+\kappa\left(\rho_0\cdot\gamma_{\alpha}^{2\lambda}\right)^2=-\frac{1}{\lambda}\,\gamma_{\alpha}^{\lambda}\cdot\left(\gamma_{\alpha}^{-\lambda}\right)^{\ddot{}}.
\end{split}
\end{equation}
Letting 
\begin{equation}
\label{eq:omega}
\begin{split}
\omega(\alpha,t)=\gamma_{\alpha}(\alpha,t)^{-\lambda}
\end{split}
\end{equation}
in (\ref{eq:seven}) and setting
\begin{equation}
\label{eq:eight}
\begin{split}
f(t)=\lambda I(t),\,\,\,\,\,\,\,\,\,\,\,\beta(\alpha)=-\lambda\kappa\rho_0(\alpha)^2,
\end{split}
\end{equation}
then yields
\begin{equation}
\label{eq:nonlinear}
\begin{split}
\ddot\omega(\alpha,t)+f(t)\omega(\alpha,t)=\beta(\alpha)\omega(\alpha,t)^{-3},
\end{split}
\end{equation}
a second-order nonlinear ODE parametrized by $\alpha$ and whose solution, according to (\ref{eq:two}) and (\ref{eq:omega}), satisfies the initial values 
\begin{equation}
\label{eq:ivomega}
\begin{split}
\omega(\alpha,0)=1,\,\,\,\,\,\,\,\,\,\,\,\,\,\,\,\,\,\,\dot\omega(\alpha,0)=-\lambda u_0'(\alpha).
\end{split}
\end{equation}
Equation (\ref{eq:nonlinear}) is known as the Ermakov-Pinney equation (\cite{Ermakov1}, \cite{Pinney1}) and it appears in several important physical contexts including quantum cosmology, quantum field theory, nonlinear elasticity, and nonlinear optics (see for instance \cite{Kevre1}). 

Our strategy for solving (\ref{eq:nonlinear})-(\ref{eq:ivomega}) is to first consider the associated linear, homogeneous problem
\begin{equation}
\label{eq:linear}
\begin{split}
\ddot y(\alpha,t)+f(t)y(\alpha,t)=0.
\end{split}
\end{equation}
Suppose we have two linearly independent solutions, $\phi_1(t)$ and $\phi_2(t)$, to (\ref{eq:linear}) satisfying $\phi_1(0)=\dot\phi_2(0)=1$ and $\dot\phi_1(0)=\phi_2(0)=0$. Then by Abel's formula, $W(\phi_1,\phi_2)\equiv1$, $t\geq0$, for $W(u_1,u_2)$ the wronskian of $u_1$ and $u_2$. We look for solutions to (\ref{eq:linear}), satisfying appropriate initial data, of the form
\begin{equation}
\label{eq:linearsol1}
\begin{split}
y(\alpha,t)=c_1(\alpha)\phi_1(t)+c_2(\alpha)\phi_2(t)
\end{split}
\end{equation}
where, by reduction of order,
\begin{equation}
\label{eq:phi2}
\begin{split}
\phi_2(t)=\phi_1(t)\eta(t),\,\,\,\,\,\,\,\,\,\,\,\,\,\,\eta(t)=\int_0^t{\frac{ds}{\phi_1(s)^2}}.
\end{split}
\end{equation}
The above reduces (\ref{eq:linearsol1}) to
\begin{equation}
\label{eq:linearsol2}
\begin{split}
y(\alpha,t)=\phi_1(t)(c_1(\alpha)+c_2(\alpha)\eta(t)).
\end{split}
\end{equation}
Now let
\begin{equation}
\label{eq:omega2}
\begin{split}
\omega(\alpha,t)=z(\eta(t))y(\alpha,t)
\end{split}
\end{equation}
for some function $z(\cdot)$ to be determined. Note that (\ref{eq:ivomega}), (\ref{eq:linearsol2}) and (\ref{eq:omega2}), along with the initial values for $\phi_1$ and $\eta$, imply that
\begin{equation}
\label{eq:constants1}
\begin{split}
c_1(\alpha)=\frac{1}{z(0)},\,\,\,\,\,\,\,\,\,\,\,\,\,\,\,\,\,c_2(\alpha)=-\frac{z'(0)+\lambda z(0)u_0'(\alpha)}{z(0)^2}
\end{split}
\end{equation}
for $z(0)\neq0$ and $'=\frac{d}{d\eta}$. For simplicity we set $z(0)=1$ and $z'(0)=0$; this yields
\begin{equation}
\label{eq:constants2}
\begin{split}
c_1(\alpha)=1,\,\,\,\,\,\,\,\,\,\,\,\,\,c_2(\alpha)=-\lambda u_0'(\alpha),
\end{split}
\end{equation}
so that (\ref{eq:linearsol2}) may be written as
\begin{equation}
\label{eq:y}
\begin{split}
y(\alpha,t)=\phi_1(t)\mathcal{J}(\alpha,t)
\end{split}
\end{equation}
for
\begin{equation}
\label{eq:j}
\begin{split}
\mathcal{J}(\alpha,t)=1-\lambda\eta(t)u_0'(\alpha),\quad\,\,\,\,\,\,\,\,\,\,\,\,\,\,\,\,\mathcal{J}(\alpha,0)=1.
\end{split}
\end{equation}
Next, plugging (\ref{eq:omega2}) into (\ref{eq:nonlinear}) and recalling that $y$ in (\ref{eq:y}) satisfies (\ref{eq:linear}), we obtain, after simplification,
\begin{equation}
\label{eq:ten}
\begin{split}
\mathcal{J}(\alpha,t)z''(\eta)-2\lambda u_0'(\alpha)z'(\eta)=\beta(\alpha)(z(\eta)\mathcal{J}(\alpha,t))^{-3}
\end{split}
\end{equation}
which, for 
\begin{equation}
\label{eq:mu}
\begin{split}
\mu(\alpha,\eta(t))=z(\eta(t))\mathcal{J}(\alpha,t),
\end{split}
\end{equation}
reduces to
\begin{equation}
\label{eq:eleven}
\begin{split}
\mu_{\eta\eta}=\beta(\alpha)\mu^{-3}
\end{split}
\end{equation}
complemented by the initial values
\begin{equation}
\label{eq:muvalues}
\begin{split}
\mu(\alpha,0)=1,\,\,\,\,\,\,\,\,\,\,\,\,\,\,\,\,\,\,\,\,\,\,\mu_{\eta}(\alpha,0)=-\lambda u_0'(\alpha).
\end{split}
\end{equation}
Rewriting (\ref{eq:eleven}) as a first-order equation leads to
\begin{equation}
\label{eq:h0}
\begin{split}
\frac{\partial}{\partial\mu}h(\alpha,\mu)^2=2\beta(\alpha)\mu^{-3},\,\,\,\,\,\,\,\,\,\,\,\,\,h(\alpha,\mu)=\mu_{\eta},
\end{split}
\end{equation}
which we integrate to find
\begin{equation}
\label{eq:hlast}
\begin{split}
\mu_{\eta}(\alpha,\eta)^2=C(\alpha)-\beta(\alpha)\mu(\alpha,\eta)^{-2}
\end{split}
\end{equation}
for
\begin{equation}
\label{eq:c0}
\begin{split}
C(\alpha)=\lambda\left(\lambda u_0'(\alpha)^2-\kappa\rho_0(\alpha)^2\right),
\end{split}
\end{equation}
by (\ref{eq:eight})ii) and (\ref{eq:muvalues}). Now, (\ref{eq:hlast}) gives
\begin{equation}
\label{eq:twelve}
\begin{split}
\mu_{\eta}(\alpha,\eta)=\pm\frac{\sqrt{C(\alpha)\mu^2-\beta(\alpha)}}{\mu(\alpha,\eta)}.
\end{split}
\end{equation}
Solving the above separable equation yields
\begin{equation*}
\begin{split}
\pm C(\alpha)\eta(t)=\sqrt{C(\alpha)z(\eta(t))^2\mathcal{J}(\alpha,t)^2-\beta(\alpha)}-\sqrt{C(\alpha)-\beta(\alpha)}\,,
\end{split}
\end{equation*}
which we solve for $z^2$ to obtain
\begin{equation}
\label{eq:z0}
\begin{split}
z(\eta(t))^2=\frac{C(\alpha)\eta(t)^2\pm\left|2\lambda u_0'(\alpha)\right|\eta(t)+1}{\mathcal{J}(\alpha,t)^2}
\end{split}
\end{equation}
for fixed $\alpha\in[0,1]$ and $\lambda\neq0$. But setting $\eta=0$ in (\ref{eq:twelve}) and using (\ref{eq:muvalues}) and
$$\sqrt{C(\alpha)-\beta(\alpha)}=\left|\lambda u_0'(\alpha)\right|,$$
implies that 
\begin{equation}
\label{eq:z}
\begin{split}
z(\eta(t))^2=\frac{\mathcal{Q}(\alpha,t)}{\mathcal{J}(\alpha,t)^2}
\end{split}
\end{equation}
for 
\begin{equation}
\label{eq:Q}
\begin{split}
\mathcal{Q}(\alpha,t)=C(\alpha)\eta(t)^2-2\lambda u_0'(\alpha)\eta(t)+1,
\end{split}
\end{equation}
which we use on (\ref{eq:omega}), (\ref{eq:omega2}) and (\ref{eq:y}), to obtain
\begin{equation}
\label{eq:a0}
\begin{split}
\gamma_{\alpha}(\alpha,t)=\left[\phi_1(t)^{2}\mathcal{Q}(\alpha,t)\right]^{-\frac{1}{2\lambda}}.
\end{split}
\end{equation}
To determine $\phi_1$ above, we note that uniqueness of solution to (\ref{eq:charac}) and periodicity of $u$ requires that 
$$\gamma(\alpha+1,t)=1+\gamma(\alpha,t)$$
for all $\alpha\in[0,1]$ and as long as $u$ is defined. Particularly, this implies that the jacobian has mean one in $[0,1]$, namely
\begin{equation}
\label{eq:meanone}
\begin{split}
\int_0^1{\gamma_{\alpha}(\alpha,t)\,d\alpha}\equiv1.
\end{split}
\end{equation}
For $i=0,1,...,n$, set
\begin{equation}
\label{eq:p}
\begin{split}
\mathcal{P}_i(\alpha,t)=\mathcal{Q}(\alpha,t)^{-i-\frac{1}{2\lambda}}
\end{split}
\end{equation}
and 
\begin{equation}
\label{eq:meanp}
\begin{split}
\bar{\mathcal{P}}_i(t)=\int_0^1{\mathcal{P}_i(\alpha,t)\,d\alpha},\,\,\,\,\,\,\,\,\,\,\,\,\,\,\,\,\bar{\mathcal{P}}_i(0)=1.
\end{split}
\end{equation}
Integrating (\ref{eq:a0}) in $\alpha$ and using (\ref{eq:meanone}) gives
\begin{equation}
\label{eq:phi1}
\begin{split}
\phi_1(t)=\bar{\mathcal{P}}_0(t)^{\lambda},
\end{split}
\end{equation}
which we substitute back into (\ref{eq:a0}) to find
\begin{equation}
\label{eq:jacobian}
\begin{split}
\gamma_{\alpha}=\mathcal{P}_0/\bar{\mathcal{P}}_0.
\end{split}
\end{equation}
Using the above on (\ref{eq:two})i) and (\ref{eq:rho0}), yields
\begin{equation}
\label{eq:ux0}
\begin{split}
u_x(\gamma(\alpha,t),t)=\frac{\mathcal{\dot P}_0}{\mathcal{P}_0}-\frac{\mathcal{\dot{\bar P}}_0}{\mathcal{\bar P}_0}
\end{split}
\end{equation}
and
\begin{equation}
\label{eq:rho00}
\begin{split}
\rho(\gamma(\alpha,t),t)=\rho_0(\alpha)\cdot\left(\mathcal{P}_0/\bar{\mathcal{P}}_0\right)^{2\lambda}
\end{split}
\end{equation}
with the strictly increasing function $\eta(t)$ satisfying the IVP
\begin{equation}
\label{eq:etaivp}
\begin{split}
\dot\eta(t)=\bar{\mathcal{P}}_0(t)^{-2\lambda},\,\,\,\,\,\,\,\,\,\,\,\,\,\,\,\,\eta(0)=0.
\end{split}
\end{equation}
In view of the above, we have now obtained a complete description of the two-component solution to (\ref{eq:hs}) valid for parameters $(\lambda,\kappa)\in\mathbb{R}\backslash\{0\}\times\mathbb{R}$ and in terms of the initial data. Explicitly, and after simplification, (\ref{eq:ux0}) may be written as
\begin{equation}
\label{eq:ux}
\begin{split}
u_x(\gamma(\alpha,t),t)=\frac{\mathcal{\bar P}_0(t)^{^{-2\lambda}}}{\lambda}\left\{\frac{\lambda u_0'(\alpha)-\eta(t)C(\alpha)}{\mathcal{Q}(\alpha,t)}-\frac{1}{\mathcal{\bar P}_0(t)}\int_0^1{\frac{\lambda u_0'(\alpha)-\eta(t) C(\alpha)}{\mathcal{Q}(\alpha,t)^{1+\frac{1}{2\lambda}}}d\alpha}\right\},
\end{split}
\end{equation}
while (\ref{eq:rho00}) becomes
\begin{equation}
\label{eq:rho}
\begin{split}
\rho(\gamma(\alpha,t),t)=\frac{\rho_0(\alpha)}{\mathcal{Q}(\alpha,t)}\left(\int_0^1{\frac{d\alpha}{\mathcal{Q}(\alpha,t)^{\frac{1}{2\lambda}}}}\right)^{-2\lambda}.
\end{split}
\end{equation}
Lastly, from (\ref{eq:etaivp}) define $0<t_*\leq+\infty$ as
\begin{equation}
\label{eq:t*}
\begin{split}
t_*\equiv\lim_{\eta\uparrow\eta_*}t(\eta)=\lim_{\eta\uparrow\eta_*}\int_0^{\eta}{\left(\int_0^1{\frac{d\alpha}{(C(\alpha)\sigma^2-2\lambda u_0'(\alpha)\sigma+1)^{\frac{1}{2\lambda}}}}\right)^{2\lambda}d\sigma}
\end{split}
\end{equation}
for some $\eta_*\in\mathbb{R}^+$ to be defined. We will get back to formula (\ref{eq:t*}) once regularity is examined.

\begin{remark}
\label{rem:dirirem}
The above formulae is also valid for $u$ satisfying Dirichlet boundary conditions
$$u(0,t)=u(1,t)=0,$$
which implies that, for as long as $u$ is defined, those characteristics that originate at the boundary stay at the boundary,
$$\gamma(0,t)\equiv0,\quad\,\,\,\,\,\,\,\,\,\,\,\,\,\,\gamma(1,t)\equiv1.$$
As a result, the jacobian still has mean one in $[0,1]$ and we may proceed as in the periodic case.
\end{remark}

\begin{remark}
Integrating the jacobian (\ref{eq:jacobian}) in $\alpha$ yields the trajectories
\begin{equation}
\label{eq:int1}
\begin{split}
\gamma(\alpha,t)=\gamma(0,t)+\frac{1}{\bar{\mathcal{P}}_0(t)}\int_0^{\alpha}{\mathcal{P}_0(y,t)\,dy},
\end{split}
\end{equation}
where $\gamma(0,t)=u(\gamma(0,t),t)$. For Dirichlet boundary conditions, $\gamma(0,t)\equiv0$, and so $u\circ\gamma$ is obtained from $\dot\gamma=u\circ\gamma$, namely
\begin{equation*}
\begin{split}
u(\gamma(\alpha,t),t)=\left(\frac{1}{\bar{\mathcal{P}}_0(t)}\int_0^{\alpha}{\mathcal{P}_0(y,t)\,dy}\right)^{\cdot{}}.
\end{split}
\end{equation*}
However, for periodic solutions $\gamma(0,t)\equiv0$ is generally not true, and so an extra condition is needed to determine $\gamma(0,t)$. Consequently, to have a completely determined description of the problem, we will assume for the remainder of the paper that the first component solution, $u$, has mean zero in $[0,1]$,
\begin{equation}
\label{eq:mean0}
\begin{split}
\int_0^1{u(x,t)\,dx}\equiv0.
\end{split}
\end{equation}
It is known in the case of the giPJ equation (see for instance \cite{Sarria1} or \cite{Aconstantin1}), that (\ref{eq:mean0}) arises naturally for initial data satisfying certain symmetries that are preserved by the PDE. Moreover, in the context of the Euler equations, (\ref{eq:mean0}) holds as long as the pressure is periodic in one of its coordinate variables (\cite{Saxton1}).

\end{remark}

\section{Blow-up, Global Estimates, and Convergence to Steady States}\hfill
\label{sec:regularity}

In this section we study the evolution in time of (\ref{eq:ux}) and (\ref{eq:rho}) for parameters $(\lambda,\kappa)\in\mathbb{R}\backslash\{0\}\times\mathbb{R}\backslash\{0\}$, particularly, their finite-time blow-up, persistence for all time, and convergence to steady states in finite or infinite time. Most of the regularity results established here, as well as in the Appendix, apply to initial data ($u_0,\rho_0$) that is either smooth or belongs to a particular class of smooth functions; however, from the solution formulae derived in \S\ref{sec:solution}, these results can actually be extended to larger classes of nonsmooth functions, specifically, bounded functions $u_0'(x)$ and $\rho_0(x)$ which are, at least, $C^0[0,1]\, a.e.$ 

First, in \S\ref{subsec:notation} we introduce some new terminology and establish useful preliminary results. Then \S\ref{subsec:regularityak<0} examines the case of parameter values $\lambda\kappa<0$, whereas, those satisfying $\lambda\kappa>0$ are deferred to \S\ref{subsec:regularityak>0}. More particularly, Theorems \ref{thm:global1ak<0} and \ref{thm:blow1ak<0} in \S\ref{subsec:regularityak<0} establish criteria for the convergence, in finite time, of solutions to steady states as well as their finite-time blow-up, respectively. Then Theorems \ref{thm:ak>0lambdanegkappapos} and \ref{thm:singular} in \S\ref{subsec:regularityak>0} examine finite-time blow-up and global existence in time for $(\lambda,\kappa)\in\mathbb{R}^-\times\mathbb{R}^-$ and respectively $(\lambda,\kappa)\in\mathbb{R}^+\times\mathbb{R}^+$, with the latter case leading to the ``most singular'' solutions. Lastly, we note that in \S\ref{subsec:regularityak>0} we treat the case where (\ref{eq:Q}) vanishes earliest at some $\eta-$value of single multiplicity; the case of a double root is deferred to Corollaries \ref{coro:special1} and  \ref{coro:singular2} in Appendix \ref{sec:doublemult}.

\subsection{Notation and Preliminary Results}\hfill
\label{subsec:notation}

For $C(\alpha)$ as in (\ref{eq:c0}), let
\begin{equation}
\label{eq:omegaset}
\begin{split}
\Omega\equiv\{\alpha\in[0,1]\,|\,C(\alpha)=0\}
\end{split}
\end{equation}
and 
\begin{equation}
\label{eq:sigma}
\begin{split}
\Sigma\equiv\{\alpha\in[0,1]\,|\,C(\alpha)\neq0\},
\end{split}
\end{equation}
so that $[0,1]=\Omega\cup\Sigma$. Although we allow for $\Omega=\emptyset$, we will assume that $C(\alpha)$ is not identically zero, namely, $\Omega\neq[0,1]$. Moreover, and mostly for simplicity, we suppose that for $\Omega\neq\emptyset$, $C(\alpha)$ vanishes at a finite number of locations\footnote[2]{The reader may refer to Appendix \ref{sec:simplecases} for the cases $\Omega=[0,1]$, as well as $\lambda=0$ and/or $\kappa=0$. For now, we note that these special cases lead to regularity results that have already been established (see, e.g., \cite{Sarria1} or \cite{Sarria2}), or to solution formulas which simplify greatly, in some cases leading to trivial solutions.}. Lastly, we limit our analysis to the case $\rho_0(x)\not\equiv0$, which by (\ref{eq:rho}) implies $\rho(x,t)\not\equiv0$. For The case of $\rho\equiv0$ see, for instance, \cite{Sarria1} and \cite{Sarria2}.

From the formulae (\ref{eq:ux}) and (\ref{eq:rho}), we see that the issue to consider when studying the evolution in time of these quantities is not just a possible vanishing of the quadratic $\mathcal{Q}$ in (\ref{eq:Q}), but also what are the locations in $[0,1]$ that yield the least, positive $\eta-$value for which vanishing occurs. With this in mind, note that whenever $\alpha\in\Omega$, $\mathcal{Q}$ reduces to a linear function of $\eta$. Instead, if $\alpha\in\Sigma$, then $C(\alpha)\neq0$ and we may factor the quadratic into terms whose zeroes have either single or double multiplicity. More particularly, since the discriminant of (\ref{eq:Q}) is given by
\begin{equation}
\label{eq:D}
\begin{split}
\mathcal{D}(\alpha)=4\lambda\kappa\rho_0(\alpha)^2,
\end{split}
\end{equation}
$\mathcal{Q}$ admits three representations. The first is
\begin{equation}
\label{eq:fac1}
\begin{split}
\mathcal{Q}(\alpha,t)=(1-\eta(t)g_1(\alpha))(1-\eta(t)g_2(\alpha))
\end{split}
\end{equation}
for
\begin{equation}
\label{eq:g}
g_{1}(\alpha)=\lambda u_0'(\alpha)+\sqrt{\lambda\kappa}\left|\rho_0(\alpha)\right|,\,\,\,\,\,\,\,\,\,\,\,\,\,\,\,g_{2}(\alpha)=\lambda u_0'(\alpha)-\sqrt{\lambda\kappa}\left|\rho_0(\alpha)\right|
\end{equation}
and valid whenever $\alpha\in\Sigma$ is such that $\rho_0(\alpha)\neq0$. The second is given by
\begin{equation}
\label{eq:fac2}
\begin{split}
\mathcal{Q}(\alpha,t)=(1-\lambda\eta(t)u_0'(\alpha))^2
\end{split}
\end{equation}
for $\alpha\in\Sigma$ and $\rho_0(\alpha)=0$, while the third representation for $\mathcal{Q}$, when $\alpha\in\Omega$, is
\begin{equation}
\label{eq:fac3}
\begin{split}
\mathcal{Q}(\alpha,t)=1-2\lambda u_0'(\alpha)\eta(t).
\end{split}
\end{equation}
Now, for $0<t_*\leq+\infty$ as in (\ref{eq:t*}), let 
\begin{equation}
\label{eq:eta*}
\begin{split}
0<\eta_*<+\infty
\end{split}
\end{equation}
denote the $\eta-$value
\begin{equation}
\label{eq:etatoeta*}
\begin{split}
\lim_{t\uparrow t_*}\eta(t)=\eta_*
\end{split}
\end{equation}
and
\begin{equation}
\label{eq:overlinealpha}
\begin{split}
\overline\alpha\in[0,1]
\end{split}
\end{equation}
a finite number of points in the unit interval such that
\begin{equation}
\label{eq:earliest}
\begin{split}
\lim_{t\uparrow t_*}\mathcal{Q}(\alpha,t)=0.
\end{split}
\end{equation}
Below we remark on the behaviour of $\mathcal{Q}$ relative to a possible, earliest root $\eta_*$ and the location(s) $\overline\alpha$ which may lead to $\eta_*$. 

\vspace{0.05in}
\textbf{Case $\lambda\kappa<0$}
\vspace{0.05in}

Clearly, if $\lambda\kappa<0$ and $\rho_0$ never vanishes, (\ref{eq:D}) implies that $0<\mathcal{Q}<+\infty$ for all $\alpha\in[0,1]$ and $\eta\in\mathbb{R}^+$ (recall that $\mathcal{Q}(\alpha,0)\equiv1$). However, if $\lambda\kappa<0$ and $\rho_0$ is zero somewhere in $[0,1]$, then (\ref{eq:fac2}) holds at those locations and $\mathcal{Q}$ may now have roots of multiplicity two. We will examine both cases in detail in \S\ref{subsec:regularityak<0}.

\vspace{0.05in}
\textbf{Case $\lambda\kappa>0$}
\vspace{0.05in}

Next suppose $\lambda\kappa>0$. Then the discriminant (\ref{eq:D}) satisfies $\mathcal{D}(\alpha)\geq0$ and $\mathcal{Q}$ now admits roots $\eta_*$ of either single or double multiplicity. \emph{Relative to our choice of $\alpha$ and initial data, the above representations (\ref{eq:fac1})-(\ref{eq:fac3}) for $\mathcal{Q}$ will play an important role in our study of regularity because the $\lambda-$values for which the integral terms in (\ref{eq:ux}) and (\ref{eq:rho}) either converge or diverge as $\eta$ approaches $\eta_*$ may, in turn, depend on the multiplicity of $\eta_*$}. Therefore, and as we will see in later estimates, it will only be necessary to consider representations (\ref{eq:fac1}) (or (\ref{eq:fac3})) and (\ref{eq:fac2}), the single and respectively double multiplicity cases. However, the case of $\eta_*$ a double root, as it turns out, has already been studied. In fact, the simplest instance in which this occurs is when $\rho_0(\alpha)\equiv0$, so that (\ref{eq:Q}) reduces to (\ref{eq:fac2}) for all $\alpha\in[0,1]$. As discussed in \S\ref{sec:intro}, for $\rho_0\equiv0$, (\ref{eq:hs}) becomes the giPJ equation, studied extensively in \cite{Sarria1}, \cite{Sarria2}, and the references therein. Moreover, if $\rho_0\not\equiv0$ but $u_0$ and $\rho_0$ are such that $\mathcal{Q}$ takes the form (\ref{eq:fac2}) (i.e. $\Sigma\neq\emptyset$),\footnote[3]{See Appendix \ref{sec:doublemult} for remarks on how plausible this case actually is.} then $\eta_*=\frac{1}{\lambda u_0'(\overline\alpha)}$ where $u_0'(\overline\alpha)$ represents the negative minimum or positive maximum of $u_0'(\alpha)$ over $\Sigma$ when $\lambda<0$ or respectively $\lambda>0$, assuming each exists in the corresponding case. If so, it follows that solutions still retain the giPJ equation behaviour from the $\rho_0\equiv0$ case. Indeed, the latter assumptions imply that the space-dependent term in (\ref{eq:ux}) will diverge earliest at $\alpha=\overline\alpha$ as $\eta\uparrow\eta_*$,
\begin{equation}
\label{eq:spacedouble}
\begin{split}
\frac{\lambda u_0'(\overline\alpha)-\eta(t)C(\overline\alpha)}{\mathcal{Q}(\overline\alpha,t)}=\frac{1-\lambda\eta(t) u_0'(\overline\alpha)}{(1-\lambda\eta(t)u_0'(\overline\alpha))^2}=\frac{\lambda u_0'(\overline\alpha)}{1-\lambda\eta(t)u_0'(\overline\alpha)}\to+\infty.
\end{split}
\end{equation}
Similarly, for both $r>0$ and $\eta_*-\eta>0$ small, the integral terms satisfy
\begin{equation}
\label{eq:podouble}
\begin{split}
\mathcal{\bar{P}}_0(t)\sim\int_{\overline\alpha-r}^{\overline\alpha+r}{\frac{d\alpha}{(1-\lambda\eta(t)u_0'(\alpha))^{\frac{1}{\lambda}}}},
\end{split}
\end{equation}
while
\begin{equation}
\label{eq:otherdouble}
\begin{split}
\int_{0}^{1}{\frac{\lambda u_0'(\alpha)-\eta(t)C(\alpha)}{\mathcal{Q}(\alpha,t)^{1+\frac{1}{2\lambda}}}d\alpha}\sim\int_{\overline\alpha-r}^{\overline\alpha+r}{\frac{\lambda u_0'(\alpha)}{(1-\lambda\eta(t)u_0'(\alpha))^{1+\frac{1}{\lambda}}}d\alpha}.
\end{split}
\end{equation}
Consequently, if $\mathcal{Q}$ has as its earliest zero, $\eta_*$, a root of double multiplicity, then for $\eta_*-\eta>0$ small the above estimates imply that the time-evolution of (\ref{eq:ux}) can be examined, alternatively, via the simpler estimate
\begin{equation}
\label{eq:uxdouble}
\begin{split}
&u_x(\gamma(\alpha,t),t)\sim\left(\int_{\overline\alpha-r}^{\overline\alpha+r}{\frac{d\alpha}{(1-\lambda\eta(t)u_0'(\alpha))^{\frac{1}{\lambda}}}}\right)^{^{-2\lambda}}
\\
&\left\{\frac{u_0'(\alpha)}{1-\lambda\eta(t)u_0'(\alpha)}-\left(\int_{\overline\alpha-r}^{\overline\alpha+r}{\frac{d\alpha}{(1-\lambda\eta(t)u_0'(\alpha))^{\frac{1}{\lambda}}}}\right)^{^{-1}}\left(\int_{\overline\alpha-r}^{\overline\alpha+r}{\frac{\lambda u_0'(\alpha)\,d\alpha}{(1-\lambda\eta(t)u_0'(\alpha))^{1+\frac{1}{\lambda}}}}\right)\right\}.
\end{split}
\end{equation}
The right-hand-side of (\ref{eq:uxdouble}) was studied in \cite{Sarria1} and \cite{Sarria2} in connection to the giPJ equation. Thus, for the double root case estimates on the behaviour of the integrals (\ref{eq:podouble}) and (\ref{eq:otherdouble}) as $\eta\uparrow\eta_*$ are readily available in these works, and we simply direct the reader to Corollaries \ref{coro:special1} and \ref{coro:singular2} in Appendix \ref{sec:doublemult} for the corresponding regularity results. However, and for the sake of completeness, we will give a brief outline on how to obtain these estimates in the proof of Theorem \ref{thm:blow1ak<0}.

In light of the above discussion, in \S\ref{subsec:regularityak>0} we will only be concerned with the representations (\ref{eq:fac1}) and (\ref{eq:fac3}), the case where $\eta_*$ is a single root of $\mathcal{Q}$. Accordingly, define
\begin{equation}
\label{eq:M2}
M\equiv\max_{\alpha\in\Omega}\{2\lambda u_0'(\alpha)\}
\end{equation}
and
\begin{equation}
\label{eq:max1}
N\equiv\max_{\alpha\in\Sigma}g_1(\alpha).
\end{equation}
Notice that, while $M$ always exists, $N$ may not due to the vanishing of $C(\alpha)$ at finitely many points, which implies that $\Sigma$ is an open set. Also, note that there is no need to consider an eventual vanishing of the linear term in (\ref{eq:fac1}) involving $g_2$. Indeed, if the initial data is such that $g_2(\alpha)\leq0$, then due to the strictly increasing nature of $\eta(t)$ and $\eta(0)=0$, such term will never vanish. Moreover, if $g_2$ is somewhere positive, it is easy to see that, over those $\alpha\in\Sigma$ where $\rho_0(\alpha)\neq0$, we have $g_1(\alpha)>g_2(\alpha)$. As a result, for parameters $\lambda\kappa>0$, we conclude that there are two cases of interest concerning the least value $\eta_*>0$ at which $\mathcal{Q}$ vanishes. If $N$ does not exist, or if it does but $M>N$, we set
\begin{equation}
\label{eq:etaomega}
\eta_*=\frac{1}{M},
\end{equation}
whereas, for $N>M$, we let 
\begin{equation}
\label{eq:etasigma}
\eta_*=\frac{1}{N}.
\end{equation}
See below for two simple examples involving single roots.

\vspace{0.1in}
\textbf{Single Multiplicity Roots}
\vspace{0.1in}

\emph{Example 1.}\,\, For $(\lambda,\kappa)=(1,1)$, take $u_0'(\alpha)=\cos(2\pi\alpha)$ and $\rho_0(\alpha)\equiv1$. Then $C(\alpha)=0$ in (\ref{eq:c0}) implies that $\Omega=\{0,1/2,1\}$, the points where $\cos(2\pi\alpha)=\pm1$. Then $M=\max_{\Omega}\{2\cos(2\pi\alpha)\}=2$ occurs at both end-points $\alpha=0,1$. Now $g_1(\alpha)=\cos(2\pi\alpha)+1$, and so $N=\max_{\Sigma}g_1(\alpha)$ does not exist since the boundary points lie in $\Omega$. We conclude that 
\begin{equation}
\label{eq:q1}
\mathcal{Q}(\alpha,t)=(\cos(2\pi\alpha)^2-1)\eta^2-2\cos(2\pi\alpha)\eta+1\to0
\end{equation}
earliest at the boundary $\overline\alpha=\{0,1\}$ as $\eta\uparrow\eta_*=\frac{1}{M}=\frac{1}{2}$. For all other $\alpha\in(0,1)$ and $0\leq\eta\leq\eta_*$, $\mathcal{Q}>0$. See Figure (\ref{fig:q})-left below. 

\vspace{0.1in}

\emph{Example 2.}\,\, For $(\lambda,\kappa)$ and $u_0'$ as above, now let $\rho_0(\alpha)\equiv\frac{1}{2}$. Then $C(\alpha)=0$ gives $\Omega=\{1/6, 1/3, 2/3, 5/6\}$, and so $M=\max_{\Omega}\{2\cos(2\pi\alpha)\}=1$ is attained at $\alpha=1/6,5/6$. Now, this time $g_1(\alpha)=\cos(2\pi\alpha)+\frac{1}{2}$ so that $N=\max_{\Sigma}g_1(\alpha)=\frac{3}{2}$ occurs at both end-points, which, as opposed to the previous example, now lie in $\Sigma$. Since $N=\frac{3}{2}>1=M$, we have $\eta_*=\frac{1}{N}=\frac{2}{3}$ and 
\begin{equation}
\label{eq:q2}
\mathcal{Q}(\alpha,t)=\left(\cos(2\pi\alpha)^2-\frac{1}{4}\right)\eta^2-2\cos(2\pi\alpha)\eta+1\to0
\end{equation}
earliest at the boundary $\overline\alpha=\{0,1\}$ as $\eta\uparrow\eta_*$, whereas, for $\alpha\in(0,1)$ and $0\leq\eta\leq\eta_*$, $\mathcal{Q}>0$. See Figure (\ref{fig:q})-right below.

\begin{center}
\begin{figure}[!ht]
\includegraphics[scale=0.3]{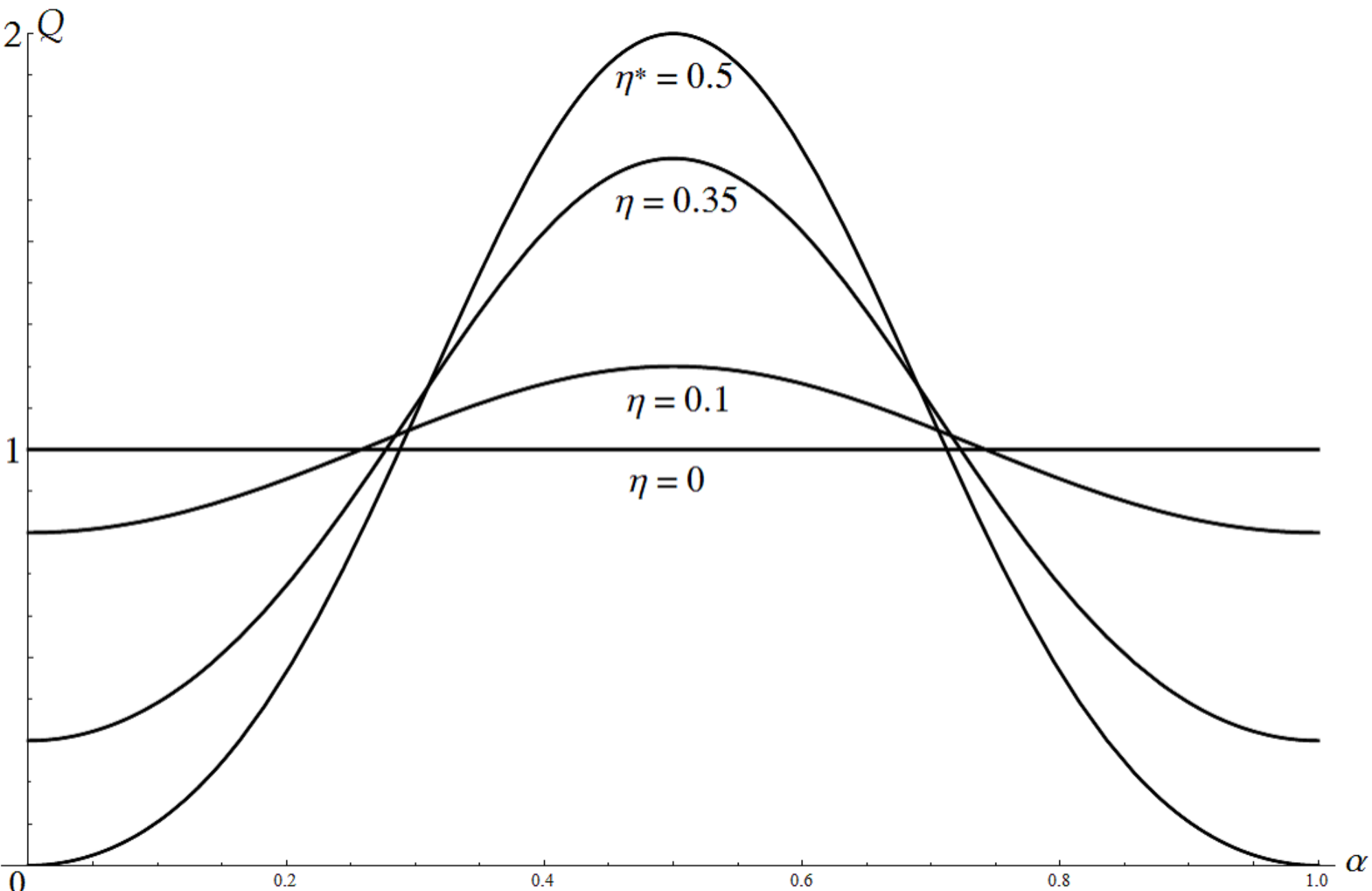} 
\includegraphics[scale=0.3]{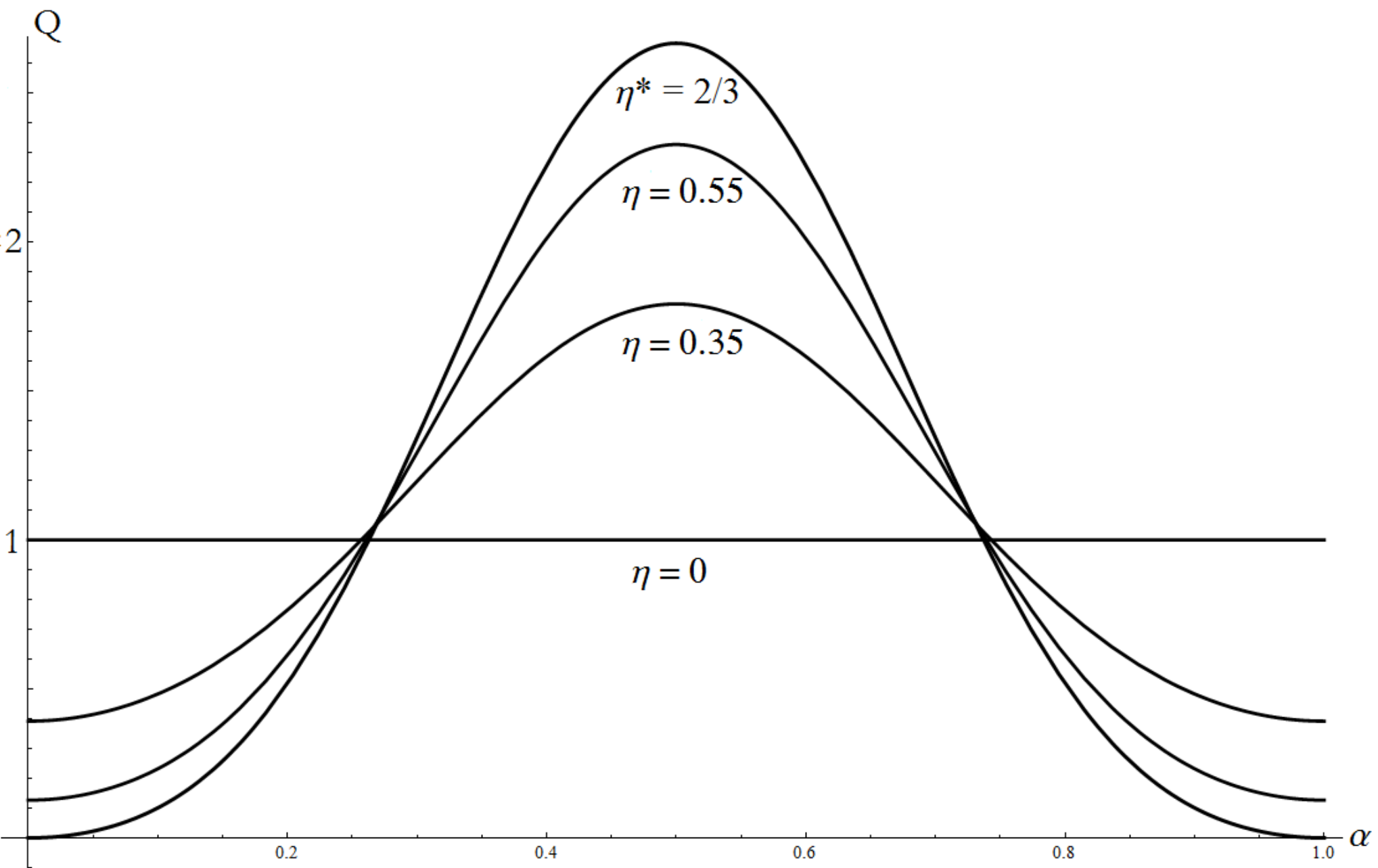} 
\caption{Plots of (\ref{eq:q1}) and (\ref{eq:q2}). Both vanish first at the boundary as $\eta$ approaches $1/2$ and $2/3$ respectively. For $\alpha\in(0,1)$ and $0\leq\eta\leq\eta_*$, $\mathcal{Q}>0$.}
\label{fig:q}
\end{figure}
\end{center}

\subsection{Regularity Results for $\lambda\kappa<0$}\hfill
\label{subsec:regularityak<0}

For parameters $\lambda\kappa<0$, Theorem \ref{thm:global1ak<0} below establishes conditions on the initial data for which both component solutions, (\ref{eq:ux}) and (\ref{eq:rho}), converge in finite time to steady states $U_{\infty}$ and respectively $P_{\infty}$, given by
\begin{equation}
\label{eq:steady}
\begin{split}
U_{\infty}(\alpha)=-\left(\frac{\mathcal{N}}{\mathcal{M}^{1+2\lambda}}+\frac{u_0'(\alpha)}{\rho_0(\alpha)}\,P_{\infty}(\alpha)\right),\,\,\,\,\,\,\,\,\,\,\,\,\,\,\,\,\,\,\,\,\,P_{\infty}(\alpha)=\frac{\rho_0(\alpha)}{C(\alpha)\mathcal{M}^{2\lambda}}
\end{split}
\end{equation}
where $C(\alpha)>0$ is as in (\ref{eq:c0}) and the real numbers $\mathcal{M}>0$ and $\mathcal{N}$ are defined as
\begin{equation}
\label{eq:steady2}
\begin{split}
\mathcal{M}\equiv\int_0^1{\frac{d\alpha}{C(\alpha)^{\frac{1}{2\lambda}}}},\,\,\,\,\,\,\,\,\,\,\,\,\,\,\,\,\,\,\,\,\,\,\,\,\,\,\mathcal{N}\equiv\int_0^1{\frac{u_0'(\alpha)}{C(\alpha)^{1+\frac{1}{2\lambda}}}d\alpha}.
\end{split}
\end{equation}
In contrast, Theorem \ref{thm:blow1ak<0} provides criteria leading to the existence of finite-time blow-up solutions. The reader may refer to \S\ref{sec:examples} for specific examples.

\subsubsection{Convergence to Steady States for $\lambda\kappa<0$}\hfill
\label{subsubsec:global1ak<0}

\begin{theorem}
\label{thm:global1ak<0}
Consider the initial boundary value problem (\ref{eq:hs})-(\ref{eq:pbc}) for parameters $\lambda\kappa<0$ and initial data $u_0'(x)$ and $\rho_0(x)$ both bounded and, at least, $C^0[0,1]\,\, a.e.$ If $\rho_0(\alpha)$ never vanishes, there exists a finite $t_{\infty}>0$ such that (\ref{eq:ux}) and (\ref{eq:rho}) converge to the steady states in (\ref{eq:steady}) as $t\uparrow t_{\infty}$. Similarly if $\lambda u_0'(\alpha_i)\leq0$ for $\alpha_i$\,, $1\leq i\leq n$, the locations where $\rho_0(\alpha)$ vanishes.
\end{theorem}

\begin{proof}

Suppose $\lambda$ and $\kappa$ are such that $\lambda\kappa<0$ and $\rho_0(\alpha)$ is never zero on $[0,1]$. The latter implies that, particularly, $\rho_0(\alpha)\neq0$ on $\Sigma$. This and $\lambda\kappa<0$ imply that (\ref{eq:D}) is negative and, thus, $0<\mathcal{Q}<+\infty$ on $\Sigma$ for all $0\leq\eta<+\infty$. Moreover $\Sigma=[0,1]$. Indeed, suppose $\Omega\neq\emptyset$, namely, that there are $\alpha\in[0,1]$ such that $C(\alpha)=0$. From (\ref{eq:c0}) this implies that
\begin{equation}
\label{eq:1}
\begin{split}
0\leq u_0'(\alpha)^2=\frac{\kappa}{\lambda}\rho_0(\alpha)^2<0,
\end{split}
\end{equation}
a contradiction. Therefore, $0<\mathcal{Q}<+\infty$ for all $\alpha\in[0,1]$ and $0\leq\eta(t)<+\infty$, but $\mathcal{Q}\to+\infty$ as $\eta\to+\infty$. Next, define real numbers $\mathcal{M}>0$ and $\mathcal{N}$ as in (\ref{eq:steady2}) and note that both are well-defined because $\lambda\kappa<0$ and $\rho_0\neq0$ imply that $C(\alpha)>0$. Then (\ref{eq:Q}) yields, for large enough $\eta>0$, the simple asymptotic estimates
\begin{equation}
\label{eq:large1}
\begin{split}
\mathcal{\bar{P}}_0\sim\mathcal{M}\,\eta^{-\frac{1}{\lambda}},\,\,\,\,\,\,\,\,\,\,\,\,\,\,\,\,\,\,\,\,\,\,\,\,\int_0^1{\frac{\lambda u_0'(\alpha)-\eta C(\alpha)}{\mathcal{Q}(\alpha,t)^{1+\frac{1}{2\lambda}}}}\sim\frac{\lambda\,\mathcal{N}-\mathcal{M}\,\eta}{\eta^{2+\frac{1}{\lambda}}}.
\end{split}
\end{equation}
Using (\ref{eq:large1})i) on (\ref{eq:rho}) we find that 
$$\rho(\gamma(\alpha,t),t)\sim\frac{\rho_0(\alpha)}{\mathcal{Q}(\alpha,t)}\left(\frac{\mathcal{M}}{\eta^{\frac{1}{\lambda}}}\right)^{-2\lambda}=\frac{\rho_0}{\mathcal{M}^{2\lambda}}\left(\frac{\eta^2}{C(\alpha)\eta^2-2\lambda u_0'\eta+1}\right).$$
Then, if $P_{\infty}(\alpha)$ denotes the limit as $\eta\to+\infty$ of the right-hand-side above, we get (\ref{eq:steady})ii). In a similar fashion, using (\ref{eq:large1}) on (\ref{eq:ux}) yields (\ref{eq:steady})i). Finally, since (\ref{eq:large1})i) implies that
\begin{equation}
\label{eq:finitet}
\lim_{\eta\to+\infty}\mathcal{\bar{P}}_0=
\begin{cases}
0,\,\,\,\,\,\,\,\,\,\,\,\,&\lambda>0,
\\
+\infty,\,\,\,\,\,\,\,\,\,\,\,&\lambda<0,
\end{cases}
\end{equation}
(\ref{eq:etaivp}) gives
$$\lim_{\eta\to+\infty}\frac{dt}{d\eta}=0,$$
that is, as $\eta\to+\infty$, $t$ ceases to be an increasing function of $\eta$ and converges to a finite value, which we denote by $t_{\infty}$. This establishes the first part of the Theorem. 

For the last part, denote by $\alpha_i\in[0,1]$, $1\leq i\leq n$, the points where $\rho_0$ vanishes. Moreover, assume there are finitely many of these points and suppose $\lambda u_0'(\alpha_i)\leq0$. Clearly, if $\alpha\notin\{\alpha_i\}$, the discriminant $\mathcal{D}$ in (\ref{eq:D}) is negative and $0<\mathcal{Q}<+\infty$ for such values of $\alpha$ and $0\leq\eta<+\infty$. Now, if $\alpha\in\{\alpha_i\}$ and $\lambda u_0'(\alpha_i)<0$, then $\rho_0(\alpha_i)=0$ so that $\mathcal{D}(\alpha_i)=0$ and
\begin{equation}
\label{eq:last1}
\mathcal{Q}(\alpha_i,t)=\left(\lambda u_0'(\alpha_i)\right)^2(\mathcal{H}-\eta(t))^2,\,\,\,\,\,\,\,\,\,\,\,\,\,\,\,\,\,\,\,\mathcal{H}=\frac{1}{\lambda u_0'(\alpha_i)}<0,
\end{equation}
which, once again, implies that $0<\mathcal{Q}<+\infty$ since $\eta\geq0$. Similarly for $u_0'(\alpha_i)=0$, in which case $\mathcal{Q}(\alpha_i,t)\equiv1$. At this point, we may now follow the argument used to prove the first part of the Theorem. This establishes our result.\end{proof}

\subsubsection{Blow-up Solutions for $\lambda\kappa<0$}\hfill
\label{subsubsec:blow1ak<0}

From Theorem \ref{thm:global1ak<0} above, note that we still have to consider the case $\rho_0(\alpha_i)=0$ and $\lambda u_0'(\alpha_i)>0$. As evidenced by the proof of the previous Theorem, the main issue with parameters $\lambda\kappa<0$ is the possibility of a vanishing discriminant (\ref{eq:D}), which in turn would lead to one real-valued, double root of $\mathcal{Q}$ and, possibly, divergent space-dependent terms and time-dependent integrals in (\ref{eq:ux}) and (\ref{eq:rho}). Below we show that, in this last case, there exist smooth initial data for which $u_x$ diverges in finite time. In contrast, the second component solution, $\rho$, will either persist globally in time, or, at least, up to the blow-up time for $u_x$. 

\begin{theorem}
\label{thm:blow1ak<0}
Consider the initial boundary value problem (\ref{eq:hs})-(\ref{eq:pbc}) for parameters $\lambda\kappa<0$. Suppose there are $\alpha_i\in[0,1]$, $1\leq i\leq n$, such that $\rho_0(\alpha_i)=0$ and $\lambda u_0'(\alpha_i)>0$. Then for $(\lambda,\kappa)\in\mathbb{R}^-\times\mathbb{R}^+$, or $(\lambda,\kappa)\in(1,+\infty)\times\mathbb{R}^-$, there exists a finite $t_*>0$ for which $u_x$ diverges as $t\uparrow t_*$ while $\rho$ remains bounded for $0\leq t\leq t_*$. In contrast, if $(\lambda,\kappa)\in(0,1]\times\mathbb{R}^-$, both solution components exist globally in time. More particularly, let the restriction of $\lambda u_0'$ to $\{\alpha_i\}$ attain its greatest value at $\alpha_1$. Then there exist smooth initial data such that
\begin{enumerate}
\item\label{it:1ak<0} For $(\lambda,\kappa)\in(-2,0)\times\mathbb{R}^+$, a ``one-sided'' singularity in $u_x$ occurs, that is, $u_x(\gamma(\alpha_1,t),t)$ diverges to minus infinity as $t\uparrow t_*$ but remains finite otherwise.

\item\label{it:2ak<0} For $(\lambda,\kappa)\in(-\infty,-2]\times\mathbb{R}^+$, $u_x$ undergoes ``two-sided, everywhere'' blow-up, namely, $u_x(\gamma(\alpha_1,t),t)\to-\infty$ as $t\uparrow t_*$ and diverges to plus infinity otherwise. 

\item\label{it:3ak<0} For $(\lambda,\kappa)\in(1,+\infty)\times\mathbb{R}^-$, $u_x(\gamma(\alpha_1,t),t)\to+\infty$ as $t\uparrow t_*$, while, for $\alpha\neq\alpha_1$, $u_x(\gamma(\alpha,t),t)$ blows up to negative infinity. 

\item\label{it:4ak<0} For $(\lambda,\kappa)\in(0,1]\times\mathbb{R}^-$, both $u_x$ and $\rho$ persist globally in time. In fact, for $(\lambda,\kappa)\in(0,1)\times\mathbb{R}^-$, $u_x$ vanishes as $t\to+\infty$ but approaches a non-trivial steady-state when $(\lambda,\kappa)\in\{1\}\times\mathbb{R}^-$. 
\vspace{0.05in}

\end{enumerate}

Properties of the global-in-time behaviour of $\rho\circ\gamma$ are given below.

\end{theorem}

\begin{proof}
For $\lambda\kappa<0$, suppose there are $\alpha_i=\{\alpha_1,...,\alpha_n\}\subset[0,1]$ where $\rho_0(\alpha_i)=0$ and $\lambda u_0'(\alpha_i)>0$. Notice that $\{\alpha_i\}\subseteq\Sigma$. Indeed, if $\alpha_i\in\Omega$ for some $i=1,...,n$, then $C(\alpha_i)=0$ implies that $u_0'(\alpha_i)=0$, a contradiction since $\lambda u_0'(\alpha_i)>0$ and $\lambda\neq0$. The reader may check, by following an argument similar to that used in Theorem \ref{thm:global1ak<0}, that 
\begin{equation}
\label{eq:qpos}
\begin{split}
0<\mathcal{Q}<+\infty
\end{split}
\end{equation}
for all $\alpha\notin\{\alpha_i\}$ and $0\leq\eta<+\infty$. This implies the boundedness of the space-dependent terms in (\ref{eq:ux}) and (\ref{eq:rho}) for $\alpha\notin\{\alpha_i\}$. Now, since $\mathcal{D}(\alpha_i)=0$ and $\lambda u_0'(\alpha_i)>0$, 
$$\mathcal{Q}(\alpha_i,t)=\left(\lambda u_0'(\alpha_i)\right)^2(\mathcal{H}-\eta(t))^2,\,\,\,\,\,\,\,\,\,\,\,\,\,\,\,\,\,\,\,\,\,\mathcal{H}=\frac{1}{\lambda u_0'(\alpha_i)}>0.$$ 
Let
\begin{equation}
\label{eq:eta*1}
\begin{split}
\eta_*=\frac{1}{c_0}>0
\end{split}
\end{equation}
where, without loss of generality, we have set
$$c_0\equiv\max_{\alpha\in\{\alpha_i\}}\left\{\lambda u_0'(\alpha)\right\}=\lambda u_0'(\alpha_1).$$
Notice that $c_0>0$ by periodicity of $u_0$. Since for $\alpha\in\{\alpha_i\}$, (\ref{eq:ux}) may be written as
\begin{equation}
\label{eq:uxak<0}
\begin{split}
u_x(\gamma(\alpha_i,t),t)=\frac{\mathcal{\bar P}_0(t)^{^{-2\lambda}}}{\lambda}\left\{\frac{1}{\mathcal{H}-\eta(t)}-\frac{1}{\mathcal{\bar P}_0(t)}\int_0^1{\frac{\lambda u_0'(\alpha)-\eta(t)C(\alpha)}{\mathcal{Q}(\alpha,t)^{1+\frac{1}{2\lambda}}}d\alpha}\right\},
\end{split}
\end{equation}
we see that its space-dependent term will diverge earliest when $\alpha=\alpha_1$ as $\eta\uparrow\eta_*$.
However, this does not necessarily imply blow-up of $u_x(\gamma(\alpha_1,t),t)$; we still have to determine the behaviour of 
\begin{equation}
\label{eq:integrals}
\begin{split}
\mathcal{\bar P}_0(t)=\int_0^1{\frac{d\alpha}{\mathcal{Q}(\alpha,t)^{\frac{1}{2\lambda}}}},\,\,\,\,\,\,\,\,\,\,\,\,\,\,\,\,\,\,\,\int_0^1{\frac{\lambda u_0'(\alpha)-\eta(t)C(\alpha)}{\mathcal{Q}(\alpha,t)^{1+\frac{1}{2\lambda}}}d\alpha}
\end{split}
\end{equation}
as $\eta\uparrow\eta_*$. Consider first the simple case where $\lambda\kappa<0$ for $(\lambda,\kappa)\in[-1/2,0)\times\mathbb{R}^+$, which implies that $\frac{1}{2\lambda}<0$ and $1+\frac{1}{2\lambda}\leq0$. Then for smooth enough initial data, both integral terms remain positive and finite for all $\eta\in[0,\eta_*]$. Indeed, if $\alpha\notin\{\alpha_i\}$, we have (\ref{eq:qpos}), whereas,
for $\alpha\in\{\alpha_i\}$, suppose there is $t_{\epsilon}>0$ and $\epsilon>0$ small such that $\eta_{\epsilon}\equiv\eta(t_{\epsilon})=\frac{1}{c_0+\epsilon}$. Then $0<\eta_{\epsilon}<\eta_*$ and
$$\mathcal{Q}(\alpha_i,t_{\epsilon})=\left(\lambda u_0'(\alpha_i)\right)^2\left(\frac{c_0-\lambda u_0'(\alpha)+\epsilon}{\lambda u_0'(\alpha_i)(c_0+\epsilon)}\right)^2>0$$
for all $\epsilon>0$. Also $\mathcal{Q}(\alpha_i,t_{\epsilon})\to0^+$, first, when $\alpha=\alpha_1$ as $\epsilon\to0$, that is, as $\eta\uparrow\eta_*$. Thus $\mathcal{Q}(\alpha_i,t)>0$ for all $0\leq\eta\leq\eta_*$. This, together with (\ref{eq:qpos}), 
implies that 
\begin{equation}
\label{eq:boundedk0ak<0}
\begin{split}
0<\mathcal{\bar P}_0(t)<+\infty
\end{split}
\end{equation}
for $0\leq\eta\leq\eta_*$ and $\lambda<0$. Letting $i=1$ in (\ref{eq:uxak<0}), we conclude that
\begin{equation}
\label{eq:blow1ak<0}
\begin{split}
u_x(\gamma(\alpha_1,t),t)\sim\frac{C}{\lambda}\left(\frac{1}{\eta_*-\eta}\right)\to-\infty
\end{split}
\end{equation}
as $\eta\uparrow\eta_*$ for $(\lambda,\kappa)\in[-1/2,0)\times\mathbb{R}^+$. If instead $\alpha\neq\alpha_1$, then $u_x\circ\gamma$ stays bounded for $\eta\in[0,\eta_*]$. The existence of a finite blow-up time $t_*>0$ such that (\ref{eq:etatoeta*}) holds follows from (\ref{eq:t*}) and (\ref{eq:boundedk0ak<0}). Moreover, since $\rho_0(\alpha_i)=0$, (\ref{eq:rho}) implies that
$$\rho(\gamma(\alpha_i,t),t)\equiv0,$$
whereas, for $\alpha\notin\{\alpha_i\}$, $\rho$ stays bounded for all $t\in\mathbb{R}^+$ due to (\ref{eq:qpos}) and (\ref{eq:boundedk0ak<0}). This establishes part (\ref{it:1ak<0}) of the Theorem for $(\lambda,\kappa)\in[-1/2,0)\times\mathbb{R}^+$. Next suppose $\lambda\kappa<0$ for $(\lambda,\kappa)\in(-\infty,-1/2)\times\mathbb{R}^+$. Then (\ref{eq:boundedk0ak<0}) holds while the behaviour of (\ref{eq:integrals})ii), as $\eta\uparrow\eta_*$, requires further study. But recall that $\eta_*$ in (\ref{eq:eta*1}) is a double root of $\mathcal{Q}$ corresponding to  $\alpha_1\in[0,1]$. Consequently, for $\alpha=\alpha_1$, the space-dependent term in (\ref{eq:ux}) will diverge earliest in the same fashion as (\ref{eq:spacedouble}), with $\overline\alpha=\alpha_1$. Furthermore, and also with $\overline\alpha$ replaced by $\alpha_1$, the integral terms in (\ref{eq:integrals}) will behave as in (\ref{eq:podouble}) and (\ref{eq:otherdouble}) for $\eta_*-\eta>0$ small. Consequently, the evolution of (\ref{eq:uxak<0}) as $\eta\uparrow\eta_*$ follows that of (\ref{eq:uxdouble}), for which, as remarked in \S\ref{subsec:notation}, estimates are readily available in \cite{Sarria1} and \cite{Sarria2}. However, and for convenience of the reader, below we give a brief outline of how to obtain these estimates for certain values of $\lambda$ and a particular class of smooth data. The reader may refer to \cite{Sarria2} for an argument suitable to larger classes of smooth and non-smooth data. First, since $\lambda<-1/2$ and $\lambda u_0'(\alpha_i)>0$, denote by $m_0<0$ the least value of $u_0'$ in $\{\alpha_i\}$, namely, $u_0'(\alpha_1)=m_0$. Then (\ref{eq:eta*1}) becomes
$$\eta_*=\frac{1}{\lambda m_0}.$$
Additionally, assume that $u_0''(\alpha_1)=0$ and $u_0'''(\alpha_1)\neq0$ so that, a Taylor expansion about $\alpha_1$, yields
$$u_0'(\alpha)\sim m_0+C_1(\alpha-\alpha_1)^2,\,\,\,\,\,\,\,\,\,\,\,\,\,\,\,C_1=\frac{u_0'''(\alpha_1)}{2}>0$$
for $0\leq\left|\alpha-\alpha_1\right|\leq r$ and arbitrarily small $r>0$. Then, for $\epsilon>0$ small,
\begin{equation}
\label{eq:long1}
\begin{split}
\int_{\alpha_1-r}^{\alpha_1+r}{\frac{(\epsilon+u_0'(\alpha)-m_0)d\alpha}{(\epsilon+u_0'(\alpha)-m_0)^{2\left(1+\frac{1}{2\lambda}\right)}}}&\sim\int_{\alpha_1-r}^{\alpha_1+r}{\frac{d\alpha}{(\epsilon+C_1(\alpha-\alpha_1)^2)^{1+\frac{1}{\lambda}}}}
\\
&=\frac{1}{\epsilon^{1+\frac{1}{\lambda}}}\int_{\alpha_1-r}^{\alpha_1+r}{\left[1+\left(\sqrt{\frac{C_1}{\epsilon}}(\alpha-\alpha_1)\right)^2\right]^{^{-1-\frac{1}{\lambda}}}d\alpha}
\\
&\sim\frac{2\epsilon^{-\frac{1}{2}-\frac{1}{\lambda}}}{\sqrt{C_1}}\int_0^{\frac{\pi}{2}}{\left(\cos\theta\right)^{\frac{2}{\lambda}}d\theta}
\end{split}
\end{equation}
where the last integral is obtained via the change of variables $\sqrt{\frac{C_1}{\epsilon}}(\alpha-\alpha_1)=\tan\theta$. Now recall the well-known properties of the gamma function (see for instance \cite{Magnus1}, \cite{Gamelin1}),
\begin{equation}
\label{eq:gamma}
\begin{split}
\int_0^1{t^{p-1}(1-t)^{s-1}dt}=\frac{\Gamma(p)\Gamma(s)}{\Gamma(p+s)},\,\,\,\,\,\,\,\,\,\,\,\,\,\,\,\,\,\,\,\Gamma(1+y)=y\Gamma(y)
\end{split}
\end{equation}
for $p,s,y>0$. Suppose $\lambda\in(-\infty,-2)$ and let $t=(\sin\theta)^2$, $p=\frac{1}{2}$, and $s=\frac{1}{\lambda}+\frac{1}{2}$ in (\ref{eq:gamma})i), to obtain
$$2\int_0^{\frac{\pi}{2}}{\left(\cos\theta\right)^{\frac{2}{\lambda}}d\theta}=\frac{\sqrt{\pi}\,\Gamma\left(\frac{1}{\lambda}+\frac{1}{2}\right)}{\Gamma\left(1+\frac{1}{\lambda}\right)}=C_2\in\mathbb{R}^+.$$
Consequently, for $\epsilon>0$ small,
\begin{equation}
\label{eq:est2}
\begin{split}
\int_{\alpha_1-r}^{\alpha_1+r}{\frac{\epsilon+u_0'(\alpha)-m_0}{(\epsilon+u_0'(\alpha)-m_0)^{2\left(1+\frac{1}{2\lambda}\right)}}d\alpha}\sim\frac{C_2}{\sqrt{C_1}}\epsilon^{-\frac{1}{2}-\frac{1}{\lambda}}.
\end{split}
\end{equation}
But since 
$$\lambda u_0'(\alpha)-\eta(t)C(\alpha)\sim\lambda u_0'(\alpha)(1-\lambda\eta(t)u_0'(\alpha))$$
for $\alpha$ arbitrarily close to $\alpha_1$,\footnote[4]{See Remark \ref{rem:special1}.} we let $\epsilon=m_0-\frac{1}{\lambda\eta}$ into (\ref{eq:est2}) to find, for $\eta_*-\eta>0$ small and $\lambda<-2$, 
\begin{equation}
\label{eq:estk11}
\begin{split}
\int_0^1{\frac{\lambda u_0'(\alpha)-\eta(t)C(\alpha)}{\mathcal{Q}(\alpha,t)^{1+\frac{1}{2\lambda}}}d\alpha}\sim\frac{C_2}{(\eta_*-\eta(t))^{\frac{1}{2}+\frac{1}{\lambda}}}.
\end{split}
\end{equation}
Setting $i=1$ in (\ref{eq:uxak<0}) and using the above, we see that the space-dependent term dominates,
\begin{equation}
\label{eq:onesided1}
\begin{split}
u_x(\gamma(\alpha_1,t),t)\sim-\frac{C}{\eta_*-\eta}\to-\infty
\end{split}
\end{equation}
as $\eta\uparrow\eta_*$ for $(\lambda,\kappa)\in(-\infty,-2)\times\mathbb{R}^+$.\footnote[5]{Throughout the paper, $C$ will denote a generic positive constant. Also note that, to avoid confusion in notation between $C$ and the function $C(\alpha)$ in (\ref{eq:c0}), we will always emphasize the $\alpha$-dependence in the latter.} In contrast, for $\alpha\neq\alpha_1$, (\ref{eq:ux}) for $\alpha\not\in\{\alpha_i\}$, or (\ref{eq:uxak<0}) if $\alpha\in\{\alpha_i\}$, imply that (\ref{eq:estk11}) now dominates and
\begin{equation}
\label{eq:onesided2}
\begin{split}
u_x(\gamma(\alpha,t),t)\sim\frac{C}{(\eta_*-\eta)^{\frac{1}{2}+\frac{1}{\lambda}}}\to+\infty.
\end{split}
\end{equation}
The existence of a finite blow-up time $t_*>0$ follows from (\ref{eq:t*}) and (\ref{eq:boundedk0ak<0}). Lastly, for $(\lambda,\kappa)\in(-2,-1/2)\times\mathbb{R}^+$, we follow the argument outlined in Appendix B, also used in the proof of Theorem 4.36 in \cite{Sarria1}, to show that (\ref{eq:integrals})ii) remains finite as $\eta\uparrow\eta_*$. Consequently, for $\alpha=\alpha_1$ and $(\lambda,\kappa)\in(-2,-1/2)\times\mathbb{R}^+$, (\ref{eq:uxak<0}) diverges according to (\ref{eq:blow1ak<0}) but remains finite otherwise. 

Now, from the argument in the case $\lambda\in[-1/2,0)$, $\rho(\gamma(\alpha_i,t),t)\equiv0$, while, for $\alpha\notin\{\alpha_i\}$, $\rho$ stays finite for $t\in\mathbb{R}^+$ due to (\ref{eq:qpos}) and the fact that (\ref{eq:estk11}) is absent in its representation formula (\ref{eq:rho}). This establishes parts (\ref{it:1ak<0}) and (\ref{it:2ak<0}). Lastly, suppose $\lambda\kappa<0$ for $(\lambda,\kappa)\in\mathbb{R}^+\times\mathbb{R}^-$. Then estimates on both terms in (\ref{eq:integrals}), as $\eta\uparrow\eta_*$, are needed. Now, because $\lambda>0$ and $\lambda u_0'(\alpha_i)>0$, (\ref{eq:eta*1}) becomes
$$\eta_*=\frac{1}{\lambda M_0}$$
where $M_0>0$ denotes the greatest value attained by $u_0'$ on $\{\alpha_i\}$, i.e. $u_0'(\alpha_1)=M_0$. As in the previous case, we estimate the integral terms for $u_0'$ satisfying $u_0''(\alpha_1)=0$ and $u_0'''(\alpha_1)\neq0$. Following an argument analogous to the one leading to estimate (\ref{eq:estk11}), we find that, for $\eta_*-\eta>0$ small,
\begin{equation}
\label{eq:estk01}
\mathcal{\bar{P}}_0(t)\sim
\begin{cases}
C_4(\eta_*-\eta(t))^{\frac{1}{2}-\frac{1}{\lambda}},\,\,\,\,\,\,\,\,\,\,\,&\lambda\in(0,2),
\\
-C\ln(\eta_*-\eta(t)),\,\,\,\,\,\,\,\,\,\,\,\,\,&\lambda=2,
\\
C,\,\,\,\,\,\,&\lambda\in(2,+\infty),
\end{cases}
\end{equation}
while\footnote[6]{Estimate (\ref{eq:estk01})iii) is obtained by following the argument outlined in Appendix \ref{sec:furtherintegrals}.}
\begin{equation}
\label{eq:estk12}
\int_0^1{\frac{\lambda u_0'(\alpha)-\eta(t)C(\alpha)}{\mathcal{Q}(\alpha,t)^{1+\frac{1}{2\lambda}}}}\sim C_5(\eta_*-\eta(t))^{-\left(\frac{1}{2}+\frac{1}{\lambda}\right)},\,\,\,\,\,\,\,\,\,\,\,\,\,\,\,\,\lambda\in\mathbb{R}^+.
\end{equation}
The generic constants $C\in\mathbb{R}^+$ in (\ref{eq:estk01})ii), iii) depend only on $\lambda$, and the positive constants $C_4$ and $C_5$ are given by
$$C_4=\frac{\Gamma\left(\frac{1}{\lambda}-\frac{1}{2}\right)}{\Gamma\left(\frac{1}{\lambda}\right)}\sqrt{\frac{\pi M_0}{\left|C_3\right|}}(\lambda M_0)^{\frac{1}{2}-\frac{1}{\lambda}},\,\,\,\,\,\,\,\,\,\,\,\,\,\,\,\,\,\,\,\,\,\,\lambda\in(0,2),$$
and
$$C_5=\frac{\Gamma\left(\frac{1}{\lambda}+\frac{1}{2}\right)}{\lambda\,\Gamma\left(1+\frac{1}{\lambda}\right)}\sqrt{\frac{\pi M_0}{\left|C_3\right|}}(\lambda M_0)^{\frac{1}{2}-\frac{1}{\lambda}},\,\,\,\,\,\,\,\,\,\,\,\,\,\,\,\lambda\in\mathbb{R}^+.$$
For future reference, we note that (\ref{eq:gamma})ii) implies that
\begin{equation}
\label{eq:c4c5}
\frac{C_5}{C_4}=\frac{1}{\lambda}-\frac{1}{2}>0,\,\,\,\,\,\,\,\,\,\,\,\,\,\,\,\,\,\,\,\,\,\,\,\,\,\,\,\,\,\,\,\,\lambda\in(0,2).
\end{equation}
First let $\lambda\in(0,2)$ and $i=1$ in (\ref{eq:uxak<0}). Then using the corresponding estimates we find that
\begin{equation}
\label{eq:estux1}
u_x(\gamma(\alpha_1,t),t)\sim\frac{(3\lambda-2)C}{(\eta_*-\eta(t))^{\lambda-1}}\to
\begin{cases}
0,\,\,\,\,\,\,\,\,\,\,\,\,\,\,\,&\lambda\in(0,1),
\\
C,\,\,\,\,\,\,&\lambda=1,
\\
+\infty,\,\,\,\,\,\,\,&\lambda\in(1,2)
\end{cases}
\end{equation}
as $\eta\uparrow\eta_*$. If instead $\alpha\neq\alpha_1$, then (\ref{eq:ux}) and (\ref{eq:uxak<0}) give
\begin{equation}
\label{eq:estux2}
u_x(\gamma(\alpha,t),t)\sim-\frac{C}{(\eta_*-\eta(t))^{\lambda-1}}\to
\begin{cases}
0,\,\,\,\,\,\,\,\,\,\,\,\,\,\,\,&\lambda\in(0,1),
\\
C,\,\,\,\,\,\,&\lambda=1,
\\
-\infty,\,\,\,\,\,\,\,&\lambda\in(1,2).
\end{cases}
\end{equation}
For the second component $\rho$ in (\ref{eq:rho}), if $\alpha\in\{\alpha_i\}$ then $\rho\equiv0$ due to $\rho_0(\alpha_i)=0$. Similarly for $\alpha\notin\{\alpha_i\}$ such that $\rho_0(\alpha)=0$. Finally, when $\alpha\notin\{\alpha_i\}$ and $\rho_0(\alpha)\neq0$, (\ref{eq:D}) gives $\mathcal{D}<0$ and so (\ref{eq:qpos}) holds. Consequently, (\ref{eq:rho}) and (\ref{eq:estk01})i) yield
$$\rho(\gamma(\alpha,t),t)\sim C\rho_0(\alpha)(\eta_*-\eta)^{2-\lambda}\to0$$
as $\eta\uparrow\eta_*$ for $\lambda\in(0,2)$. Lastly, from (\ref{eq:etaivp}) we have that
\begin{equation}
\label{eq:etaivpest1}
dt=\mathcal{\bar{P}}_0(t)^{2\lambda}d\eta.
\end{equation}
Then using (\ref{eq:estk01})i) on the above gives
\begin{equation}
\label{eq:etaivpest2}
t_*-t\sim C\int_{\eta}^{\eta_*}{(\eta_*-\mu)^{\lambda-2}d\mu}.
\end{equation}
As a result $t_*=+\infty$ for $\lambda\in(0,1]$ but $0<t_*<+\infty$ if $\lambda\in(1,2)$. 

Last suppose $(\lambda,\kappa)\in(2,+\infty)\times\mathbb{R}^-$. Then letting $i=1$ in (\ref{eq:uxak<0}), and using (\ref{eq:estk01})iii) and (\ref{eq:estk12}), we find that
$$u_x(\gamma(\alpha_1,t),t)\sim\frac{C}{\eta_*-\eta}\to+\infty$$
as $\eta\uparrow\eta_*$. If instead $\alpha\neq\alpha_1$, the integral (\ref{eq:estk12}) in (\ref{eq:ux}), or (\ref{eq:uxak<0}), dominates and
$$u_x(\gamma(\alpha,t),t)\sim-C(\eta_*-\eta)^{-\left(\frac{1}{2}+\frac{1}{\lambda}\right)}\to-\infty.$$
Moreover, using (\ref{eq:rho}), (\ref{eq:qpos}) and (\ref{eq:estk01})iii), we find that, as in the previous case,  $\rho\circ\gamma\equiv0$ whenever $\rho_0$ is zero, while for $\alpha\notin\{\alpha_i\}$ such that $\rho_0(\alpha)\neq0$, $\rho\circ\gamma\to C\in\mathbb{R}^+$ as $\eta\uparrow\eta_*$. Finally, the existence of a finite blow-up time $t_*>0$ for $u_x\circ\gamma$ follows from (\ref{eq:estk01})iii) and (\ref{eq:etaivpest1}) as $\eta\uparrow\eta_*$. Also, by using (\ref{eq:estk01})ii) and a similar argument as above, it can be shown that $(u_x,\rho)\circ\gamma$, for $(\lambda,\kappa)\in\{2\}\times\mathbb{R}^-$, behave as in the case $(\lambda,\kappa)\in(2,+\infty)\times\mathbb{R}^-$. This concludes the proof of parts (\ref{it:3ak<0}) and (\ref{it:4ak<0}), and thus establishes the Theorem.
\end{proof}

%
%
%

\subsection{Regularity Results for $\lambda\kappa>0$}\hfill
\label{subsec:regularityak>0}

In this section, we are concerned with regularity properties of (\ref{eq:ux}) and (\ref{eq:rho}) for $\Omega$ in (\ref{eq:omegaset}) non-empty and parameters $\lambda\kappa>0$.\footnote[7]{The case $\Omega=\emptyset$ follows similarly.} Below we will see how, of the two cases $\lambda\kappa<0$ or $\lambda\kappa>0$, the latter represents the ``most singular'' in the sense that, relative to a class of smooth initial data, spontaneous singularities may now form in $\rho$. This should not come as a surprise if we note that for $\lambda\kappa>0$, as opposed to $\lambda\kappa<0$, the discriminant (\ref{eq:D}) now satisfies $\mathcal{D}(\alpha)\geq0$, and so a root (\ref{eq:eta*}) of single multiplicity corresponding to $\overline\alpha\in[0,1]$ with $\rho_0(\overline\alpha)\neq0$, may now occur. Furthermore, we remind the reader that only the case where (\ref{eq:eta*}) is a single root of $\mathcal{Q}$ is considered in this section. Although such case arises the most for $\rho_0(\alpha)\not\equiv0$, in Appendix \ref{sec:doublemult} regularity results for the instance of a double multiplicity root are presented and examples of nonsmooth initial data for which it occurs are given.

\begin{theorem}
\label{thm:ak>0lambdanegkappapos}
Consider the initial boundary value problem (\ref{eq:hs})-(\ref{eq:pbc}) for $(\lambda,\kappa)\in\mathbb{R}^-\times\mathbb{R}^-$. Let sets $\Omega$ and $\Sigma$ be defined as in (\ref{eq:omegaset}) and (\ref{eq:sigma}), and denote by $\overline\alpha$ the finite number of locations in $[0,1]$ where the largest of
\begin{equation}
\label{eq:space2}
M\equiv\max_{\stackrel{\alpha\in\Omega}{\rho_0(\alpha)\neq0}}\{2\lambda u_0'(\alpha)\},\,\,\,\,\,\,\,\,\,\,\,\,\,\,\,\,\,\,\,\,\,\,\,\,N\equiv\max_{\stackrel{\alpha\in\Sigma}{\rho_0(\alpha)\neq0}}\{\lambda u_0'(\alpha)+\sqrt{\lambda\kappa}\left|\rho_0(\alpha)\right|\}
\end{equation}
is attained. Then there exist smooth initial data and a finite $t_*>0$ such that
\begin{enumerate}
\item\label{it:ak>011} For $(\lambda,\kappa)\in\mathbb{R}^-\times\mathbb{R}^-$, $u_x(\gamma(\overline\alpha,t),t)\to-\infty$ as $t\uparrow t_*$, whereas, if $\alpha\neq\overline\alpha$, it remains finite for $(\lambda,\kappa)\in(-1,0)\times\mathbb{R}^-$ and $0\leq t\leq t_*$, but diverges to plus infinity, as $t\uparrow t_*$, when $(\lambda,\kappa)\in(-\infty,-1]\times\mathbb{R}^-$. 

\item\label{it:ak>012} For $\rho_0(\overline\alpha)>0$ or $\rho_0(\overline\alpha)<0$, $\rho(\gamma(\overline\alpha,t),t)$ diverges, as $t\uparrow t_*$, to plus or respectively minus infinity, but remains finite otherwise.
\end{enumerate}
\end{theorem}

\begin{proof}

We consider the case where $\eta_*>0$, the earliest zero of $\mathcal{Q}$, has multiplicity one. Refer to Corollary \ref{coro:special1} in Appendix \ref{sec:doublemult} for the double multiplicity case ($N>M$ with $\rho_0(\overline\alpha)=0$), and see Appendix \ref{sec:simplecases} for $\Omega=\emptyset$. 

For $(\lambda,\kappa)\in\mathbb{R}^-\times\mathbb{R}^-$, let $\overline\alpha\in[0,1]$ denote the finite number of points where the largest between $M$ and, if it exists, $N$, both as defined in (\ref{eq:space2}), is attained\footnote[8]{If $N$ is not defined simply use $M$.}. Without loss of generality, we will assume that $N$ exists and $N>M$; otherwise, you may use an almost identical argument to the one presented below. Set 
$$\eta_*=\frac{1}{N}.$$
Then the space-dependent term in (\ref{eq:ux}) will vanish, first, when $\alpha=\overline\alpha$ as $\eta\uparrow\eta_*$. However, we still need to consider the behaviour of the integral terms in (\ref{eq:integrals}). As in the proof of Theorem \ref{thm:blow1ak<0}, we begin with the simple case where $(\lambda,\kappa)\in[-1/2,0)\times\mathbb{R}^-$. For such values of $\lambda$, the integral terms in (\ref{eq:ux}) and (\ref{eq:rho}) remain finite for smooth enough initial data and, thus, the space-dependent term in (\ref{eq:ux}) leads to blow-up, i.e. for $\alpha=\overline\alpha$,
\begin{equation*}
u_x(\gamma(\overline\alpha,t),t)\to-\infty
\end{equation*}
as $\eta\uparrow\eta_*$. In contrast, if  $\alpha\neq\overline\alpha$, then the space-dependent term, and consequently $u_x(\gamma(\alpha,t),t)$, remain finite for all $0\leq\eta\leq\eta_*$ and $(\lambda,\kappa)\in[-1/2,0)\times\mathbb{R}^-$. Moreover, (\ref{eq:space2})ii) implies that $\rho_0(\overline\alpha)\neq0$, consequently (\ref{eq:rho}) and boundedness of (\ref{eq:integrals})i) yields, as $\eta\uparrow\eta_*$, 
\begin{equation}
\label{eq:rhoblow1}
\rho(\gamma(\overline\alpha,t),t)\to
\begin{cases}
+\infty,\,\,\,\,\,\,\,\,\,\,\,\,\,\,\,&\rho_0(\overline\alpha)>0,
\\
-\infty,\,\,\,\,\,\,\,\,\,\,\,\,\,\,\,&\rho_0(\overline\alpha)<0,
\end{cases}
\end{equation}
but remains finite otherwise. The existence of a finite blow-up time $t_*>0$ in this case follows from (\ref{eq:t*}) in the limit as $\eta\uparrow\eta_*$. Actually, because only the integral term (\ref{eq:integrals})i) appears in (\ref{eq:rho}), we have in fact established part (\ref{it:ak>012}) of the Theorem.

To finish the proof of part (\ref{it:ak>011}), let $(\lambda,\kappa)\in(-\infty,-1/2)\times\mathbb{R}^-$, so that estimates on (\ref{eq:integrals})ii), as $\eta\uparrow\eta_*$, are needed. We will use the approach in Theorem \ref{thm:blow1ak<0} (recall estimate (\ref{eq:estk11})). For $\overline\alpha$ as defined above, namely $N=g_1(\overline\alpha)>M$ with $g_1$ as in (\ref{eq:g})i), suppose $g'(\overline\alpha)=0$ and $g''(\overline\alpha)<0$. Then smoothness of the initial data implies, by a simple Taylor expansion about $\overline\alpha$, that
$$g_1(\alpha)\sim N+C_1(\alpha-\overline\alpha)^2,\,\,\,\,\,\,\,\,\,\,\,\,\,\,\,\,C_1=\frac{g''(\overline\alpha)}{2}<0$$
for $0\leq\left|\alpha-\overline\alpha\right|\leq r$ and small $r>0$. Consequently, for $\epsilon>0$ small,
$$\epsilon-g_1(\alpha)+N\sim\epsilon+\left|C_1\right|(\alpha-\overline\alpha)^2,$$
so that
\begin{equation}
\label{eq:ak>0eq1}
\begin{split}
\int_{\overline\alpha-r}^{\overline\alpha+r}{\frac{d\alpha}{(\epsilon-g_1(\alpha)+N)^{1+\frac{1}{2\lambda}}}}&\sim\int_{\overline\alpha-r}^{\overline\alpha+r}{\frac{d\alpha}{(\epsilon+\left|C_1\right|(\alpha-\overline\alpha)^2)^{1+\frac{1}{2\lambda}}}}
\\
&\sim\frac{2\epsilon^{-\frac{1}{2}\left(1+\frac{1}{\lambda}\right)}}{\sqrt{\left|C_1\right|}}\int_0^{\frac{\pi}{2}}{(\cos\theta)^{\frac{1}{\lambda}}d\theta}.
\end{split}
\end{equation}
Following the derivation of (\ref{eq:estk11}), with $\epsilon=\frac{1}{\eta}-N$ instead, we find that
\begin{equation}
\label{eq:ak>0eq2}
\begin{split}
\int_{0}^{1}{\frac{\lambda u_0'(\alpha)-\eta(t)C(\alpha)}{\mathcal{Q}(\alpha,t)^{1+\frac{1}{2\lambda}}}d\alpha}\sim\frac{ C_6}{(1-\eta(t)N)^{^{\frac{1}{2}\left(1+\frac{1}{\lambda}\right)}}}
\end{split}
\end{equation}
for $\eta_*-\eta(t)>0$ small, $\lambda\in(-\infty,-1)$, and
$$C_6=\frac{\Gamma\left(\frac{1}{2}+\frac{1}{2\lambda}\right)}{\Gamma\left(1+\frac{1}{2\lambda}\right)}\sqrt{\frac{\pi N}{\left|C_1\right|}}\in\mathbb{R}^+.$$
Setting $\alpha=\overline\alpha$ into (\ref{eq:ux}) and using (\ref{eq:ak>0eq2}) leads to a dominating space-dependent term,
$$u_x(\gamma(\overline\alpha,t),t)\sim-\frac{C}{1-\eta(t)N}\to-\infty$$
as $\eta\uparrow\eta_*$ for $(\lambda,\kappa)\in(-\infty,-1]\times\mathbb{R}^-$. If $\alpha\neq\overline\alpha$, (\ref{eq:ak>0eq2}) takes control and
$$u_x(\gamma(\alpha,t),t)\sim\frac{C}{(1-\eta(t)N)^{\frac{1}{2}\left(1+\frac{1}{\lambda}\right)}}\to+\infty.$$
The value $\lambda=-1$ is considered separately; it yields a logarithmic blow-up rate for (\ref{eq:integrals})ii) which leads to the same blow-up behaviour as above for $u_x\circ\gamma$. By slightly modifying the argument outlined in Appendix \ref{sec:furtherintegrals}, it can be shown that for $\lambda\in(-1,-1/2)$, (\ref{eq:integrals})ii) converges as $\eta\uparrow\eta_*$. As a result, the regularity results derived above for $\lambda\in[-1/2,0)$ apply. This concludes the proof of the Theorem. For specific examples, the reader may turn to \S\ref{sec:examples}.
\end{proof}


\begin{remark}
\label{rem:special1}
To derive (\ref{eq:ak>0eq2}) above, we have assumed that the top and the bottom terms in the integrand of (\ref{eq:integrals})ii) do not vanish simultaneously as $\eta\uparrow\eta_*$. Recall that neither $\eta_*=\frac{1}{N}$ nor, when applicable, $\eta_*=\frac{1}{M}$ are double roots of $\mathcal{Q}$. This and the identity
$$-\frac{1}{2}\frac{\partial\mathcal{Q}}{\partial\eta}=\lambda u_0'(\alpha)-\eta(t)C(\alpha)$$
imply that a simultaneous vanishing of both terms in (\ref{eq:integrals})ii), as $\eta\uparrow\eta_*$, occurs only if $\eta_*$ is a double root. In fact, the term on the right above remains positive near $\overline\alpha$ as $\eta\uparrow\eta_*$. Indeed, since $C(\alpha)=g_1(\alpha)g_2(\alpha)$ for $g_1$ and $g_2$ as in (\ref{eq:g}) and $C(\alpha)$ in (\ref{eq:c0}), we see that for $\alpha\sim\overline\alpha$,
\begin{equation}
\label{eq:limit12}
\lambda u_0'(\alpha)-\eta(t)C(\alpha)\sim\lambda u_0'(\overline\alpha)-\eta(t)Ng_2(\overline\alpha),
\end{equation}
and then, as $\eta\uparrow\eta_*=\frac{1}{N}$,
\begin{equation}
\label{eq:limit1}
\lambda u_0'(\alpha)-\eta(t)C(\alpha)\to\lambda u_0'(\overline\alpha)-g_2(\overline\alpha)=\sqrt{\lambda\kappa}\left|\rho_0(\overline\alpha)\right|.
\end{equation}
But $\eta_*$ is not a double root of $\mathcal{Q}$, that is $\mathcal{D}(\overline\alpha)$ in (\ref{eq:D}) is positive, which gives $\rho_0(\overline\alpha)\neq0$ and, thus, the term on the right is positive.
\end{remark}

\begin{remark}
Notice that, if $N$ does not exist, $g_1$ cannot attain positive values greater than $M$ at infinitely many points in $\Sigma$. Indeed, suppose $N$ does not exist and recall that $\Omega$ is assumed to be a discrete finite set. Then there is $\tilde\alpha\in\Omega$ such that
\begin{equation}
\label{eq:eq1}
g_1(\alpha)=\lambda u_0'(\alpha)+\sqrt{\lambda\kappa\rho_0(\alpha)^2}<\lambda u_0'(\tilde\alpha)+\sqrt{\lambda\kappa\rho_0(\tilde\alpha)^2}
\end{equation}
for all $\alpha\in\Sigma$. But $\tilde\alpha\in\Omega$ and (\ref{eq:c0}) imply $\lambda\kappa\rho_0(\tilde\alpha)^2=\left(\lambda u_0'(\tilde\alpha)\right)^2$, which we substitute into the right-hand-side of (\ref{eq:eq1}) to obtain
\begin{equation}
\label{eq:eq2}
g_1(\alpha)<
\begin{cases}
2\lambda u_0'(\tilde\alpha),\,\,\,\,\,\,\,\,\,\,\,\,\,&\lambda u_0'(\tilde\alpha)\geq0,
\\
0,\,\,\,\,\,\,\,\,\,\,\,\,\,&\lambda u_0'(\tilde\alpha)<0.
\end{cases}
\end{equation}
Clearly, if (\ref{eq:etaomega}) holds, $2\lambda u_0'(\tilde\alpha)$ above becomes $M$ for $\tilde\alpha=\overline\alpha$. The above argument justifies (\ref{eq:etaomega}) as the earliest $\eta-$value causing blow-up of the space-dependent terms in the case where $N$ does not exist. A similar argument to the one above may be used for the case where $g_2$ in (\ref{eq:g})ii) has a positive maximum. 
\end{remark}

This last part of the paper studies the most singular class of solutions. Suppose $\lambda$ and $\kappa$ satisfy $\lambda\kappa>0$ with $(\lambda,\kappa)\in\mathbb{R}^+\times\mathbb{R}^+$. Below we show the existence of smooth initial data for which, as long as the associated root $\eta_*$ of $\mathcal{Q}$ has single multiplicity, both $u_x$ and $\rho$ diverge in finite time for all $(\lambda,\kappa)\in\mathbb{R}^+\times\mathbb{R}^+$. More particularly, $u_x$ undergoes a two-sided, everywhere blow-up, while $\rho$ diverges at finitely many points to either plus or minus infinity. In contrast, if $\eta_*$ represents a double root, then both global-in-time and blow-up solutions exist. For the latter result refer to Corollary \ref{coro:singular2} in Appendix \ref{sec:doublemult}.

\begin{theorem}
\label{thm:singular}
Consider the initial boundary value problem (\ref{eq:hs})-(\ref{eq:pbc}) for $(\lambda,\kappa)\in\mathbb{R}^+\times\mathbb{R}^+$. Let sets $\Omega$ and $\Sigma$ be as in (\ref{eq:omegaset}) and (\ref{eq:sigma}), and denote by $\overline\alpha$ the finite number of locations in $[0,1]$ where the largest of
\begin{equation}
\label{eq:space3}
M\equiv\max_{\stackrel{\alpha\in\Omega}{\rho_0(\alpha)\neq0}}\{2\lambda u_0'(\alpha)\},\,\,\,\,\,\,\,\,\,\,\,\,\,\,\,\,\,\,\,\,\,\,\,\,N\equiv\max_{\stackrel{\alpha\in\Sigma}{\rho_0(\alpha)\neq0}}\{\lambda u_0'(\alpha)+\sqrt{\lambda\kappa}\left|\rho_0(\alpha)\right|\}
\end{equation}
is attained. Then there exist smooth initial data and a finite $t_*>0$ such that, as $t\uparrow t_*$, $u_x(\gamma(\overline\alpha,t),t)\to+\infty$, while $u_x\circ\gamma\to-\infty$ otherwise. Moreover, if $\rho_0(\overline\alpha)>0$ or $\rho_0(\overline\alpha)<0$, $\rho(\gamma(\overline\alpha,t),t)$ diverges to plus or respectively minus infinity as $t\uparrow t_*$, whereas, for $\alpha\neq\overline\alpha$ such that $\rho_0(\overline\alpha)\neq0$, $\rho\circ\gamma$ vanishes as $t\uparrow t_*$ for $(\lambda,\kappa)\in(0,1]\times\mathbb{R}^+$, but converges to a non-trivial steady state when $(\lambda,\kappa)\in(1,+\infty)\times\mathbb{R}^+$. Lastly, if $\alpha\neq\overline\alpha$ but $\rho_0(\alpha)=0$, then $\rho\circ\gamma\equiv0$ for all time and $(\lambda,\kappa)\in\mathbb{R}^+\times\mathbb{R}^+$.

\end{theorem}

\begin{proof}
As with Theorem \ref{thm:ak>0lambdanegkappapos}, we consider the case where $\eta_*>0$ has multiplicity one. More particularly, and without loss of generality, assume $N$ in (\ref{eq:space3})ii) exists and $N>M$, so that the least, positive zero of $\mathcal{Q}$ with single multiplicity is given by
$$\eta_*=\frac{1}{N}.$$
The reader may refer to Corollary \ref{coro:singular2} in Appendix \ref{sec:doublemult} for the double multiplicity case $N>M$ with $\rho_0(\overline\alpha)=0$ as well as Appendix \ref{sec:simplecases} for $\Omega=\emptyset$. 

Let $(\lambda,\kappa)\in\mathbb{R}^+\times\mathbb{R}^+$ and denote by $\overline\alpha\in[0,1]$ the finite number of points where $N$ is attained, that is $N=g_1(\overline\alpha)$ for $g_1$ as in (\ref{eq:g})i). More particularly, we will be concerned with a class of smooth initial data for which $g_1'(\overline\alpha)=0$ and $g_1''(\alpha)\neq0$. Therefore, following an argument similar to that of Theorem \ref{thm:ak>0lambdanegkappapos}, we use a Taylor expansion of $g_1$ around $\overline\alpha$ to derive, for $\eta_*-\eta>0$ small, 
\begin{equation}
\label{eq:lastest1}
\mathcal{\bar{P}}_0(t)\sim
\begin{cases}
C_7(1-N\eta(t))^{\frac{1}{2}-\frac{1}{2\lambda}},\,\,\,\,\,\,\,\,\,\,\,\,\,&\lambda\in(0,1),
\\
-C\ln(\eta_*-\eta),\,\,\,\,\,\,\,\,\,\,\,\,\,&\lambda=1,
\\
C,\,\,\,\,\,\,\,\,\,\,\,\,&\lambda\in(1,+\infty)
\end{cases}
\end{equation}
where
$$C_7=\frac{\Gamma\left(\frac{1}{2\lambda}-\frac{1}{2}\right)}{\Gamma\left(\frac{1}{2\lambda}\right)}\sqrt{\frac{\pi N}{\left|C_1\right|}}\in\mathbb{R}^+,\,\,\,\,\,\,\,\,\,\,\,\,\,C_1=\frac{g_1''(\overline\alpha)}{2}<0.$$
The convergence result (\ref{eq:lastest1})iii) can be obtained via the hypergeometric series argument outlined in Appendix \ref{sec:furtherintegrals}. Further, for the integral term (\ref{eq:integrals})ii), we find that
\begin{equation}
\label{eq:lastest2}
\int_0^1{\frac{\lambda u_0'(\alpha)-\eta(t)C(\alpha)}{\mathcal{Q}(\alpha,t)^{1+\frac{1}{2\lambda}}}d\alpha}\sim C_8(1-N\eta(t))^{-\frac{1}{2}\left(1+\frac{1}{\lambda}\right)}
\end{equation}
for $\lambda\in\mathbb{R}^+$ and $C_8\in\mathbb{R}^+$ given by $$C_8=\frac{\Gamma\left(\frac{1}{2}\left(1+\frac{1}{\lambda}\right)\right)}{\Gamma\left(1+\frac{1}{2\lambda}\right)}\sqrt{\frac{\pi N}{\left|C_1\right|}}.$$

For $\kappa\in\mathbb{R}^+$, suppose $\lambda\in(0,1)$ and set $\alpha=\overline\alpha$ in (\ref{eq:ux}). We obtain
$$u_x(\gamma(\overline\alpha,t),t)\sim\frac{C}{(1-N\eta(t))^{\lambda}}\to+\infty$$
as $\eta\uparrow\eta_*$, whereas, for $\alpha\neq\overline\alpha$,
$$u_x(\gamma(\alpha,t),t)\sim-\frac{(1-\lambda)C}{(1-N\eta(t))^{\lambda}}\to-\infty.$$
Above we use the fact that (\ref{eq:gamma})ii) implies $\frac{C_8}{C_7}=1-\lambda>0$ for $\lambda\in(0,1)$. Moreover, blow-up occurs in finite time due to (\ref{eq:etaivpest1}) and (\ref{eq:lastest1})i), which yield the asymptotic relation
$$t_*-t\sim C\int_{\eta}^{\eta_*}{(\eta_*-\mu)^{\lambda-1}d\mu},\,\,\,\,\,\,\,\,\,\,\,\,\,\,\,\,\,\,\lambda\in(0,1).$$
Then integrating the above implies the existence of a finite $t_*>0$. The above results hold for $\lambda=1$ as well by following a similar argument with (\ref{eq:lastest1})ii) instead and recalling (\ref{eq:limit1}). 

Now let $\lambda\in(1,+\infty)$. Using (\ref{eq:lastest1})iii), (\ref{eq:lastest2}) and (\ref{eq:limit1}) on (\ref{eq:ux}) implies, for $\alpha=\overline\alpha$,
$$u_x(\gamma(\overline\alpha,t),t)\sim\frac{C}{1-N\eta(t)}\to+\infty$$
as $\eta\uparrow\eta_*$, while, if $\alpha\neq\overline\alpha$, (\ref{eq:lastest2}) dominates and
$$u_x(\gamma(\alpha,t),t)\sim-\frac{C}{(1-N\eta(t))^{\frac{1}{2}+\frac{1}{2\lambda}}}\to-\infty.$$
The existence of a finite blow-up time for $u_x$ follows from (\ref{eq:etaivpest1}) and (\ref{eq:lastest1})iii). This establishes the first part of the Theorem for $(\lambda,\kappa)\in\mathbb{R}^+\times\mathbb{R}^+$. Last we study the evolution of (\ref{eq:rho}). As mentioned above $\rho_0(\overline\alpha)\neq0$ due to the single multiplicity of $\eta_*$ and (\ref{eq:D}). Then (\ref{eq:rho}) and (\ref{eq:lastest1}) imply that, for $\alpha=\overline\alpha$,
\begin{equation}
\label{eq:lastrho1}
\rho(\gamma(\overline\alpha,t),t)\sim\frac{\rho_0(\overline\alpha)}{1-N\eta(t)}
\begin{cases}
\left(C_7(1-N\eta(t))^{\frac{1}{2}-\frac{1}{2\lambda}}\right)^{-2\lambda},\,\,\,\,\,\,\,\,\,\,\,\,\,&\lambda\in(0,1),
\\
\left(-C\ln(\eta_*-\eta)\right)^{-2},\,\,\,\,\,\,\,\,\,\,\,\,\,&\lambda=1,
\\
C,\,\,\,\,\,\,\,\,\,\,\,\,&\lambda\in(1,+\infty),
\end{cases}
\end{equation}
so that, as $t$ approaches the finite time $t_*>0$,
\begin{equation}
\label{eq:lastrho2}
\rho(\gamma(\overline\alpha,t),t)\to
\begin{cases}
+\infty,\,\,\,\,\,\,\,\,\,\,\,\,\,\,\,\,\,\,\,&\rho_0(\overline\alpha)>0,
\\
-\infty,\,\,\,\,\,\,\,\,\,\,\,\,\,\,\,\,\,\,\,&\rho_0(\overline\alpha)<0
\end{cases}
\end{equation}
for $(\lambda,\kappa)\in\mathbb{R}^+\times\mathbb{R}^+$. Finally, if $\alpha\neq\overline\alpha$ and $\rho_0(\alpha)=0$, $\rho\circ\gamma\equiv0$ for all time and $(\lambda,\kappa)\in\mathbb{R}^+\times\mathbb{R}^+$, whereas, if $\alpha\neq\overline\alpha$ is such that $\rho_0(\alpha)\neq0$, $\mathcal{Q}$ remains positive for all $0\leq\eta\leq\eta_*$ and (\ref{eq:rho}) yields
\begin{equation}
\label{eq:lastrho3}
\rho(\gamma(\alpha,t),t)\sim C\rho_0(\alpha)
\begin{cases}
C(1-N\eta(t))^{1-\lambda},\,\,\,\,\,\,\,\,\,\,\,\,\,&\lambda\in(0,1),
\\
C(\ln(\eta_*-\eta))^{-2},\,\,\,\,\,\,\,\,\,\,\,\,\,&\lambda=1,
\\
C,\,\,\,\,\,\,\,\,\,\,\,\,&\lambda\in(1,+\infty)
\end{cases}
\end{equation}
for $\eta_*-\eta>0$ small. Consequently, as $t\uparrow t_*$,
\begin{equation}
\label{eq:lastrho4}
\rho(\gamma(\alpha,t),t)\to
\begin{cases}
0,\,\,\,\,\,\,\,\,\,\,\,\,\,&\lambda\in(0,1],
\\
C\rho_0(\alpha),\,\,\,\,\,\,\,\,\,\,\,\,&\lambda\in(1,+\infty).
\end{cases}
\end{equation}
This concludes the proof of the Theorem.
\end{proof}

\section{Examples}
\label{sec:examples}
In this section, we use the \textsc{Mathematica} software to aid in the closed-form computation of time-dependent integral terms and the subsequent plotting of (\ref{eq:ux}) and (\ref{eq:rho}). Examples 1-4 are instances of Theorems \ref{thm:ak>0lambdanegkappapos}, \ref{thm:blow1ak<0}, \ref{thm:singular}, and \ref{thm:global1ak<0}, respectively. We remark that plots depict either (\ref{eq:ux}), or (\ref{eq:rho}), versus the variable $\eta$ instead of $t$.\footnote[9]{With the exception of Figures \ref{fig:ex1} and \ref{fig:ex4} where representations in $t$ were available.} Also, only Figure (\ref{fig:ex4}) uses the Eulerian variable, $x$, as opposed to the Lagrangian coordinate $\alpha$. For practical reasons, specific details of the computations in some examples are omitted. Lastly, and for simplicity in the computations, Example 4 use first component initial data, $u_0$, satisfying Dirichlet boundary conditions instead of periodic. See Remark \ref{rem:dirirem} at the end of \S\ref{sec:solution}.

\subsection{Example 1}

For $(\lambda,\kappa)=(-1/2,-1)$ take $u_0'(\alpha)=\cos(2\pi\alpha)$ and $\rho_0(\alpha)\equiv1/2$. Then $C(\alpha)=0$ implies that $\Omega=\{1/8,3/8,5/8,7/8\}$, so that $M=\max_{\Omega}\{2\lambda u_0'(\alpha)\}=\sqrt{2}/2$. Now 
$$g_1(\alpha)=\frac{1}{2}\left(\frac{1}{\sqrt{2}}-\cos(2\pi\alpha)\right)$$ 
with $N_1=g_1(1/2)$. Since $N_1>M$, we conclude that $\mathcal{Q}$ will vanish earliest when $\overline\alpha=1/2$ as 
$$\eta\uparrow\eta_*=\frac{2\sqrt{2}}{1+\sqrt{2}}\sim1.17.$$ 
Evaluating the integrals in (\ref{eq:integrals}) now yield
\begin{equation}
\label{eq:intex1}
\mathcal{\bar{P}}_0(t)\equiv1,\,\,\,\,\,\,\,\,\,\,\,\,\,\,\,\,\,\,\,\,\,\,\,\,\,\,\,\int_0^1{\frac{\lambda u_0'(\alpha)-\eta(t)C(\alpha)}{\mathcal{Q}(\alpha,t)^{1+\frac{1}{2\lambda}}}d\alpha}\equiv0,
\end{equation}
so that (\ref{eq:ux}) and (\ref{eq:rho}) become
\begin{equation}
\label{eq:ex1ux}
u_x(\gamma(\alpha,t),t)=\frac{8\cos(2\pi\alpha)+2t\cos(4\pi\alpha)}{8+8t\cos(2\pi\alpha)+t^2\cos(4\pi\alpha)}
\end{equation}
and
\begin{equation}
\label{eq:ex1p}
\rho(\gamma(\alpha,t),t)=\frac{4}{8+8t\cos(2\pi\alpha)+t^2\cos(4\pi\alpha)}.
\end{equation}
In the above we used $\eta(t)=t$, a consequence of (\ref{eq:etaivp}) and (\ref{eq:intex1})i), and which yields 
$$t_*=\eta_*=\frac{2\sqrt{2}}{1+\sqrt{2}}.$$
Letting $\alpha=1/2$ in (\ref{eq:ex1ux}), we see that 
$$u_x(\gamma(1/2,t),t)\to-\infty$$
but remains bounded for $\alpha\neq1/2$ and $0\leq t\leq t_*$. Similarly, since $\rho_0(\alpha)\equiv1/2>0$, 
$$\rho(\gamma(1/2,t),t)\to+\infty,$$
while it stays finite otherwise. See Figure \ref{fig:ex1} below.

\begin{center}
\begin{figure}[!ht]
\includegraphics[scale=0.33]{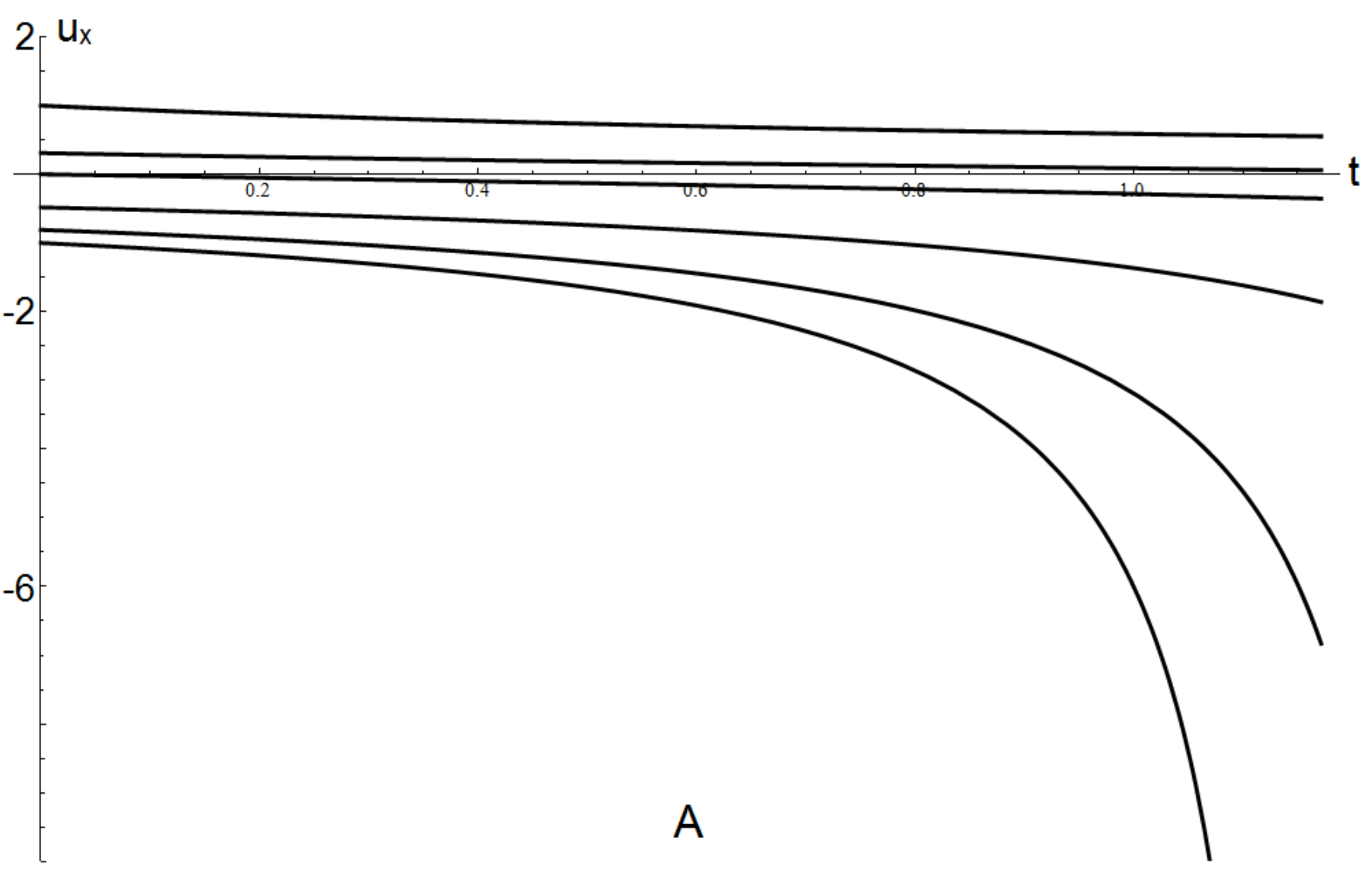} 
\includegraphics[scale=0.30]{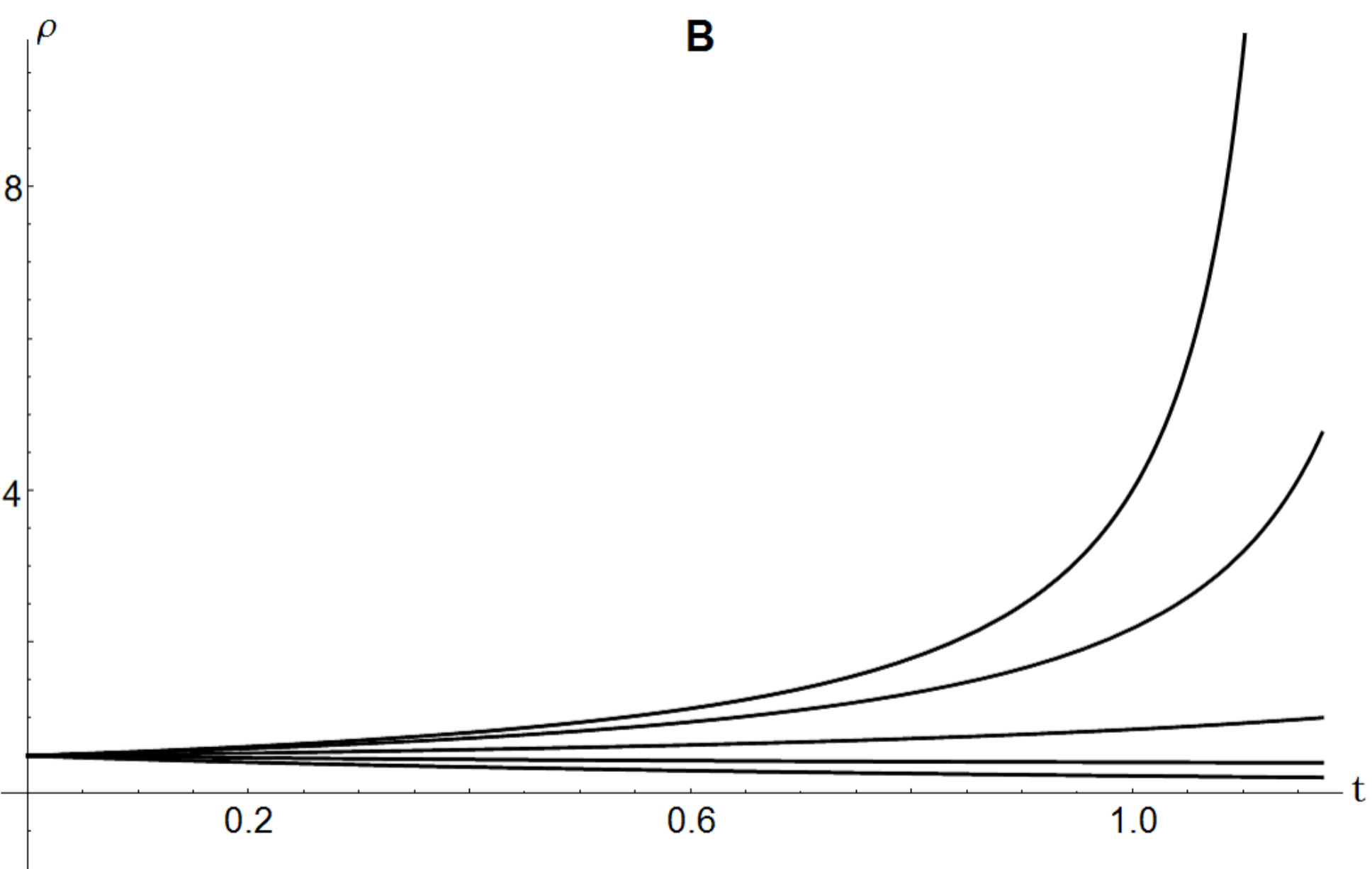} 
\caption{For Example 1 with $(\lambda,\kappa)=(-1/2,-1)$, Figures A and B represent blow-up of (\ref{eq:ex1ux}) and (\ref{eq:ex1p}) for $\alpha=1/2$ to $-\infty$ and respectively $+\infty$ as $t\uparrow t_*\sim1.17$. In contrast, if $\alpha\neq1/2$, both remain finite for $0\leq t\leq t_*$.}
\label{fig:ex1}
\end{figure}
\end{center}

\subsection{Example 2}

For $(\lambda,\kappa)=(-1,1)$, take $u_0'(\alpha)=\cos(2\pi\alpha)$ and $\rho_0(\alpha)=\sin(2\pi\alpha)$. We follow the criteria in Theorem \ref{thm:blow1ak<0}. Note that $\rho_0=0$ for $\alpha\in\{0,1/2,1\}$, but, of those points, $\lambda u_0'=-u_0'>0$ only at $\alpha=1/2$. Therefore $\alpha_1=1/2$ and $\eta_*=1$. Evaluation of the integrals in (\ref{eq:integrals}) yield solutions in terms of Elliptic integrals, while (\ref{eq:t*}), in the limit as $\eta\uparrow 1$, gives $t_*\sim0.86$. We find that $\rho$ stays bounded for all $\alpha\in[0,1]$ and $0\leq t\leq t_*$, whereas
$$u_x(\gamma(1/2,t),t)\to-\infty$$
as $t\uparrow t_*$, but remains finite otherwise. See Figure \ref{fig:ex2} below.     

\begin{center}
\begin{figure}[!ht]
\includegraphics[scale=0.33]{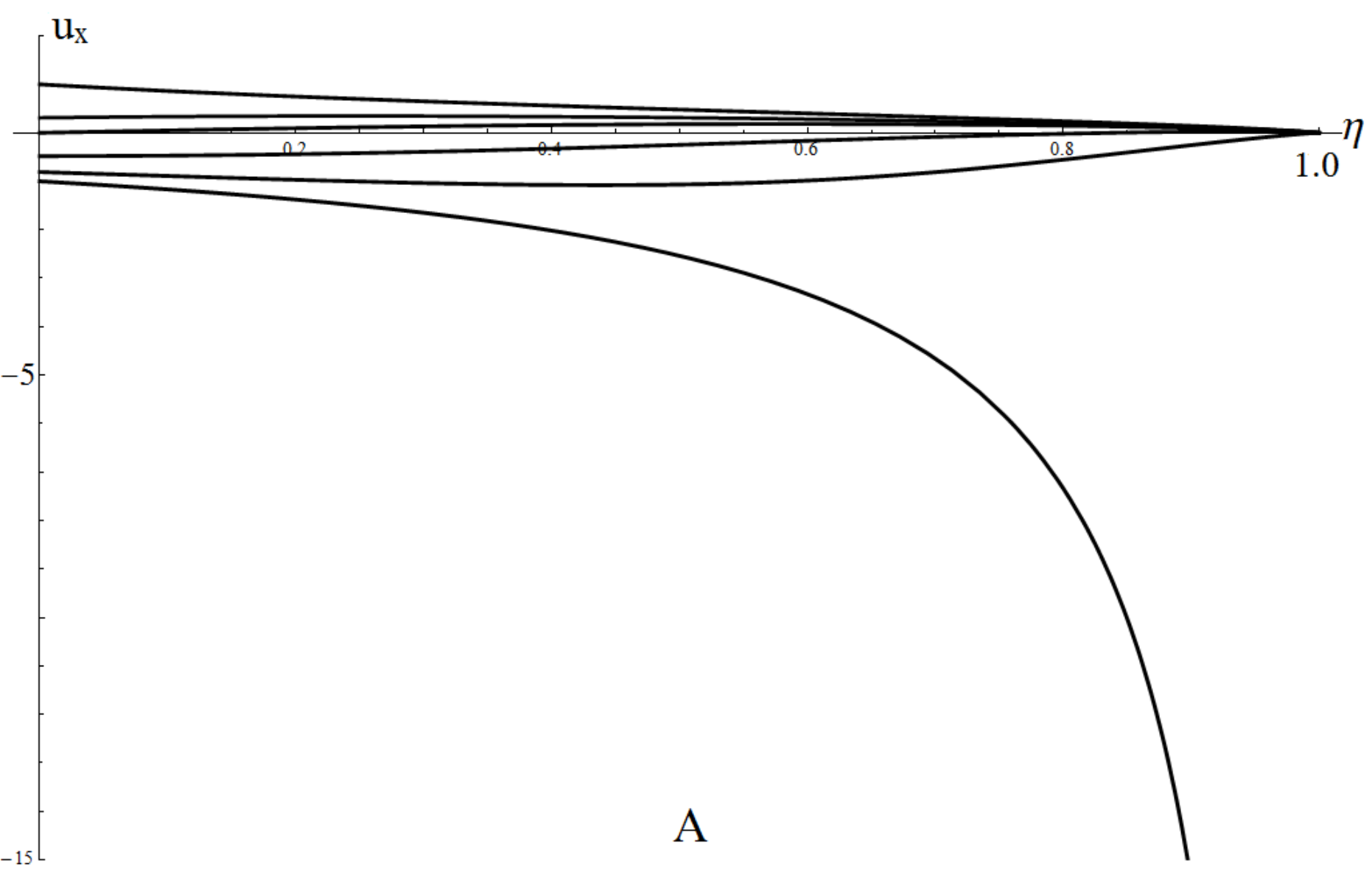} 
\includegraphics[scale=0.34]{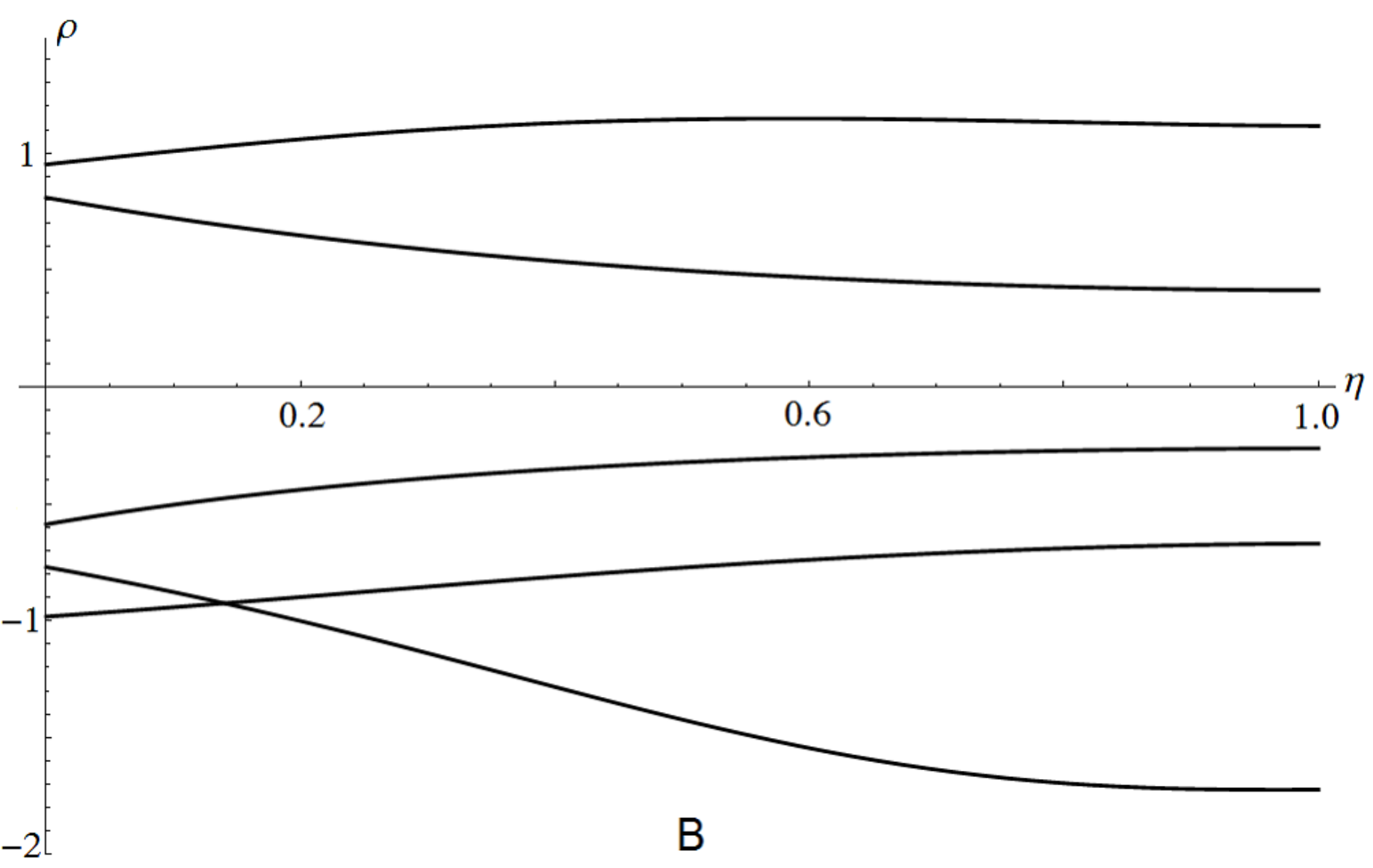} 
\caption{For Example 2 with $(\lambda,\kappa)=(-1,1)$, Figure A depicts blow-up of $u_x(\gamma(1/2,t),t)$ to $-\infty$ as $t\uparrow t_*\sim0.86$. Figure B shows a bounded $\rho\circ\gamma$ for several choices of $\alpha$ and $0\leq t\leq t_*$.}
\label{fig:ex2}
\end{figure}
\end{center}

\subsection{Example 3}

For $(\lambda,\kappa)=(1,1)$, take $u_0'(\alpha)=\cos(2\pi\alpha)$ and $\rho_0(\alpha)\equiv1$, which corresponds to Theorem \ref{thm:singular}. Notice that $C(\alpha)=0$ gives $\Omega=\{0,1/2,1\}$ so that $\max_{\Omega}\{2\lambda u_0'\}=2$ occurs at both end-points $\alpha=0,1$. This implies that (\ref{eq:etaomega}) could be given by $1/2$. Now, $\rho_0$ never vanishes, so no double roots; however, since $g_1(\alpha)=\cos(2\pi\alpha)+1$, then $g_1$ attains its maximum value at $\alpha=0,1$, both of which lie in $\Omega$, thus, $N$ does not exist. We conclude that $\eta_*=1/2$ and $\overline\alpha=0,1$. Evaluating the integrals yield representations in terms of Elliptic integrals,
$$\mathcal{\bar{P}}_0(t)=\frac{2\,\text{EllipticK}\left[\frac{4\eta(t)^2}{4\eta(t)^2-1}\right]}{\pi\,\sqrt{1-4\eta(t)^2}}.$$
For practical purposes, we omit the formula for (\ref{eq:integrals})ii). Using (\ref{eq:ux}) and (\ref{eq:rho}), we find that $u_x(\gamma(\overline\alpha,t),t)\to+\infty$ as $t\uparrow t_*\sim0.4$, and diverges to $+\infty$ otherwise. Moreover, $\rho(\gamma(\overline\alpha,t),t)\to+\infty$ but vanishes, as $t\uparrow t_*$, for $\alpha\neq\overline\alpha$. See Figure \ref{fig:ex3} below.

\begin{center}
\begin{figure}[!ht]
\includegraphics[scale=0.31]{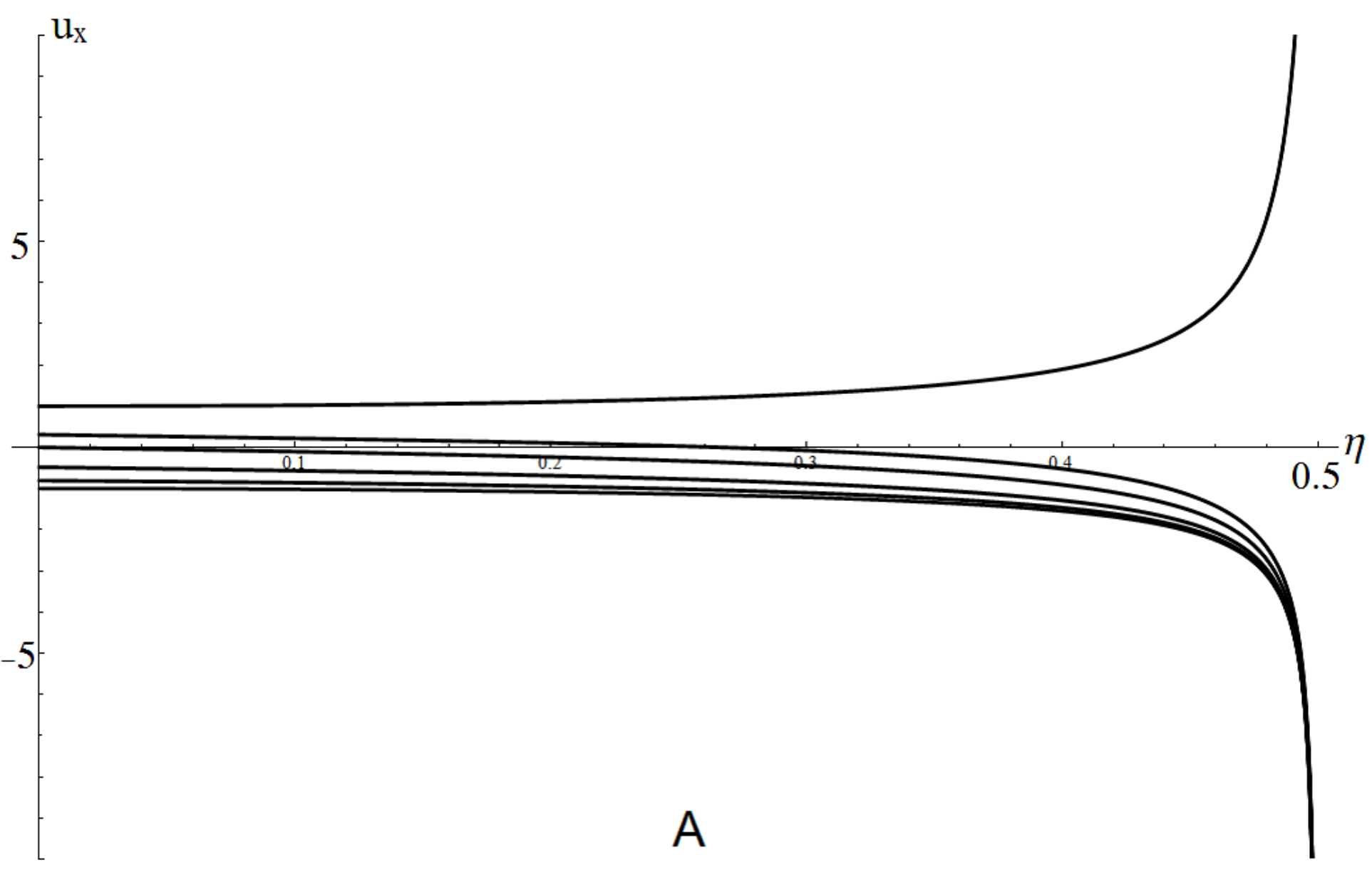} 
\includegraphics[scale=0.32]{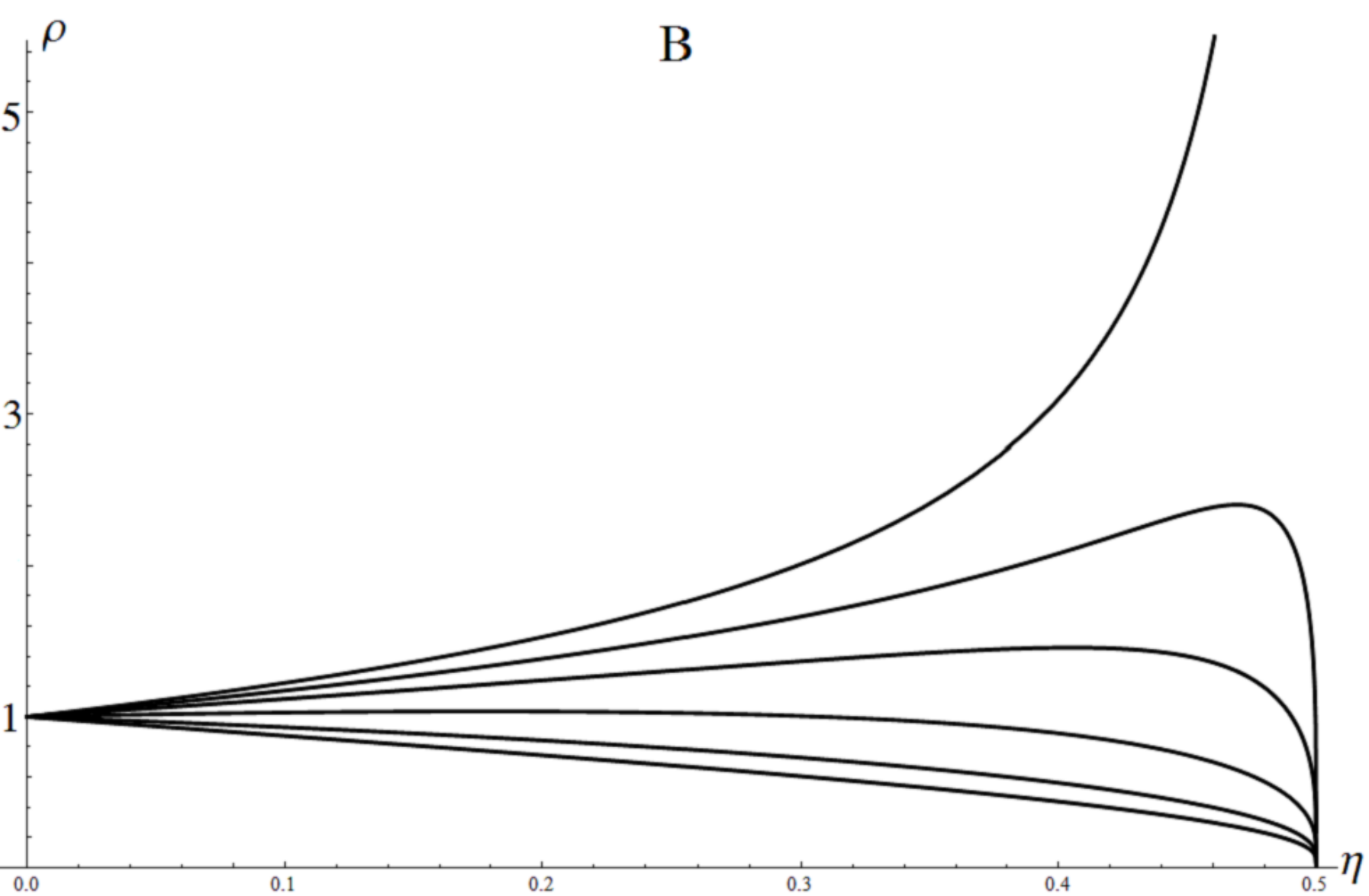} 
\caption{In Example 3 for $(\lambda,\kappa)=(1,1)$, Figure A shows $u_x\circ\gamma$ undergoing a two-sided, everywhere blow-up as $t\uparrow t_*\sim0.4$, whereas, in Figure B, $\rho\circ\gamma$ diverges, to $+\infty$, only at two locations in the domain and vanishes everywhere else.}
\label{fig:ex3}
\end{figure}
\end{center}

\subsection{Example 4}

For $(\lambda,\kappa)=(-1/2,1)$ take $u_0'(\alpha)=1-2\alpha$ and $\rho_0(\alpha)\equiv1$. In this case the discriminant (\ref{eq:D}) of the quadratic is always negative and so $\mathcal{Q}$ remains positive and finite for all $0\leq\eta<+\infty$. We are then interested in the behaviour of solutions as $\eta\to+\infty$. Evaluating the integrals in (\ref{eq:integrals}) yields
$$\mathcal{\bar{P}}_0(t)=1+\frac{7\eta(t)^2}{12},\,\,\,\,\,\,\,\,\,\,\,\,\,\,\,\,\,\,\,\,\,\int_0^1{\frac{\lambda u_0'(\alpha)-\eta(t)C(\alpha)}{\mathcal{Q}(\alpha,t)^{1+\frac{1}{2\lambda}}}d\alpha}=-\frac{7\eta(t)}{12}.$$
Then using the above to solve (\ref{eq:etaivp}) gives
$$t(\eta)=\sqrt{\frac{12}{7}}\,\text{arctan}\left(\sqrt{\frac{7}{12}}\,\eta\right),$$
so that $t_{\infty}\equiv\lim_{\eta\to+\infty}t(\eta)=\frac{\pi}{2}\sqrt{\frac{12}{7}}\sim2.06$ represents the finite-time it take solutions to reach steady states. The solution components are computed as
$$u_x\circ\gamma=\frac{12-24\alpha+4\eta(t)(1+6\alpha(\alpha-1))+7(2\alpha-1)\eta(t)^2}{3(4+\eta(t)(4+3\eta(t)+4\alpha(\eta(t)(\alpha-1)-2)))}$$
and 
$$\rho\circ\gamma=\frac{12+7\eta(t)^2}{3(4+\eta(t)(4+3\eta(t)+4\alpha(\eta(t)(\alpha-1)-2)))}.$$
We note that we are also able to obtain the inverse jacobian function in closed-form, which is then used to plot the solution in Eulerian coordinates. Lastly, the steady states, $u_x^{\infty}$ and $\rho^{\infty}$, are given by
$$\rho^{\infty}=\frac{7}{3(4\alpha^2-4\alpha+3)},\,\,\,\,\,\,\,\,\,\,\,\,\,\,\,\,\,\,\,\,\,\,\,\,\,\,\,\,u_x^{\infty}=-u_0'(\alpha)\rho^{\infty}.$$
Additionally, and for the sake of comparison, below we plot (\ref{eq:ux}) and (\ref{eq:rho}) for the same initial data but $(\lambda,\kappa)=(1/2,-1)$. The steady states in that case are
$$\rho^{\infty}=\frac{\sqrt{2}}{(4\alpha^2-4\alpha+3)\,\text{arcCot}\sqrt{2}},\,\,\,\,\,\,\,\,\,\,\,\,\,\,\,\,\,\,\,\,\,\,\,u_x^{\infty}=-u_0'(\alpha)\rho^{\infty},$$
which are reached at, approximately, $t_{\infty}\sim2.22$. It is simple to check that both steady states coincide with the formulae in Theorem \ref{thm:global1ak<0}. See Figures \ref{fig:ex4} and \ref{fig:ex41} below.

\begin{center}
\begin{figure}[!ht]
\includegraphics[scale=0.35]{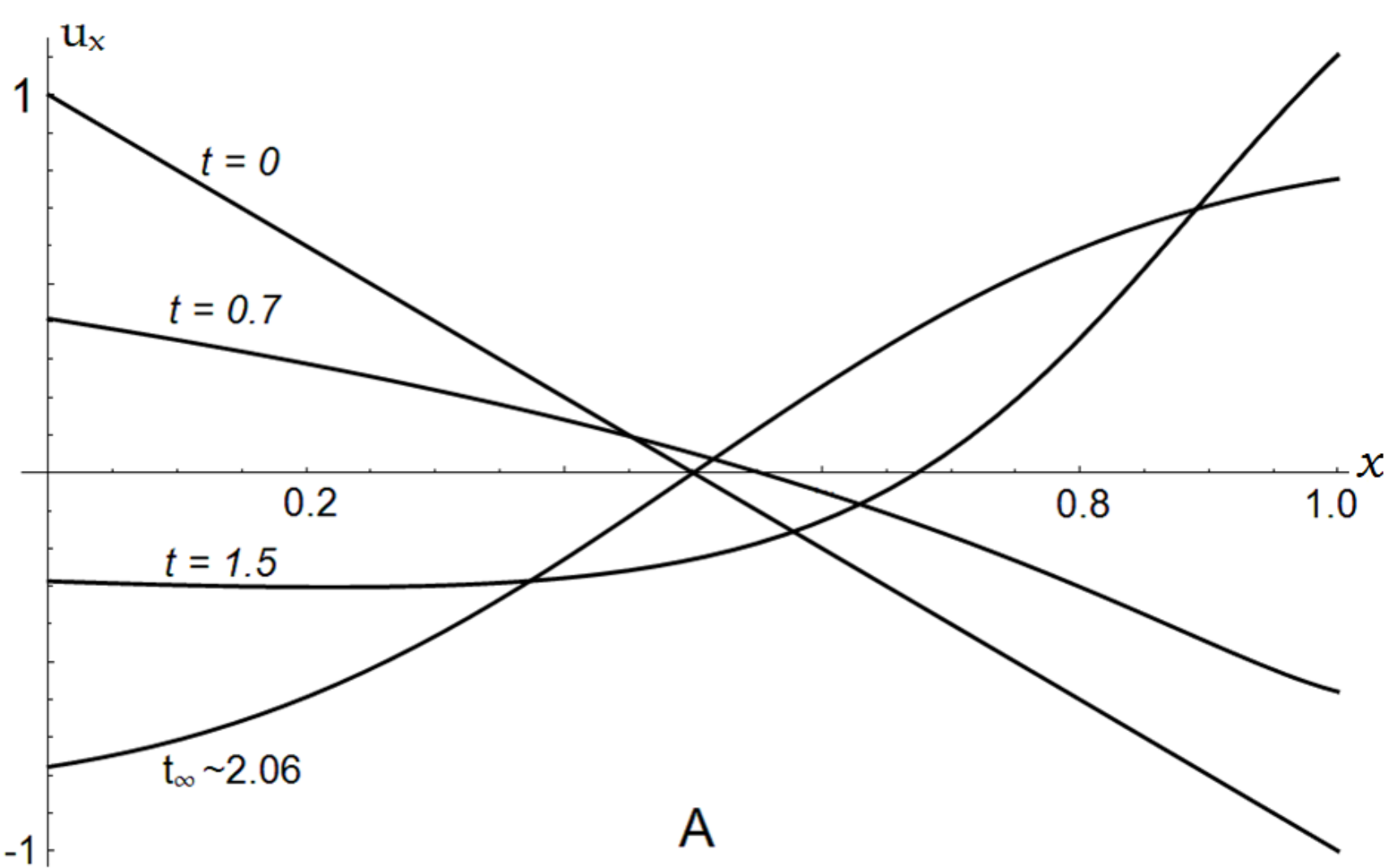} 
\includegraphics[scale=0.35]{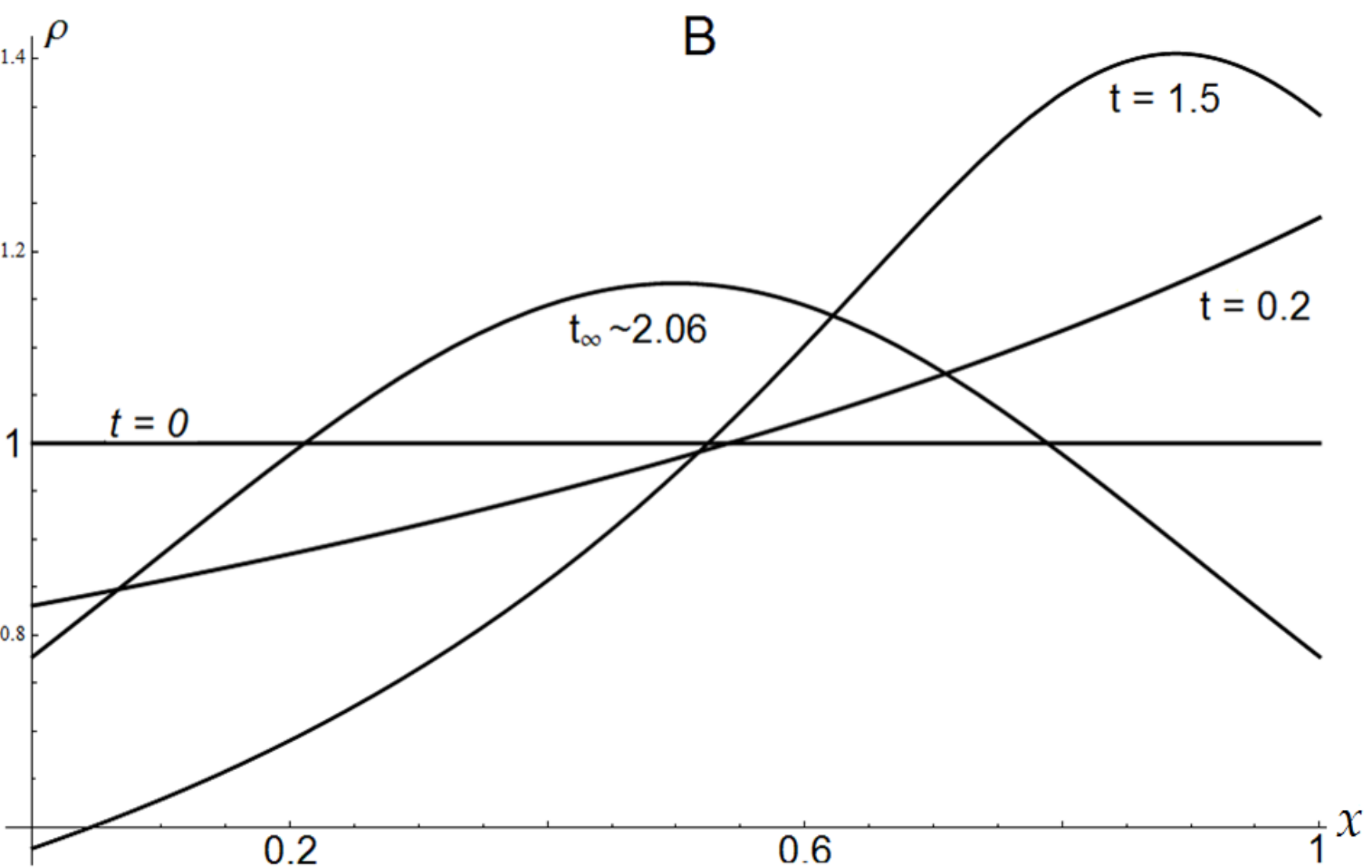} 
\caption{For Example 4 with $(\lambda,\kappa)=(-1/2,1)$, convergence of $u_x(x,t)$ and $\rho(x,t)$ to steady states as $\eta\to+\infty$ ($t\uparrow t_{\infty}\sim2.06$).}
\label{fig:ex4}
\end{figure}
\end{center}

\begin{center}
\begin{figure}[!ht]
\includegraphics[scale=0.32]{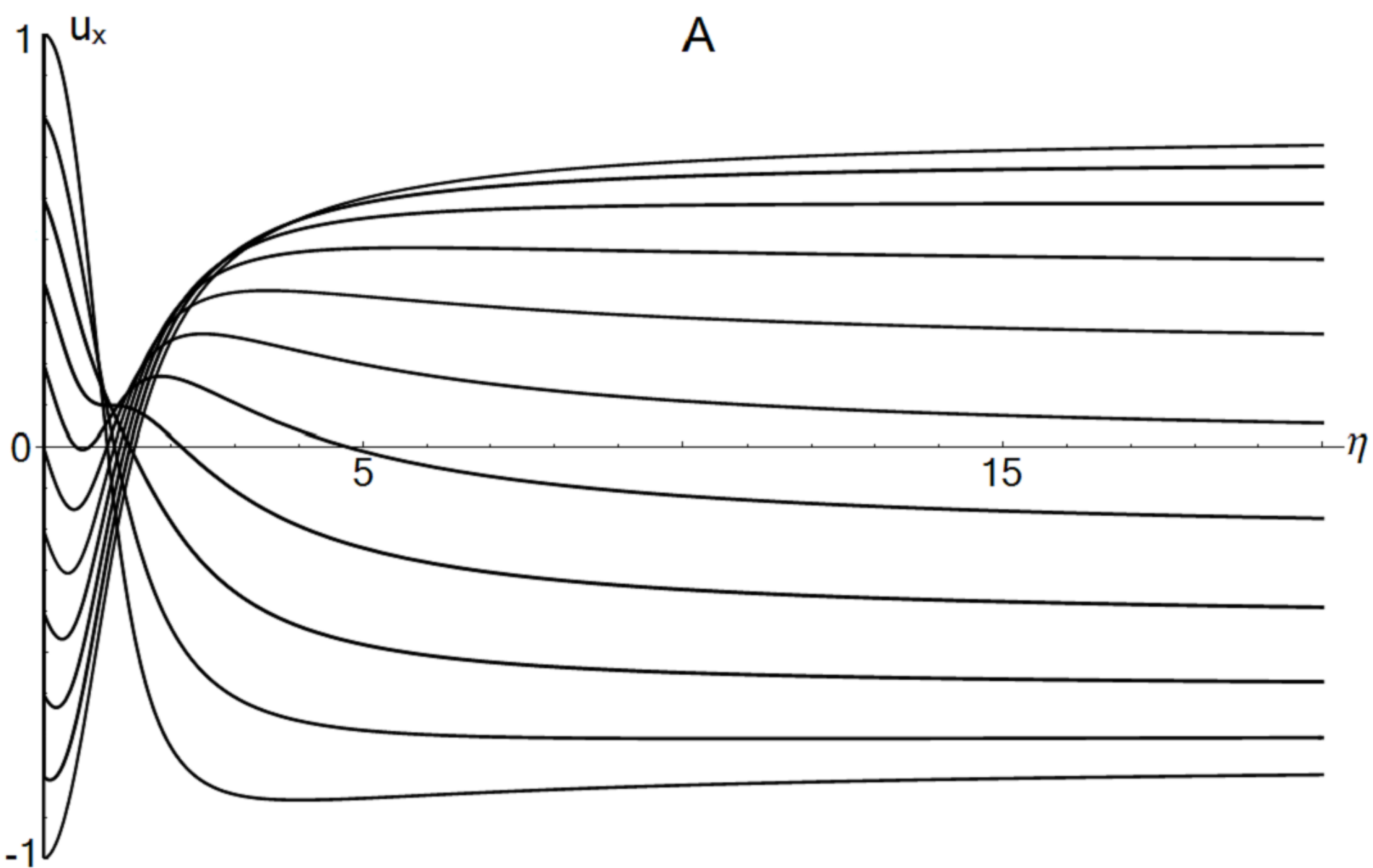} 
\includegraphics[scale=0.32]{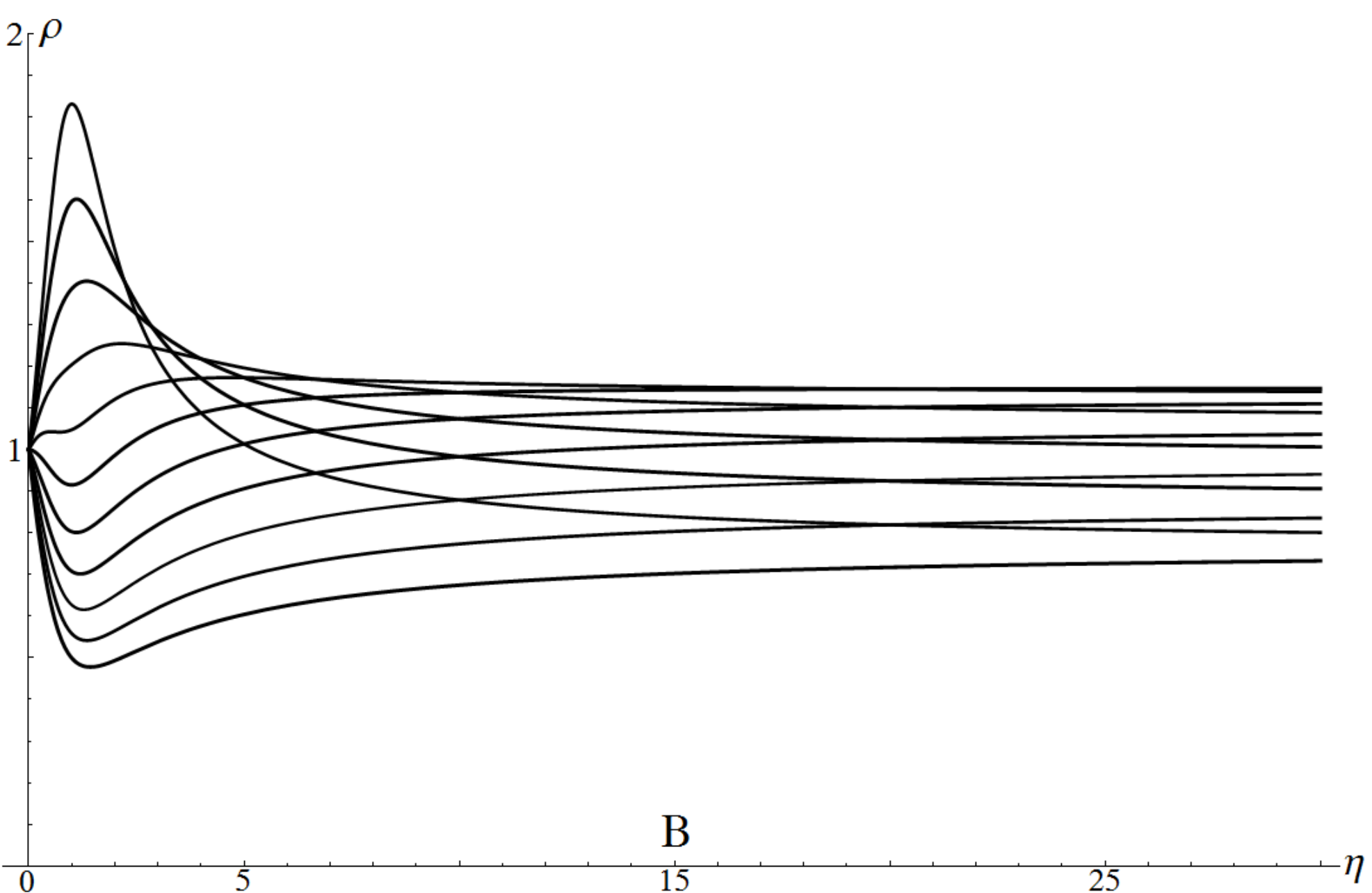} 
\caption{In Example 4 for $(\lambda,\kappa)=(1/2,-1)$, convergence of $u_x\circ\gamma$ and $\rho\circ\gamma$ to steady states as $\eta\to+\infty$ ($t\uparrow t_{\infty}\sim2.22$).}
\label{fig:ex41}
\end{figure}
\end{center}


\appendix

\section{Some Simple Cases}\hfill
\label{sec:simplecases}

In this section, we consider some special and trivial cases not studied in the paper.

\subsection{Case $\kappa=0$ and $C(\alpha)\not\equiv0$}
\label{subsec:k=0}

For $\kappa=0$, (\ref{eq:hs})i) reduces to the giPJ equation, for which a general solution was derived, and analyzed extensively, in \cite{Sarria1} and  \cite{Sarria2}. The reader may also refer to \cite{Okamoto1}, \cite{Okamoto2}, \cite{Aconstantin1}, \cite{Saxton1}, \cite{Wunsch1}, \cite{Wunsch2} and \cite{Wunsch3} for additional results. Excluding the trivial case $\rho_0\equiv0$, regularity of $\rho$ is easily studied in this case by using (\ref{eq:rho0}) and the estimates established in \cite{Sarria1} and \cite{Sarria2} for the corresponding jacobian function.

\subsection{Case $\kappa=0$ and $C(\alpha)\equiv0$}
\label{subsec:k=0,c=0}

Suppose $\kappa=0$ and $C(\alpha)\equiv0$ for $C(\alpha)$ as in (\ref{eq:c0}). Assuming $\lambda\neq0$, this implies that $u_0'(\alpha)\equiv0$. Then the formulas in \S\ref{sec:solution} yield $(u_x,\rho)\circ\gamma(\alpha,t)=(0,\rho_0(\alpha)).$

\subsection{Case $\kappa\neq0$ and $C(\alpha)\equiv0$}
\label{subsec:c=0}

For $(\lambda,\kappa)\neq(0,0)$, suppose $C(\alpha)\equiv0$. Then for $\lambda\kappa>0$, the initial data $(u_0',\rho_0)$ will be restricted by the condition
\begin{equation}
\label{eq:c=0}
u_0'(\alpha)^2=\frac{\kappa}{\lambda}\rho_0(\alpha)^2,\,\,\,\,\,\,\,\,\,\,\,\,\,\,\,\,\alpha\in[0,1].
\end{equation}
Moreover, setting $C(\alpha)\equiv0$ into (\ref{eq:p})-(\ref{eq:etaivp}) leads to representation formulae derived in \cite{Sarria1}\footnote[10]{See equation (3.19) in \cite{Sarria1}.}. If instead $\lambda\kappa<0$, then (\ref{eq:c=0}) implies that $(u_0'(\alpha),\rho_0(\alpha))\equiv(0,0)$, the trivial solution $(u_x,\rho)=(0,0)$.

\subsection{Case $\lambda=0$}
\label{subsec:lambda}
For smooth enough initial data, this case may be treated by following the argument in Appendix A of \cite{Sarria1}, and leads to global-in-time solutions.

\section{Further Integral Estimates}\hfill
\label{sec:furtherintegrals}

In this section, we give a brief outline of the method used to obtain several of the integral estimates used throughout the paper for (\ref{eq:integrals})i), particularly those leading to a convergent integral as $\eta\uparrow\eta_*$ for certain values of $\lambda$. First we need some auxiliary results. Recall the Gauss hypergeometric series (\cite{Barnes1}, \cite{Magnus1}, \cite{Gasper1})
\begin{equation}
\label{eq:2f1}
\begin{split}
{}_2F_1\left[a,b;c;z\right]\equiv\sum_{k=0}^{\infty}\frac{\left(a\right)_k(b)_k}{\left(c\right)_k\,k!}z^k,\,\,\,\,\,\,\,\,\,\,\,\,\,\,\,\lvert z\rvert< 1
\end{split}
\end{equation} 
defined for $c\notin\mathbb{Z}^-\cup\{0\}$ and $(x)_k,\, k\in\mathbb{N}\cup\{0\}$, the Pochhammer symbol $(x)_0=1$, $(x)_k=x(x+1)...(x+k-1).$ Then consider the following results (\cite{Magnus1}, \cite{Gasper1}):
\begin{proposition}
\label{lem:analcont}
Suppose $\lvert\text{arg}\left(-z\right)\rvert<\pi$ and $a,b,c,a-b\notin\mathbb{Z}.$ Then, the analytic continuation for $\lvert z\rvert>1$ of the series (\ref{eq:2f1}) is given by 
\begin{equation}
\label{eq:analform}
\begin{split}
{}_2F_1[a,b;c;z]=&\frac{\Gamma(c)\Gamma(a-b)(-z)^{-b}{}_2F_1[b,1+b-c;1+b-a;z^{-1}]}{\Gamma(a)\Gamma(c-b)}
\\
&+\frac{\Gamma(c)\Gamma(b-a)(-z)^{-a}{}_2F_1[a,1+a-c;1+a-b;z^{-1}]}{\Gamma(b)\Gamma(c-a)}
\end{split}
\end{equation} 
where $\Gamma(\cdot)$ denotes the standard gamma function.
\end{proposition}
\begin{lemma}
\label{prop:prop}
\cite{Sarria1}\,
Suppose $b\in(-\infty,2)\backslash\{1/2\},\,\,0\leq\left|\beta-\beta_0\right|\leq1$ and $\epsilon\geq C_0$ for some $C_0>0.$ Then 
\begin{equation}
\label{eq:derseries}
\begin{split}
\frac{1}{\epsilon^b}\,\frac{d}{d\beta}\left((\beta-\beta_0){}_2F_1\left[\frac{1}{2},b;\frac{3}{2};-\frac{C_0(\beta-\beta_0)^2}{\epsilon}\right]\right)=(\epsilon+C_0(\beta-\beta_0)^2)^{-b}.
\end{split}
\end{equation} 
\end{lemma}
Below we give a simple example on how to use the above results to estimate the behaviour, as $\eta\uparrow\eta_*$, of the integral
$$\int_0^1{\frac{d\alpha}{(1-\lambda\eta(t)u_0'(\alpha))^{\frac{1}{\lambda}}}}$$
for $\lambda\in(1,+\infty)\backslash\{2\}$ and
\begin{equation}
\label{eq:etaah}
\eta_*=\frac{1}{\lambda M_0}
\end{equation} 
where $M_0>0$ denotes the largest value attained by $u_0'(\alpha)$ at finitely many points $\overline\alpha\in[0,1]$. For $\epsilon>0$ small, we start with the approximation
\begin{equation}
\label{eq:exp}
\begin{split}
\epsilon-u_0^\prime(\alpha)+M_0\sim\epsilon-C_1(\alpha-\overline\alpha)^2,
\end{split}
\end{equation}
which originates from a Taylor expansion, with non-vanishing quadratic coefficient, of $u_0'$ about $\overline\alpha$. Now, suppose $\lambda\in(1,+\infty)\backslash\{2\}$ and set $b=\frac{1}{\lambda}$ in Lemma \ref{prop:prop} to obtain
\begin{equation}
\label{eq:gen2f12}
\begin{split}
\int_{\overline\alpha-s}^{\overline\alpha+s}{\frac{d\alpha}{(\epsilon-C_1(\alpha-\overline\alpha)^2)^{\frac{1}{\lambda}}}}=2s\epsilon^{-\frac{1}{\lambda}}\,{}_2F_1\left[\frac{1}{2},\frac{1}{\lambda};\frac{3}{2};\frac{s^2C_1}{\epsilon}\right],
\end{split} 
\end{equation}
where the above series is defined by (\ref{eq:2f1}) as long as $\epsilon\geq-C_1\geq-s^2C_1>0$, namely $-1\leq\frac{s^2C_1}{\epsilon}<0$. However, ultimately, we are interested in the behaviour of (\ref{eq:gen2f12}) for $\epsilon>0$ arbitrarily small, so that, eventually, $\frac{s^2C_1}{\epsilon}<-1$. To achieve the transition of the series' argument across $-1$ in a well-defined, continuous fashion, we use proposition \ref{lem:analcont} which provides us with the analytic continuation of (\ref{eq:gen2f12}) from argument values inside the unit circle, particularly on the interval  $-1\leq\frac{s^2C_1}{\epsilon}<0$, to those found outside and thus for $\frac{s^2C_1}{\epsilon}<-1$. Consequently, for $\epsilon$ small enough, so that $-s^2C_1>\epsilon>0$, proposition \ref{lem:analcont} implies
\begin{equation}
\label{eq:gen2f122}
\begin{split}
2s\epsilon^{-\frac{1}{\lambda}}\,{}_2F_1\left[\frac{1}{2},\frac{1}{\lambda};\frac{3}{2};\frac{s^2C_1}{\epsilon}\right]= C\,\Gamma\left(\frac{1}{\lambda}-\frac{1}{2}\right)\epsilon^{\frac{1}{2}-\frac{1}{\lambda}}+\frac{C}{\lambda-2}+\psi(\epsilon)
\end{split}
\end{equation}
for $\psi(\epsilon)=\textsl{o}(1)$ as $\epsilon\to0$ and $C\in\mathbb{R}^+$ which may depend on $\lambda$ and can be obtained explicitly from (\ref{eq:analform}). Then, substituting $\epsilon=\frac{1}{\lambda\eta}-M_0$ into (\ref{eq:gen2f122}) and using (\ref{eq:exp}) along with (\ref{eq:gen2f12}), yields
\begin{equation}
\label{eq:alphaposgen6}
\begin{split}
\int_0^1{\frac{d\alpha}{(1-\lambda\eta(t)u_0'(\alpha))^{\frac{1}{\lambda}}}}\sim
\begin{cases}
C(1-\lambda\eta(t)M_0)^{\frac{1}{2}-\frac{1}{\lambda}},\,\,\,\,\,\,\,\,\,\,\,&\lambda\in(1,2),
\\
C,\,\,&\lambda\in(2,+\infty)
\end{cases}
\end{split}
\end{equation}
for $\eta_*-\eta>0$ small. We remark that the above blow-up rate for $\lambda\in(1,2)$ could also have been obtained for the whole interval $\lambda\in(0,2)$ via the simpler method used in the proofs of most Theorems in this article.

\section{Roots with double multiplicity}\hfill
\label{sec:doublemult}

In this section, we present blow-up and global existence results for solutions to (\ref{eq:hs}) in the case where the quadratic (\ref{eq:Q}) vanishes earliest at $\eta_*$ having double multiplicity. As shown in \S\ref{subsec:notation}, in such case the asymptotic analysis of the space-dependent terms in (\ref{eq:ux}) and (\ref{eq:rho}), as well as that of the integrals (\ref{eq:integrals}), follows that of \cite{Sarria1} and \cite{Sarria2}. Consequently, we direct the reader to those works for specific details on the corresponding estimates.

Suppose the initial data is such that $\mathcal{Q}$ vanishes earliest at 
$$\eta_*=\frac{1}{N_1}$$
where
\begin{equation}
\label{eq:N2}
N_1\equiv\max_{\alpha\in\Sigma}\{\lambda u_0'(\alpha)\}=\lambda u_0'(\overline\alpha)
\end{equation}
is attained at finitely many points $\overline\alpha\in[0,1]$. In other words, we are assuming that the least, positive zero of the quadratic has multiplicity two (see (\ref{eq:fac2})). In turn, this means that the discriminant of $\mathcal{Q}$ satisfies $\mathcal{D}(\overline\alpha)=0$, namely $\rho_0(\overline\alpha)=0$. By comparing $N_1$ with the maximum value over $\Omega$ of $2\lambda u_0'$ or, if it exists, the maximum over $\Sigma$ of $g_1$ in (\ref{eq:g})i), both of which lead to single roots, we see that for most choices of smooth, non-trivial initial data, $\eta_*$ will be a zero of multiplicity one; however, double roots may occur in some trivial cases as well as for (non-)smooth data with particular growth conditions near locations where $\rho$ vanishes. For instance, as mentioned before, for the trivial case $\rho_0(\alpha)\equiv0$ (the giPJ equation case), $\eta_*$ is always a double root of $\mathcal{Q}$. For the non-trivial case, consider for instance piecewise continuous data
\begin{equation}
\label{eq:pieceu0}
u_0'(\alpha)\equiv
\begin{cases}
\frac{1}{2},\,\,\,\,\,\,\,\,\,\,\,\,&\alpha\in[0,1/4),
\\
1,\,\,\,&\alpha\in[1/4,3/4],
\\
-\frac{5}{2},\,\,\,\,\,\,&\alpha\in(3/4,1]
\end{cases}
\end{equation}
and
\begin{equation}
\label{eq:piecep0}
\rho_0(\alpha)\equiv
\begin{cases}
-\frac{1}{4},\,\,\,\,\,\,\,\,\,\,\,\,&\alpha\in[0,1/4),
\\
0,\,\,\,&\alpha\in[1/4,3/4],
\\
-\frac{1}{4},\,\,\,\,\,\,&\alpha\in(3/4,1].
\end{cases}
\end{equation}
In this case $\Omega=\emptyset$ and
\begin{equation}
\label{eq:pieceg1}
g_1(\alpha)=
\begin{cases}
\frac{3}{4},\,\,\,\,\,\,\,\,\,\,\,\,&\alpha\in[0,1/4),
\\
1,\,\,\,&\alpha\in[1/4,3/4],
\\
-\frac{9}{4},\,\,\,\,\,\,&\alpha\in(3/4,1].
\end{cases}
\end{equation}
Therefore, if $\alpha\in[1/4,3/4]$ we have that $\mathcal{Q}(\alpha,t)=(1-\eta(t))^2$, which vanishes as $\eta\uparrow\eta_*=1$, a root of double multiplicity. For $\alpha\in[0,1]\backslash[1/4,3/4]$, $\mathcal{Q}$ has single roots that are either negative, or larger than $\eta_*$. Notice that for the above non-smooth initial data formulas (\ref{eq:ux}) and (\ref{eq:rho}) are still defined. Although we are not concerned with non-smooth data in this work, the above may serve as a prototype on how double roots, with $\rho_0\not\equiv0$, can occur. We claim that for smooth data a single root is most common because, if we are to compare the greatest values of $g_1=\lambda u_0'+\sqrt{\lambda\kappa}\left|\rho_0\right|$ and $\lambda u_0'$ over $\Sigma$ for $\lambda\kappa>0$, then for the former to be less that than the latter, assumptions are needed on how steep $u_0'$ and $\left|\rho_0\right|$ decrease and respectively increase near zeroes of $\rho_0$ in $\Sigma$, and that is assuming $u_0'$ attains its greatest value there. If not, then the maximum of $g_1$ would be greater. Clearly, in the other case the maximum of $2\lambda u_0'$ is greater than that of $\lambda u_0$.

As opposed to Theorem \ref{thm:ak>0lambdanegkappapos}, which deals with an earliest root $\eta_*$ of single multiplicity, Corollary \ref{coro:special1} below considers a double multiplicity root for parameters $(\lambda,\kappa)\in\mathbb{R}^-\times\mathbb{R}^-$. We find that regularity results in both cases are rather similar, with the only difference being a scaling between the $\lambda$ values in both results. In contrast, Corollary \ref{coro:singular2} represents the case where, for $(\lambda,\kappa)\in\mathbb{R}^+\times\mathbb{R}^+$, $\mathcal{Q}$ has an earliest root of multiplicity two. As opposed to the single multiplicity case in Theorem \ref{thm:singular}, we find that double multiplicity in $\eta_*$ now allows for global solutions in certain ranges of the parameter $\lambda$.

\begin{corollary}
\label{coro:special1}
Consider the initial boundary value problem (\ref{eq:hs})-(\ref{eq:pbc}) for $(\lambda,\kappa)\in\mathbb{R}^-\times\mathbb{R}^-$. Suppose the initial data is such that (\ref{eq:Q}) vanishes earliest at $\eta_*=\frac{1}{\lambda m_0}$, a zero of double multiplicity, where $m_0<0$ represents the least value achieved by $u_0'$ in $[0,1]$ at finitely many points $\overline\alpha$. Then
\begin{enumerate}
\item\label{it:1ak<0special} For $(\lambda,\kappa)\in(-2,0)\times\mathbb{R}^-$, there is a finite $t_*>0$ for which $u_x(\gamma(\overline\alpha,t),t)$ diverges to minus infinity, as $t\uparrow t_*$, but remains bounded otherwise. Instead, if $(\lambda,\kappa)\in(-\infty,-2]\times\mathbb{R}^-$, $u_x(\gamma(\overline\alpha,t),t)\to-\infty$ as $t\uparrow t_*$ and blows up to positive infinity otherwise.

\item\label{it:ak>012special} For $(\lambda,\kappa)\in\mathbb{R}^-\times\mathbb{R}^-$, $\rho$ stays bounded for all $0\leq t\leq t_*$ and $\alpha\in[0,1]$. 
\end{enumerate}

\end{corollary}

\begin{proof}
Recall that for $\lambda\kappa<0$, the only instance leading to finite-time blow-up involved $\mathcal{Q}$ having a double root $\eta_*$. Based on this assumption, we derived estimate (\ref{eq:estk11}) in the proof of Theorem \ref{thm:blow1ak<0}. Comparing such derivation to the one leading to estimate (\ref{eq:ak>0eq2}) in Theorem \ref{thm:ak>0lambdanegkappapos}, we note that the two correspond simply if we replace $\lambda$ in the latter by $\frac{\lambda}{2}$, you may check that both (\ref{eq:long1}) and (\ref{eq:ak>0eq1}) coincide under the suggested substitution. In this way, for instance, the regularity result for $u_x$ on the interval $-1<\lambda<0$ in Theorem \ref{thm:ak>0lambdanegkappapos}, will apply to our $u_x$ for $-1<\frac{\lambda}{2}<0$, namely, $-2<\lambda<0$, and similarly for the remaining values. In contrast, greater care is needed when studying $\rho$. Since $\eta_*=\frac{1}{\lambda m_0}$ is the earliest root of $\mathcal{Q}$ and has multiplicity two, we have, from $\mathcal{D}(\overline\alpha)=0$, that $\rho_0(\overline\alpha)=0$, whose opposite was precisely the requirement we needed on $\rho_0$ in Theorem \ref{thm:ak>0lambdanegkappapos} to obtain blow-up. Because this is not possible in the present double root case, we obtain part (\ref{it:ak>012special}) of the theorem.
\end{proof}


\begin{corollary}
\label{coro:singular2}
Consider the initial boundary value problem (\ref{eq:hs})-(\ref{eq:pbc}) for $(\lambda,\kappa)\in\mathbb{R}^+\times\mathbb{R}^+$. Suppose the initial data is such that (\ref{eq:Q}) vanishes earliest at $\eta_*=\frac{1}{\lambda M_0}$, a zero of double multiplicity, where $M_0>0$ represents the greatest value achieved by $u_0'$ in $[0,1]$ at finitely many points $\overline\alpha$. Then 

\begin{enumerate}

\item\label{it:sing21} For $(\lambda,\kappa)\in(0,1]\times\mathbb{R}^+$, $u_x$ persists globally in time. More particularly, it vanishes as $t\to+\infty$ if $(\lambda,\kappa)\in(0,1)\times\mathbb{R}^+$, but converges to a non-trivial steady-state for $(\lambda,\kappa)\in\{1\}\times\mathbb{R}^+$. In contrast, for $(\lambda,\kappa)\in(1,+\infty)\times\mathbb{R}^+$, there exists a finite $t_*>0$ such that $u_x(\gamma(\overline\alpha,t),t)\to+\infty$ as $t\uparrow t_*$, but blows up to $-\infty$ otherwise.

\item\label{it:sing22} For $(\lambda,\kappa)\in\mathbb{R}^+\times\mathbb{R}^+$, $\rho(\gamma(\overline\alpha,t),t)\equiv0$; however, if $\alpha\neq\overline\alpha$ is such that $\rho_0(\alpha)\neq0$, then $\rho\circ\gamma\to0$ as $t\to+\infty$ when $(\lambda,\kappa)\in(0,1]\times\mathbb{R}^+$, while, for $(\lambda,\kappa)\in(1,2]\times\mathbb{R}^+$ and $t_*>0$ finite as above, $\rho\circ\gamma\to0$ as $t\uparrow t_*$, whereas $\rho\circ\gamma\to C\rho_0(\alpha)$ if $(\lambda,\kappa)\in(2,+\infty)\times\mathbb{R}^+$.   

\end{enumerate}
\end{corollary}

\begin{proof}
Suppose $(\lambda,\kappa)\in\mathbb{R}^+\times\mathbb{R}^+$ and assume that the initial data is such that $\mathcal{Q}$ vanishes earliest at $\eta_*=\frac{1}{\lambda M_0}$, namely, $\eta_*$ is a zero of double multiplicity. Here, $M_0>0$ is the maximum of $u_0'$ in $[0,1]$, which we assume occurs at finitely many points $\overline\alpha$. Then by our usual assumption that $u_0''(\overline\alpha)=0$ and $u_0'''(\overline\alpha)\neq0$, we follow an argument analogous to the one that led to estimate (\ref{eq:estk11}) and find that, for $\eta_*-\eta>0$ small,
\begin{equation}
\label{eq:lastest3}
\mathcal{\bar{P}}_0(t)\sim
\begin{cases}
C_9(1-\lambda\eta(t)M_0)^{\frac{1}{2}-\frac{1}{\lambda}},\,\,\,\,\,\,\,\,\,\,\,\,\,&\lambda\in(0,2),
\\
-C\ln(\eta_*-\eta),\,\,\,\,\,\,\,\,\,\,\,\,\,&\lambda=2,
\\
C,\,\,\,\,\,\,\,\,\,\,\,\,&\lambda\in(2,+\infty)
\end{cases}
\end{equation}
where
$$C_9=\frac{\Gamma\left(\frac{1}{\lambda}-\frac{1}{2}\right)}{\Gamma\left(\frac{1}{\lambda}\right)}\sqrt{\frac{\pi M_0}{\left|C_1\right|}}\in\mathbb{R}^+.$$
Part (\ref{eq:lastest3})iii) is obtained via the argument outlined in Appendix \ref{sec:furtherintegrals}. Furthermore, for (\ref{eq:integrals})ii), we have that
\begin{equation}
\label{eq:lastest4}
\int_0^1{\frac{\lambda u_0'(\alpha)-\eta(t)C(\alpha)}{\mathcal{Q}(\alpha,t)^{1+\frac{1}{2\lambda}}}d\alpha}\sim C_{10}(1-\lambda\eta(t)M_0)^{-\frac{1}{2}-\frac{1}{\lambda}},\,\,\,\,\,\,\,\,\,\,\,\,\,\lambda\in(0,+\infty),
\end{equation}
where $$C_{10}=\frac{\Gamma\left(\frac{1}{2}+\frac{1}{\lambda}\right)}{\Gamma\left(1+\frac{1}{\lambda}\right)}\sqrt{\frac{\pi M_0}{\left|C_1\right|}}\in\mathbb{R}^+.$$
Using the above estimates on (\ref{eq:ux}), (\ref{eq:rho}), and (\ref{eq:etaivp}) yields our result.
\end{proof}


\begin{thebibliography}{99}

\bibitem{Barnes1}
E.W. Barnes, ``\emph{A New Development of the Theory of Hypergeometric Functions}'', Proc. London Math. Soc. (2) \textbf{6} (1908), 141-177.
\bibitem{Beals1}
R. Beals, D. H. Sattinger and J. Szmigielski, Inverse scattering solutions of the Hunter-Saxton equation, Appl. Anal. 78 (3 and 4) (2001) 255-269.
\bibitem{Bressan1}
A. Bressan and A. Constantin, Global solutions of the Hunter-Saxton equation, SIAM J. Math. Anal 37(3) (2005), 996-1026.
\bibitem{Camassa1}
R. Camassa and D. Holm, An integrable shallow water equation with peaked solitons, Phys. Rev. Lett. 71 (11) (1993), 1661-1664.
\bibitem{Cao2}
C. Cao and E.S. Titi, Global well둃osedness of the three-dimensional viscous primitive equations of large scale ocean and atmosphere dynamics, Ann. of Math. 166 (2007), 245-267.
\bibitem{Cao3}
C. Cao and E.S. Titi, Global well-posedness of the three-dimensional stratified primitive equations with partial vertical mixing turbulence diffusion, Comm. Math. Phys. 310 (2012), 537-568.
\bibitem{Cao1}
C. Cao, S. Ibrahim, K. Nakanishi, and E.S. Titi, Finite-time blowup for the inviscid primitive equations of oceanic and atmospheric dynamics, arXiv preprint arXiv:1210.7337, 10 2012.
\bibitem{Chae1}
D. Chae, On the blow-up problem for the axisymmetric 3D Euler equations, Nonlinearity 21 (2008) 2053-2060.
\bibitem{Chen1}
X. Chen and H. Okamoto, Global existence of solutions to the Proudman-Johnson equation, Proc.Japan Acad. 76 (2000), 149-152.
\bibitem{Chen2}
X. Chen and H. Okamoto, Global existence of solutions to the generalized Proudman-Johnson equation, Proc.Japan Acad. 78 (2002), 136-139.
\bibitem{Childress}
S. Childress, G.R. Ierley, E.A. Spiegel and W.R. Young, Blow-up of unsteady two-dimensional Euler and Navier-Stokes equations having stagnation-point form, J. Fluid Mech. 203 (1989), 1-22.
\bibitem{Wunsch3}
C. H. Cho and M. Wunsch, Global and singular solutions to the generalized Proudman-Johnson equation, J. Diff. Eqns. 249, (2010), 392-413.
\bibitem{Wunsch2}
C.H. Cho and M. Wunsch, Global weak solutions to the generalized Proudman-Johnson equation, Commun. Pure Appl. Ana. 11 \textbf{4}, (2012) 1387-1396.
\bibitem{Aconstantin0}
A. Constantin, R. I. Ivanov, On an integrable two-component Camassa-Holm shallow water system. Physics Letters A 372 (2008) 7129-7132.
\bibitem{Aconstantin1}
A. Constantin and M. Wunsch, On the inviscid Proudman-Johnson equation, Proc. Japan Acad. Ser. A Math. Sci., 85, 7, (2009), 81-83.
\bibitem{Aconstantin2}
A. Constantin, D. Lannes, The hydrodynamical relevance of the Camassa-Holm and Degasperis-Procesi Equations, Arch. Rational Mech. Anal. 192 (2009) 165-186.
\bibitem{Const2}
P. Constantin, P. D. Lax, A. Majda, A simple one dimensional model for the three dimensional  vorticity equation, Commun. Pure Appl. Math., 38, (1985), 715.
\bibitem{Dafermos1}
C.M. Dafermos, Generalized characteristics and the Hunter-Saxton equation, J. Hyperbol. Differ. Eq., 8 \textbf{1}, (2011), 159-168.
\bibitem{Dullin1}
H. R. Dullin, G. A. Gottwald, D. D. Holm, Camassa-Holm, Korteweg-de Vries-5 and other asymptotically equivalent equations for shallow water waves, Fluid Dyn. Res., 33 (2003), 73-95.
\bibitem{Magnus1}
A. Erd\'elyi, W. Magnus, F. Oberhettinger and F.G. Tricomi, ``\emph{Higher transcendental functions, Vol. I}'', McGraw-Hill, (1981), 56-119.
\bibitem{Ermakov1}
V.P. Ermakov, Univ. Izv. Kiev 20, 1 (1880).
\bibitem{Escher0}
J. Escher, O. Lechtenfeld, Z. Yin, Well-posedness and blow up phenomena for the 2-component Camassa-Holm equations, Discrete Contin. Dyn. Syst. 19 (3) (2007) 493-513.
\bibitem{Gamelin1}
T.W. Gamelin, ``\emph{Complex Analysis}'', Undergraduate Texts in Mathematics, Springer (2000), 361-365.
\bibitem{Gasper1}
G. Gasper and M. Rahman, ``\emph{Basic Hypergeometric Series}'', Encyclopedia of Mathematics and Its Applications, 96, Second ed., Cambridge University Press, (2004) 113-119.
\bibitem{Gill1}
A. E. Gill, ``\emph{Atmosphere-Ocean Dynamics}'', Academic Press (London), 1982.
\bibitem{Green1} 
A. Green and P. Naghdi, A derivation of equations for wave propagation in water of variable depth, J. Fluid Mech. 78 (1976) 237-246.
\bibitem{Guan1} 
C. Guan and Z. Yin, Global existence and blow-up phenomena for an integrable two-component Camassa-Holm shallow water system. J. Differential Equations 248 (8) (2010) 2003-2014.
\bibitem{Guo1} 
Z. Guo, Blow up and global solutions to a new integrable model with two components. J. Math. Anal. Appl., 372 \textbf{1}, (2010) 316-327.
\bibitem{Guo2} 
Z. Guo and Y. Zhou, On Solutions to a two-component generalized Camassa-Holm equation, Stud. Appl. Math., 124 \textbf{3}, (2010) 307-322.
\bibitem{Holm1} 
D.D. Holm and M.F. Staley, Wave structure and nonlinear balances in a family of evolutionary PDEs, SIAM J. Appl. Dyn. Syst. 2 (2003), 323-380.
\bibitem{Hou1} 
T. Y. Hou and C. Li, Dynamic stability of the three-dimensional axisymmetric Navier-Stokes equations with swirl. Comm. Pure Appl. Math. LXI (2008) 661-697
\bibitem{Hunter1}
J.K. Hunter and R. Saxton, Dynamics of director fields, SIAM J. Appl. Math. 51(6) (1991), 1498-1521.
\bibitem{Johnson1}
R.S. Johnson, Camassa-Holm, Korteweg-de Vries and related models for water waves, J. Fluid. Mech. 455 (2002) 63-82.
\bibitem{Johnson2}
R.S. Johnson, The Camassa-Holm equation for water waves moving over a shear flow Fluid Dyn. Res. 33 (2003) 97.
\bibitem{Kevre1}
P. Kevrekidis and Y. Drossinos, Nonlinearity from linearity: The Ermakov-Pinney equation revisited, Mathematics and Computers in Simulation, 74 (2007), 196-202.
\bibitem{Khesin1}
B. Khesin and G. Misiolek, Euler equations on homogeneous spaces and Virasoro orbits, Adv. Math. 176 (2003) no. 1, 116-144.
\bibitem{Khesin2}
B. Khesin, J. Lenells, G. Misiolek and S.C. Preston, Geometry of diffeomorphism groups, complete integrability and optimal transport, Geom. and Funct. Anal., vol.23, no.1 (2013), 334-366; (arXiv:1105.0643)
\bibitem{Lenells1}
J. Lenells, Weak geodesic flow and global solutions of the Hunter-Saxton equation, Discrete Contin. Dyn. Syst. 18 (4) (2007), 643-656.
\bibitem{Lenells2}
J. Lenells and O. Lechtenfeld, On the $N=2$ supersymmetric Camassa-Holm and Hunter-Saxton equations, J. Math. Phys. 50 (2009) 1-17.
\bibitem{Lenells3}
J. Lenells and M. Wunsch, The Hunter-Saxton system and the geodesics on a pseudosphere, Commun. Part. Diff. Eq., 38 \textbf{5}, (2013) 860-881
\bibitem{Liu1}
J. Liu and Z. Yin, Blow-up phenomena and global existence for a periodic two-component Hunter-Saxton system, arXiv preprint 1012.5448
\bibitem{Liu2}
J. Liu and Z. Yin, Global weak solutions for a periodic two-component $\mu-$Hunter됩axton system, Monatshefte f$\ddot{\text{u}}$r Mathematik, 168, \textbf{3-4}, (2012) 503-521.
\bibitem{Majda1}
A. J. Majda, ``\emph{Introduction to PDEs and Waves for the Atmosphere and Ocean}'', Courant Lecture Notes in Mathematics 9, AMS/CIMS, 2003.
\bibitem{Misiolek1}
G. Misiolek, Classical solutions of the periodic Camassa-Holm equation, Geom. Funct. Anal., 12 \textbf{5} (2002) 1080-1104.
\bibitem{Mohajer1}
K. Mohajer, A note on traveling wave solutions to the two-component Camassa-Holm equation. J. Nonlinear Math. Phys. 16 (2009) 117-125.
\bibitem{Moon0}
B. Moon and Y. Liu, Wave breaking and global existence for the generalized periodic two-component Hunter-Saxton system, J. Differ. Equations, 253 (2012) 319-355.
\bibitem{Moon1}
B. Moon, Solitary wave solutions of the generalized two-component Hunter됩axton system,  Nonlinear Anal-Theor., 89, (2013), 242-249.
\bibitem{Mustafa1}
O. G. Mustafa, On smooth traveling waves of an integrable two-component Camassa-Holm equation, Wave Motion 46 (2009) 397-402.
\bibitem{Ohkitani1}
H. Okamoto and K. Ohkitani, On the role of the convection term in the equations of motion of incompressible fluid, J. Phys. Soc. Japan 74 (2005), 2737-2742.
\bibitem{Okamoto1}
H. Okamoto and J. Zhu, Some similarity solutions of the Navier-Stokes equations and related topics, Taiwanese J. Math. 4 (2000), 65-103.
\bibitem{Okamoto2}
H. Okamoto, Well-posedness of the generalized Proudman-Johnson equation without viscosity, J. Math. Fluid Mech. 11 (2009), 46-59.
\bibitem{Okamoto3}
H. Okamoto, T. Sakajo, M. Wunsch, On a generalization of the Constantin-Lax-Majda equation, Nonlinearity 21, (2008), 2447-2461.
\bibitem{Pavlov1}
M. V. Pavlov, The Gurevich-Zybin system, J. Phys. A: Math. Gen. 38 (2005) 3823-3840.
\bibitem{Pedlosky1}
J. Pedlosky, ``\emph{Geophysical Fluid Dynamics}'', Springer-Verlag, New York, 1987.
\bibitem{Pinney1}
E. Pinney, Proc. Am. Math. Soc. 1 (1950) 681.
\bibitem{Proudman1}
I. Proudman and K. Johnson, Boundary-layer growth near a rear stagnation point, J. Fluid Mech. 12 (1962), 161-168.
\bibitem{Sarria1}
A. Sarria and R. Saxton, Blow-up of solutions to the generalized inviscid Proudman-Johnson equation, J. Math Fluid Mech., 15, 3, 493-523.
\bibitem{Sarria2}
A. Sarria and R. Saxton, The role of initial curvature in solutions to the generalized inviscid Proudman-Johnson equation, Q. Appl. Math (in press).
\bibitem{Sarria3}
A. Sarria, Regularity of stagnation point-form solutions of the two-dimensional Euler equations, Differential and Integral Equations (in press).
\bibitem{Sarria5}
A. Sarria and J. Wu, Blowup in stagnation-point form solutions of the inviscid 2d Boussinesq equations, submitted, arXiv preprint arXiv:1408.6625.
\bibitem{Saxton1}
R. Saxton and F. Tiglay, Global existence of some infinite energy solutions for a perfect incompressible fluid, SIAM J. Math. Anal. 4 (2008), 1499-1515.
\bibitem{Tiglay1}
F. Tiglay, The periodic Cauchy problem of the modified Hunter-Saxton equation, J. Evol. Eq. 5 (4) (2005), 509-527.
\bibitem{Weyl1}
H. Weyl, On the differential equations of the simplest boundary-layer problems, Ann. Math. 43 (1942), 381-407.
\bibitem{Wu1}
H. Wu and M. Wunsch, Global Existence for the Generalized Two-Component Hunter-Saxton System, 14 \textbf{3}, (2012) 455-469.
\bibitem{Wunsch1}
M. Wunsch, The generalized Proudman-Johnson equation revisited, J. Math. Fluid Mech. 13 (1) (2009), 147-154.
\bibitem{Wunsch4}
M. Wunsch, On the Hunter-Saxton system, Discrete Contin. Dyn. Syst. Ser. B, 12 (2009), 647-656.
\bibitem{Wunsch5}
M. Wunsch, The generalized Hunter-Saxton system. SIAM J. Math. Anal. 42 (3) (2010) 1286-1304.
\bibitem{Wunsch6}
M. Wunsch, Weak geodesic flow on a semidirect product and global solutions to the periodic Hunter-Saxton system, Nonlinear Anal-Theor., 74 \textbf{15}, (2011) 4951-4960. 
\bibitem{Yin1}
Z. Yin, On the Structure of Solutions to the Periodic Hunter-Saxton Equation, SIAM J. Math. Anal., 36 \textbf{1}, (2004) 272-283.
\bibitem{Zhang1}
P. Zhang and Y. Liu, Stability of solitary waves and wave-breaking phenomena for the two-component Camassa-Holm system, Int. Math. Res. Notices, Vol. 2010, \textbf{11}, 1981-2021.



\end{thebibliography}
\end{document}